\documentclass[12pt,twoside]{amsart}
\usepackage{amssymb}
\usepackage{amscd}
\usepackage[abbrev,alphabetic]{amsrefs}
\usepackage{hyperref}
\usepackage{comment}
\usepackage{array,multirow,tabularx,longtable}
\usepackage{tikz}
\usetikzlibrary{cd}
\usepackage{here}
\usepackage{multirow}
\usepackage[margin=1.25in]{geometry}
\usepackage{mathtools}

\usepackage{framed}
\usepackage{fancybox}
\usepackage{ascmac}

\title[Prime Fano threefolds in positive characteristic]
{Prime Fano threefolds in positive characteristic} 

\author{Akihiro Kanemitsu}
\address{Department of Mathematical Sciences, Graduate School of Science, Tokyo Metropolitan University, 1-1 Minami-Osawa, Hachioji-shi, Tokyo 192-0397, Japan}
\email{kanemitsu@tmu.ac.jp}
\thanks{The first author is supported by JSPS KAKENHI Grant Number JP23K12948. The second author is supported by JSPS KAKENHI Grant Numbers JP22H01112 and JP23K03028.}

\author{Hiromu Tanaka} 
\address{Department of Mathematics, 
Graduate School of Science, 
Kyoto University, 
Kyoto 606-8502, JAPAN} 
\email{tanaka.hiromu.7z@kyoto-u.ac.jp}

\DeclareMathOperator{\Orb}{Orb}
\DeclareMathOperator{\Stab}{Stab}

\DeclareMathOperator{\Sym}{Sym}
\DeclareMathOperator{\LG}{LG}
\DeclareMathOperator{\Sp}{Sp}

\DeclareMathOperator{\id}{id}

\newcommand{\CH}[0]{{\operatorname{CH}}}
\newcommand{\ol}{\overline}

\newcommand{\wt}{\widetilde}

\newcommand{\Bl}[0]{{\operatorname{Bl}}}

\newcommand{\et}[0]{{\operatorname{\acute{e}t}}}
\newcommand{\rank}[0]{\operatorname{rank}}

\newcommand{\codim}[0]{{\operatorname{codim}}}

\newcommand{\red}[0]{{\operatorname{red}}}
\newcommand{\Coker}[0]{\operatorname{Coker}}
\newcommand{\Ker}[0]{\operatorname{Ker}}
\newcommand{\Image}[0]{\operatorname{Image}}
\renewcommand{\Im}[0]{\operatorname{Im}}

\newcommand{\Hom}[0]{{\operatorname{Hom}}}

\newcommand{\Pic}[0]{\operatorname{Pic}}

\newcommand{\Ex}[0]{{\operatorname{Ex}}}

\newcommand{\Gr}[0]{{\operatorname{Gr}}}
\newcommand{\OGr}[0]{{\operatorname{OGr}}}
\newcommand{\Ext}[0]{{\operatorname{Ext}}}

\newcommand{\GL}[0]{{\operatorname{GL}}}

\newtheorem{thm}{Theorem}[section]
\newtheorem{lem}[thm]{Lemma}
\newtheorem*{lem*}{Lemma}
\newtheorem{lemma}[thm]{Lemma}

\newtheorem{prop}[thm]{Proposition}
\newtheorem{proposition}[thm]{Proposition}

\newtheorem*{claim*}{Claim}

\theoremstyle{definition}

\newtheorem{dfn}[thm]{Definition}
\newtheorem{definition}[thm]{Definition}

\newtheorem{rem}[thm]{Remark}
\newtheorem{remark}[thm]{Remark}
      
\newtheorem{nota}[thm]{Notation}         

\makeatletter
  
  \@addtoreset{equation}{thm}
  \makeatother

\newcommand{\bP}{\mathbb{P}}

\newcommand{\bZ}{\mathbb{Z}}

\newcommand{\cE}{\mathcal{E}}
\newcommand{\cF}{\mathcal{F}}
\newcommand{\cG}{\mathcal{G}}

\newcommand{\cI}{\mathcal{I}}

\newcommand{\cK}{\mathcal{K}}

\newcommand{\cN}{\mathcal{N}}
\newcommand{\cO}{\mathcal{O}}

\newcommand{\cQ}{\mathcal{Q}}

\newcommand{\cU}{\mathcal{U}}

\newcommand{\MO}{\mathcal{O}}

\newcommand{\Q}{\mathbb{Q}}
\newcommand{\Z}{\mathbb{Z}}

\renewcommand{\P}{\mathbb{P}}

\newcommand{\la}{\langle}
\newcommand{\ra}{\rangle}

\DeclareMathOperator{\Fl}{Fl}

\usepackage{listings}
\begin{document}

\maketitle

\begin{abstract}
We prove that every prime Fano threefold of genus $\geq 7$ in positive characteristic  is a linear section of a Mukai variety. 
\end{abstract}

\tableofcontents

\section{Introduction} 

The classification of Fano varieties has been a central subject in algebraic geometry.
A one-dimensional Fano variety is nothing but the projective line $\P^1$, while a two-dimensional Fano variety is called a del Pezzo surface, 
which is isomorphic either to $\P^1 \times \P^1$ or to the blowup of $\P^2$ at no more than eight points in general position.

The classification becomes substantially richer in dimension three.
Over an algebraically closed field of characteristic zero, Fano threefolds were classified by Mori and Mukai
\cite{MM81}, \cite{MM83}, \cite{MM03},
following fundamental work of Iskovskih and Shokurov
\cite{Isk77}, \cite{Isk78}, \cite{Sho79a}, \cite{Sho79b};
see also \cite{IP99}, \cite{Tak89}.

After the completion of this classification, Mukai discovered remarkable explicit geometric models of prime Fano threefolds. 
Roughly speaking, these models realise prime Fano threefolds as linear sections of rational homogeneous spaces or closely related projective varieties. 
Mukai announced these descriptions in \cite{Muk89}, and complete proofs in characteristic zero were later given by Bayer--Kuznetsov--Macri
\cite{BKM26a}, \cite{BKM26b}.

\begin{dfn}
Let $k$ be an algebraically closed field.
A {\em prime Fano threefold} $X$ over $k$ is a smooth projective threefold $X$ over $k$ such that $-K_X$ is ample and $\Pic X$ is generated by $\omega_X$.
The {\em genus} of $X$ is the positive integer $g$ defined by
\[
(-K_X)^3=2g-2.
\]
\end{dfn}

Recently, the Mori--Mukai classification has been extended 
to the case of positive characteristic by Asai and the second author 
\cite{FanoI}, \cite{FanoII}, \cite{FanoIII}, \cite{FanoIV}.
Against this background, a fundamental problem in the study of prime Fano threefolds in positive characteristic is to determine whether they admit the explicit projective models described by Mukai. 
Together with the previously known cases $g=6$ and $g=8$, 
our result for 
$g\in\{7,9,10,12\}$ completes Mukai's description of prime Fano threefolds in positive characteristic. 
More precisely, the main theorem  of this article is the following.

\begin{thm}\label{intro main}
Let $k$ be an algebraically closed field of characteristic $p>0$ 
and let $X$ be a prime Fano threefold over $k$ of genus $g \geq 6$. 
Then the following hold: 
\begin{enumerate}
\item 
If $g = 6$, then $X$ is a complete intersection of a quadric hypersurface and three hyperplanes 
in the cone ${\rm CGr}(5, 2) \subset \P^{10}$ over $\Gr(5, 2) \subset \P^9$ 
\cite[Theorem 5.4]{KTLift1}.
\item 
If $g = 7$,  
then $X$ is isomorphic to 
$\Sigma_7 \cap L$ for some  $8$-dimensional linear subvariety $L$ of $\P^{15}$ 
$($Theorem \ref{t g=7 Lin Sec Thm}$)$. 
\item 
If $g = 8$, then 
$X$ is isomorphic to $\Sigma_8 \cap L$ 
for some $9$-dimensional linear subvarliety $L$ of $\P^{14}$ \cite[Theorem 1.2]{KT26}. 
\item 
If $g = 9$,  then 
$X$ is isomorphic to $\Sigma_9 \cap L$ 
for some $10$-dimensional linear subvariety $L$ of $\P^{13}$  
$($Theorem \ref{t g=9 Lin Sec Thm}$)$. 
\item 
If $g = 10$,  then 
$X$ is isomorphic to $\Sigma_{10} \cap  L$ 
for some $11$-dimensional linear subvariety $L$ of $\P^{13}$ 
$($Theorem \ref{t g=10 Lin Sec Thm}$)$.  
\item 
If $g=12$, then there exist a $7$-dimensional vector space $W_7$, a $3$-dimensional vector space $N_3$, and a non-degenerate net of alternating forms
$\nu:\wedge^2W_7\to N_3$ such that $X\simeq \Sigma_{12}(\nu)$ 
$($Theorem \ref{t g=12 lin sec thm}$)$.
\end{enumerate}
Here $\Sigma_g$ and $\Sigma_{12}(\nu)$ are  defined as in Definition \ref{intro d Sigmag} and Definition \ref{intro d Sigma12}. 
\end{thm}

\begin{dfn}\label{intro d Sigmag}
For $g \in \{7, 8, 9, 10\}$, 
the {\em Mukai variety} $\Sigma_g$ of genus $g$ is defined as follows: 
\begin{itemize}
\item 
$\Sigma_7$ is defined as a connected component of 
the orthogonal Grassmanniaan variety $\OGr(10, 5) \subset \P^{15}$. 
\item 
$\Sigma_8 := \Gr(6, 2) \subset \P^{14}$. 
\item 
$\Sigma_9 := \LG(6, 3) \subset \P^{13}$, which is the Lagrangian Grassmannian variety. 
\item 
$\Sigma_{10} := G_2/P \subset \P^{13}$, 
where $G_2$ is a connected semisimple algebraic group of type $G_2$ and 
$P$ is its reduced parabolic subgroup corresponding to the long root. 
\end{itemize}
Here each closed embedding into a projective space is 
given by the complete linear system of the ample generator of $\Pic (\Sigma_g) (\simeq \Z)$. 

\end{dfn}

\begin{dfn}\label{intro d Sigma12}
Let $W_7$ be a $7$-dimensional vector space. 
\begin{enumerate}
\item A {\em non-degenerate net of alternating forms} $\nu \colon \wedge^2 W_7 \to N_3$ is 
a surjective linear map to a three-dimensional vector space $N_3$ 
such that the composite alternating form $f\circ \nu \colon \wedge^2 W_7 \to k$ is of rank $6$ for every linear surjective map $f \colon N_3 \to k$. 
\item 
For a non-degenerate net of alternating forms $\nu \colon \wedge^2 W_7 \to N_3$, 
{\em the Mukai variety $\Sigma_{12}(\nu)$ of genus $12$}  associated with $\nu$ 
is defined as the zero locus $\Sigma_{12}(\nu) \subset \Gr(W_7,4)$ of the composite $\cO_{\Gr(W_7,4)}$-module morphism
\[
\wedge^2\cU
\longrightarrow
\wedge^2W_7\otimes\cO_{\Gr(W_7,4)}
\xrightarrow{\nu \otimes \cO_{\Gr(W_7,4)}}
N_3\otimes\cO_{\Gr(W_7,4)},
\]
where $\cU$ denotes the universal subbundle on $\Gr(W_7,4)$.
\end{enumerate}

\end{dfn}

\subsection{Overview of the proof for $g \in \{ 9, 10, 12\}$}
Let $X$ be a prime Fano threefold over $k$ of genus $g \in \{9, 10, 12\}$. 
For each $g\in\{9,10,12\}$, we fix a factorisation $g=rs$ as follows:
\[
(g,r,s)=(9,3,3),\qquad (g,r,s)=(10,2,5),\qquad (g,r,s)=(12,3,4).
\]
By \cite{Tan-bdl}, 
there exists a globally generated vector bundle $\cQ_X$ 
of rank $r$ 
satisfying the following properties: 
\[
c_1(\cQ_X) = -K_X,\quad  
h^0(X, \cQ_X)=r+s, \quad 
H^{>0}(X,\cQ) = 0, \quad 
H^\bullet(X, \cQ_X^\vee) = 0. 
\]
For $V := H^0(X, \cQ_X) (\simeq k^{r+s})$, 
the induced surjection $V  \otimes \MO_X \twoheadrightarrow  
\cQ_X$ corresponds to the morphism 
\[
\varphi : X \to \Gr(V, r)
\]
satisfying $\varphi^*\cQ \simeq \cQ_X$ for the universal quotient bundle $\cQ$ on $\Gr(V, r)$. 
Fix a general member $S \in |-K_X|$. 
Set $\cQ_S := \cQ_X|_S$ and $H :=-K_X|_S$. 
Roughly speaking, the proof consists of 
the following four steps:
\begin{enumerate}
\item[(i)] $\varphi|_S : S \to \Gr(V, r)$ is a closed immersion satisfying $\dim \la \varphi(S)\ra = g$ (i.e., 
$\varphi|_S$ is induced by the complete linear system $|H|$).
\item[(ii)] $\varphi : X \to \Gr(V, r)$ is a closed immersion, and hence we identify $X$ with its image $\varphi(X)$. 
\item[(iii)] There exists a Mukai variety $\Sigma_g$ satisfying $X \subset \Sigma_g \subset \Gr(V, r)$. 
\item[(iv)] $X$ is a linear section of $\Sigma_g$.  
\end{enumerate}
Here $\la Y \ra$ denotes the smallest linear subvarity containing $Y$ inside the ambient projective space $\P(\wedge^r V)$. 

\medskip

(i) 
We have the following commutative diagram
\[
\begin{tikzcd}
\wedge^r H^0(X, \cQ_X) \arrow[r,"\lambda_X"] \arrow[d,"\simeq"] & H^0(X,\wedge^r \cQ_X) = H^0(X, -K_X) \arrow[d] \\
 \wedge^r H^0(S,\cQ_S) \arrow[r,"\lambda_S"] & H^0(S,\wedge^r \cQ_S) = H^0(S, H). 
\end{tikzcd}
\]
By using the fact that $\varphi(X)$ is not contained in any Schubert divisor on $\Gr(V, r)$ (Proposition~\ref{p Sch avoid X}), we see that $\bP((\Ker \lambda_X)^\vee) \cap \Gr(r,V) = \emptyset$ in $\bP(\wedge^rV^\vee)$. 
Thus we get
\[
\dim \Ker \lambda_X = \dim \bP((\Ker \lambda_X)^\vee) +1 < \dim \bP(\wedge^rV^\vee) - \dim \Gr(r,V) +1 = 
\dim \wedge^rV 
-rs.
\]
Then
\[
\dim \Im \lambda_ X = \dim (\wedge^r H^0(X, \cQ_X)) - \dim (\Ker \lambda_X)
= \dim \wedge^rV - \dim (\Ker \lambda_X) > rs =g. 
\]
Hence $\dim \Im \lambda_X =g+2$ or $\dim \Im \lambda_X =g+1$. 
As $S \in |-K_X|$ is chosen to be general, $\lambda_S$ is surjective. 

\medskip

(ii) 
As $\varphi|_S : S \to \Gr(V, r)$ is a closed immersion by (i), 
we identify $S$ with its image $\varphi(S)$. 
By the argument in (i), 
either $\dim \la \varphi(X) \ra  =g$  or $\dim \la \varphi(X) \ra = g+1$. 
There is nothing to show for the latter case $\dim \la \varphi(X)  \ra = g+1$, 
as the equality $\dim \la \varphi(X)  \ra = g+1$ implies that 
$\varphi : X \to \Gr(V, r)$ is the closed immersion induced by the complete linear system $|-K_X|$. 
Hence we may assume that $\la \varphi(X)  \ra  = \la S \ra =: \P^g$. 
In this case, 
we have $Y := \varphi(X) \subset \Gr(V, r) \cap \P^g$, 
and hence 
the image $Y=\varphi(X)$ is a smooth projective threefold by $(\star)_0$ below. 
\begin{enumerate}
\item[$(\star)_{\nu}$] 
If $\varphi(X) \subset Y \subset \Gr(V, r) \cap \P^{g+\nu}$ for some 
$(g+\nu)$-dimensional linear subvariety $\P^{g+\nu}$, then $\dim T_{Y, y} \leq 3+\nu$ for every closed point $y \in Y$, where 
$T_{Y, y}$ denotes the tangent space of $Y$ at $y$.   
\end{enumerate}
It is easy to show that $\deg \psi = 1$ or $\deg \psi =2$ for the induced morphism $\psi : X \to Y$. 
Then it is enough to exclude the case when $\deg \psi =2$, 
which is carried out by using the fact that $Y \subset \P^g$ is a 
smooth threefold of almost minimal degree satisfying $\rho(Y) \leq \rho(X) = 1$. 
For more details on (i) and (ii), see Subsection \ref{ss emb Gr}.

\medskip

(iii)  
As the arguments are similar in all the cases
$g\in\{9,10,12\}$, we only treat the case $g=9$.
In this case, $(r, s)=(3, 3)$ and we set $V_6 := V (\simeq k^6)$. 
We have the following restriction maps: 
\[
\begin{tikzcd}[column sep=large]
\wedge^2V_6 
=
H^0\left(\Gr(V_6, 3),\wedge^2\cQ\right)
\arrow[r, "\rho^{\Gr(V_6, 3)}_X"]
\arrow[rr, bend right=20, "\rho^{\Gr(V_6, 3)}_S"]
&
H^0\left(X,\wedge^2\cQ_X\right)
\arrow[r, "\rho^{X}_S", hook]
&
H^0\left(S,\wedge^2\cQ_S\right), 
\end{tikzcd}
\]
where $\rho^{X}_S$ is injective, since
$H^0(X, \wedge^2 \cQ_X (-S)) \simeq 
H^0(X, \wedge^2 \cQ_X (K_X)) \simeq 
H^0(X,\cQ_X^\vee) =0$. 
In particular, we get 
\[
\Ker(\rho^{\Gr(V_6, 3)}_X) = \Ker(\rho^{\Gr(V_6, 3)}_S).
\]
Then it  suffices to show $\Ker(\rho^{\Gr(V_6, 3)}_S) \neq 0$, 
which holds by $\dim \wedge^2 V_6 = \binom{6}{2} =15$ and  
the Riemann-Roch theorem 
$h^0\left(S,\wedge^2\cQ_S\right) = 
h^0(S, \cQ_S^\vee(H)) \overset{(\star)}{=}
\chi(S,\cQ_S^\vee(H)) 
= 14$. 
Here  $(\star)$ follows from $H^{>0}(S, \cQ_S^\vee(H))=0$, 
which is established by using the generic member of $|-K_X|$  
(see Subsection \ref{ss 4.4 vanishing} for the details of this vanishing result).

\medskip

(iv)  
In this step, we treat the cases $g\in\{9,10,12\}$ separately.

\medskip

\underline{Case $g=12$}: 
By (iii), we have $X \subset \Sigma_{12}$, 
where $\deg X = 22 = \Sigma_{12}$ and $\Sigma_{12}$ is a 
pure $3$-dimensional closed subscheme on $\Gr(V, 3)$. 
Then we get an equality $[X] = [\Sigma_{12}]$ of algebraic $3$-cycles on $\Gr(V, 3)$, which implies the scheme-theoretic equality $X = \Sigma_{12}$, because $\Sigma_{12}$ is Cohen-Macaulay. 

\medskip

\underline{Case $g=10$}: 
By (iii), we have $X \subset \Sigma_{10} \subset \P^{13}$. 
As $\dim \Sigma_{10} = \dim G_2/P =5$, 
$X \subset \Sigma_{10}$ is of codimension two. 
Pick a hyperplane $H$ on $\P^{13}$ containing $X$. 
Then $\Sigma_{10} \cap H$ is a $4$-dimensional integral scheme, as it is an ample generator of $\Pic \Sigma_{10} (\simeq \Z)$. 
Taking a general hyperplane $H'$ containing $X$, 
we can check that $\Sigma_{10} \cap H \cap H' \subsetneq \Sigma_{10} \cap H$. 
Then 
$\Sigma_{10} \cap H \cap H'$ is a pure $3$-dimensional lci closed subscheme on $\Sigma_{10}$ 
satisfying $\deg X = 18 = \deg \Sigma_{10} = \deg (\Sigma_{10} \cap H \cap H')$. 
We then get the scheme-theoretic equality $X = \Sigma_{10} \cap H \cap H'$ 
by  the same argument as in the case $g=12$.

\medskip

\underline{Case $g=9$}: 
By (iii), we have $X \subset \Sigma_{9} \subset \P^{13}$. 
Recall that $S$ is a general member of $|-K_X|$. 
If $\la X \ra = \la S \ra =:\P^9$, then 
$(\star)_0$ is applicable for $\Sigma_9 \cap \la X \ra$, and hence 
$\dim \Sigma_9 \cap \la X \ra =3$. 
In this case, we may apply the same argument as in the case $g=12$ by taking 
three general hyperplanes $H_1, H_2, H_3$ containing $X$.

In what follows, we treat the case  $\la X \ra \neq \la S \ra = \P^9$. 
As $9 = \dim \la S \ra < \dim \la X \ra \leq 10$, 
we have $\dim \la X \ra =10$, 
and the closed immersion $\varphi :X \hookrightarrow \Sigma_9$ is given by the complete linear system $|-K_X|$. 
Set $\P^{10} := \la X\ra$. 
Applying $(\star)_1$ for $\Sigma_9 \cap \P^{10}$, 
we get $\dim (\Sigma_9 \cap \P^{10}) \leq 4$. 
In particular, we can find hyperplanes $H_1$ and $H_2$ 
on $\P^{13}$ containing $X$ 
such that 
$\Sigma_9 \cap H_1 \cap H_2$ is a pure $4$-dimensional lci closed subscheme on $\Sigma_9$.
If $\Sigma_9 \cap H_1 \cap H_2$ is an integral scheme, then it is easy to find another hyperplane $H_3$ containing $X$ such that $\Sigma_9 \cap H_1 \cap H_2 \cap H_3$ is lci and of pure dimension $3$. 
In this case, we get the scheme-theoretic equality $X = \Sigma_9 \cap H_1 \cap H_2 \cap H_3$ as before. 
Hence we may assume that $\Sigma_9 \cap H_1 \cap H_2$ is not an integral scheme. By Schubert calculus on $\Sigma_9$, we get an equality 
\[
[\Sigma_9 \cap H_1 \cap H_2] = [Y] +[Y']
\]
of algebraic $4$-cycles for some $4$-dimensional subvarieties $Y$ and $Y'$ on $\Sigma_9$ 
satisfying $X \subset Y$ and $\deg Y = \deg Y' =8$ (possibly $Y = Y'$). 

If $\la X \ra \subsetneq \la Y \ra$, then we obtain $X \subset Y \cap H$ and $\dim(Y \cap H) =3$ for a general hyperplane $H$ containing $X$, 
which leads to the following contradiction: 
$16 = \deg X \leq \deg (Y \cap H) = \deg Y =8$. 
Thus the remaining case is when 
$\P^{10} = \langle X \rangle = \langle Y \rangle$. 
In this case, 
$Y$ is a smooth fourfold by $(\star)_1$. 
This possibility is ruled by the fact that 
an anti-canonically embedded prime Fano threefold $X \subset \P^{10}$ of genus $9$ can not be contained 
in a smooth fourfold $Y \subset \P^{10}$ of almost minimal degree (Lemma~\ref{lemma:not_in_almost_minimal_sm}). 

\begin{rem}
We now point out some differences between the above strategy and that of
\cite{BKM26b}.
One difference is that \cite{BKM26b} establishes the 
analogue of Theorem \ref{intro main} for prime K3 surfaces, 
whereas we do not obtain such a result because some of the required vanishing results are unavailable. 
As for the argument outlined above, (iii) is quite similar to the corresponding arguments in \cite{BKM26b}, while 
(i), (ii), and (iv) are new ingredients in our argument. 
\end{rem}

\subsection{Overview of the proof for $g =7$}

Let $X \subset \P^8$ be a prime Fano threefold of genus $7$. 
Fix a general linear section $C:= X\cap \bP^6$ 
and set $\cQ_C := \cN^{\vee}_{C/\bP^6}(2).$ 
By abuse of notation, we use the following identifications: 
\[
V_{10} := H^0(\P^8, \cI_{X/\bP^8}(2)) 
= H^0(\P^6, \cI_{C/\bP^6}(2)).  
\]
The natural surjection $V_{10} \otimes \cO_X \twoheadrightarrow \cN^\vee_{X/ \bP^8}(2) = \cQ_{X}$ 
induces a  morphism  
\[
\varphi : X \to \Gr(V_{10}, 5). 
\]
We have the following commutative diagram: 
\[
\begin{tikzcd}
\Sym^2H^0(\bP^8,\cI_{X/\bP^8}(2))
\arrow[r, "\mu_X"]
\arrow[d, equal]
&
H^0(X,\Sym^2\cQ_X)
\arrow[d]
\\
\Sym^2H^0(\bP^6,\cI_{C/\bP^6}(2))
\arrow[r, "\mu_C"]
&
H^0(C,\Sym^2\cQ_C).
\end{tikzcd}
\]
A key step is to prove that $\varphi(X)$ is contained in some orthogonal Grassmannian $\OGr(V_{10}, 5)$, which is equivalent to
the existence of 
a non-degenerate quadratic form $q \in \Sym^2 V_{10} = \Sym^2H^0(\bP^8,\cI_{X/\bP^8}(2))$ satisfying $\mu_X(q)=0$. 
As 
\[
\Ker(\mu_X) \subset \Ker(\mu_C), 
\]
it is enough to check the following: 
\begin{enumerate}
\item[(a)] $\Ker(\mu_C) = k q$ for some non-degenerate quadratic form 
$q \in \Sym^2 V_{10} = \Sym^2H^0(\bP^6,\cI_{C/\bP^6}(2))$. 
\item[(b)] $\Ker(\mu_X) \neq 0$. 
\end{enumerate}
Here (a) follows from  Mukai's result for curves of genus $7$ \cite{Muk95}, 
which is applicable by the Brill-Noether generality established in \cite{Tan-bdl}. 
In order to show (b), 
we take the blowup $\beta :\Bl_X \P^8 \to \P^8$ of $\P^8$ along $X$ and the 
induced morphism $ \alpha: \Bl_X \P^8 \to \P(H^0(\Bl_X \P^8, \beta^*\MO_{\P^8}(2) \otimes \MO_{\Bl_X \P^8}(-E)) =: \P^9$ for $E := \Ex(\beta)$. 
\[
\begin{tikzcd}
&
\Bl_X \P^8
\arrow[dl, "\beta"']
\arrow[dr, "\alpha"]
&
\\
\P^8
&&
\P^9
\end{tikzcd}
\]
Then we can check that (b) is equivalent to the condition that $\alpha(E)$ is contained in a quadric hypersurface, which follows from the inclusion $\alpha(E) \subset \alpha(\Bl_X \P^8)$ and 
the fact that $\alpha(\Bl_X \P^8)$ is a quadric hypersurface. 
For more details, see Section \ref{s g=7}.

\subsection{Obstacles in positive characteristic}

We now discuss the two main obstacles to carrying out the above strategy in positive characteristic.

\medskip 

\underline{Vanishing theorems}: 
To adapt the argument of \cite{BKM26b} to positive characteristic, 
one of the main obstacles is the lack of several vanishing results. 
We have established some of the required vanishings, but not all of them. 
For instance, we avoid using \cite[Corollary A.3]{BKM26b}, since its proof relies on the Borel--Bott--Weil theorem, which fails in positive characteristic. 

\medskip

\underline{Generic members}: 
In characteristic zero, 
the argument in \cite{BKM26b} is based on the fact that 
$\Pic S = \Z (-K_X|_S)$ for a very general member 
$S \in |-K_X|$. 
As far as the authors know, the same result is not available in positive characteristic. 
As a replacement, we 
use the generic member $S$ of $|-K_X|$, which satisfies $\Pic S = \Z (-K_X|_S)$. 
On the other hand, the base field $\kappa$ of $S$ is no longer an algebraically closed field. 
Although the property 
$\Pic S = \Z (-K_X|_S)$ can break up by taking the base change to the algebraic closure $\ol{\kappa}$, 
we have managed to avoid its subtlety 
(cf.\ Subsection \ref{ss 4.4 vanishing}).


\medskip
\noindent {\bf Acknowledgements.}
The authors express their gratitude to Professor Shigeru Mukai for invaluable discussions.
ChatGPT (OpenAI) was used during the preparation of this article for assistance with mathematical exploration, computations, and language editing.

\section{Preliminaries}

\subsection{Notation}\label{ss notation}

Here, we summarise notation used in this paper. 

\begin{enumerate}
\item We will freely use the notation and terminology in \cite{Har77}. 
\item 
Unless otherwise specified, we work over an algebraically closed field $k$ 
of characteristic $p>0$.
\item 
We do not distinguish between invertible sheaves and 
Cartier divisors.
For example, $\Pic X= \Z K_X$ 
means that $\Pic X$ is generated by $\omega_X$ as an abelian group. 
\item A {\em variety} $X$ over a field $\kappa$ is a separated integral scheme which is of finite type over $\kappa$. 
A {\em curve} (resp. a {\em surface}, resp. a {\em threefold}) over $\kappa$ is a variety over $\kappa$ of dimension one (resp. {\em two}, resp. {\em three}).
\item
Let $V$ be a vector spece over $k$, and let $r$ be a positive integer.
$\Gr(V,r)$ denotes the Grassmannian variety of $r$-dimensional \emph{quotients} of $V$,
and $\Gr(r,V)$ is the Grassmannian variety of $r$-dimensional \emph{subspaces} of $V$.
\item 
Given a vector space  $V$  and a positive integer $m$, 
an element $\omega\in\wedge^m V$ is called \emph{decomposable} 
if there exist $v_1,\ldots,v_m\in V$ such that 
$\omega=v_1\wedge\cdots\wedge v_m$. 
\item 
Let $V$ be a finite-dimensional vector space over $k$.
An \emph{alternating map} 
\[
e\colon V^\vee\to V
\]
is a linear map such that $f(e(f))=0$ for every $f\in V^\vee$.
Equivalently, the induced bilinear form 
\[
V^\vee\times V^\vee\to k,
\qquad
(f,g)\mapsto g(e(f)),
\]
is alternating. The rank of an alternating map is even.
\item 
We say that $S$ is a {\em smooth K3 surface} (over a field $\kappa$) if 
$S$ is a smooth projective geometrically connected surface over $\kappa$ such that $K_S \sim 0$ and $H^1(S, \MO_S)=0$. 
We say that 
$(S, H)$  is a  {\em polarised smooth K3 surface} 
if $S$ is a smooth K3 surface and $H$ is an ample Cartier divisor. 
Given a vetor bundle $\cF$ on a smooth K3 surface $S$, 
the triple $(r(\cF), c_1(\cF), s(\cF)) \in \Z \oplus \Pic S \oplus \Z$ is called the {\em Mukai vector} of $\cF$, 
where $r(\cF) := \rank \cF$ and $s(\cF) :=  \rank \cF + \frac{1}{2} c_1(\cF)^2-c_2(\cF)$. 
\end{enumerate}

\subsection{Fano threefolds and Mukai bundles}

A \emph{Fano threefold} $X$ over $k$ is, by definition, a smooth projective threefold $X$ over $k$ with ample anticanonical divisor $-K_X$.
A variety $X$ is called \emph{prime Fano threefold} if $X$ is a Fano threefold and $\Pic(X) = \bZ K_X$.
The \emph{genus} of a prime Fano threefold $X$ is the integer $g$ satisfying $(-K_X)^3=2g-2$.
By \cite[Theorem 1.1]{FanoI}, \cite[Theorem 1.2]{FanoII}, we have the following:
\begin{enumerate}
\item $2 \leq g \leq 12$ and $g \neq 11$. 
\item The anti-canonical linear system $|-K_X|$ is very ample if $g \geq 4$.
\end{enumerate}

\begin{dfn}
Let 
$X$ be a prime Fano threefold of genus $g$. 
Take integers $r$ and $s$ satisfying 
$g =rs$, $r \geq 2$, and $s \geq 2$. 
We say that $\cQ_X$ is a  {\em dual Mukai bundle} if 
\begin{enumerate}
\item 
$\rank(\cQ_X) = r$, $c_1(\cQ_X) = -K_X$, 
\item 
$H^i(X,\cQ^\vee_X) = 0$ for every $i \in \Z$, 
\item 
$\cQ_X$ is globally generated, $\dim H^0(X, \cQ_X)=r+s$, and $H^i(X, \cQ_X) = 0$ for every $i>0$. 
\end{enumerate}
\end{dfn}

\begin{thm}
Let $X$ be a prime Fano threefold of genus $g \geq 8$. 
Take integers $r$ and $s$ satisfying 
$g =rs$, $r \geq 2$, and $s \geq 2$. 
Then there exists a dual Mukai bundle $\cQ_X$ on $X$ of type $(r, s)$.
{Moreover such a vector bundle is $\mu_{-K_X}$-stable, and unique up to isomorphisms.}
\end{thm}

\begin{proof}
See \cite[Theorem 1.2]{Tan-bdl}.
\end{proof}

\subsection{Varieties of almost minimal degree}

Recall that a non-degenerate variety $Y \subset \bP^n$ is called a 
{\em variety of almost minimal degree} if
\[
\deg Y = \codim Y +2.
\]
The following result excludes some possibilities of morphisms $X \to Y$ from a prime Fano threefold $X$ to a variety $Y$ of almost minimal degree.

\begin{lemma}\label{lemma:not_in_almost_minimal_sm}
Let $X \subset \bP^{g+1}$ be a smooth prime Fano threefold of genus $g \geq 7$ such that $H^0(\P^{g+1}, \MO_{\P^{g+1}}(1)) \xrightarrow{\simeq} H^0(X,  \MO_{\P^{g+1}}(1)|_X)$. 
Then the following hold: 
\begin{enumerate}
\item $X$ is not contained in a smooth subvariety $Y\subset \bP^{g+1}$ 
satisfying 
$\dim Y =4$ and $\deg Y =g-1$.
\item For any closed point $P \in \bP^{g+1}\setminus X$ and the projection $X \to \P^g$ from $P$, 
 its image $Y \subset \P^g$ is not a smooth variety satisfying $\dim Y =3$ and 
 $\deg Y =g-1$. 
 \end{enumerate}
\end{lemma}

\begin{proof}
Let us show (1). 
Suppose that such a smooth variety $Y$ exists.
We have $\deg Y = g-1 = ((g+1) -4) +2 =\codim Y +2$, and hence $Y$ is a variety of almost minimal degree in the sense of \cite[Convention 2.4]{BS07}. 
By \cite[Theorem 1.2]{BS07}, there are two possibilities:
\begin{enumerate}
 \item[(i)] $Y$ is maximally del Pezzo and normal;
 \item[(ii)] $Y$ is a birational projection of a variety $Y'$ of minimal degree.
\end{enumerate}

Assume (i). 
By \cite[Theorem~6.8]{BS07},  $Y$ is a smooth del Pezzo variety in the sense of \cite[Definition 5.6]{Fuj82b}. 
Then the inequality $\deg Y = g-1 \geq 6$ implies $Y \simeq \bP^2 \times \bP^2$ \cite[Thoerem (6.3)]{Fuj82b}.
This is absurd, since $\rho(X)=1$ and $\dim X =3$.

Assume (ii). 
Note that $\pi : Y' \to Y$ is a finite birational morphism, because $\pi$ is a projection from a point outside $Y$. 
As $Y$ is smooth, we get $\pi : Y' \xrightarrow{\simeq} Y$. 
In particular, $Y'$ is a smooth variety of minimal degree satisfying $\deg Y' = \deg Y = g-1 \geq 6$. 
In particular, $Y \simeq Y' \simeq \P_{\P^1}(E)$ for a locally free sheaf $E$ on $\P^1$ of rank $4$. 
As $\rho(X)=1$, $X$ does not dominate $\P^1$, and hence $X$ is contained in a fibre of the $\P^3$-bundle structure $Y \to \P^1$. 
We then get $X \simeq \P^3$, which is absurd.
Thus (1) holds.

Let us show (2). 
Suppose that $Y$ is a  smooth variety
 satisfying $\dim Y =3$ and 
 $\deg Y =g-1$. 
Let $\psi: X \to Y$ be the induced morphism.  
It follows from  $2g -2 =\deg X = (\deg \psi) \deg Y = (\deg \psi) (g-1)$ that 
$\deg \psi =2$. 
We have $\deg Y = g-1 = (g -3)+2= \codim Y +2$, and hence $Y$ is a variety of almost minimal degree in the sense of \cite[Convention 2.4]{BS07}. 
By the same argument as in (1), 
either $Y$ is a smooth del Pezzo threefold (i.e., a smooth Fano threefold of index $2$) or 
$Y \simeq \P_{\P^1}(E)$ for a vector bundle on $\P^1$ of rank $3$. 
The latter case is impossible, because the inequality $1 = \rho(X) \geq \rho(Y)$ implies $\rho(Y)=1$. 
Then $Y$ is a smooth Fano threefold of index $2$ with $\rho(Y)=1$. 
In this case, we get $1  \leq H_Y^3 \leq 5$ for $H_Y := \MO_{\P^g}(1)$ \cite[Remark 4.9]{FanoII}. 
This leads to the following contradiction: 
\[
12=2 \cdot 7-2 \leq 2g-2 =(-K_X)^3 = (\psi^*H_Y)^3 = (\deg \psi) H_Y^3 \leq 2 \cdot 5 =10. 
\]
Thus (2) holds. 
\end{proof}
\section{Genus-free results}

\subsection{$H^{>0}(S,\cQ^\vee_S(H)) = 0$}\label{ss 4.4 vanishing}

\begin{lem}\label{l too big rank}
Let $(S, H)$ be a smooth polarised K3 surface over a perfect $C_1$-field $\kappa$ such that $\Pic(S)   = \Z H$. 
Let  $\cF$ be a $\mu$-stable vector bundle on $S$. 
Then both $r(\cF)$ and $s(\cF)$ are divisible by $\dim_\kappa\Hom(\cF, \cF)$. 
In particular, $\dim_\kappa\Hom(\cF, \cF) \leq r(\cF)$ and 
$\dim_\kappa\Hom(\cF, \cF) \leq s(\cF)$. 
\end{lem}

\begin{proof}
For $n := \dim_\kappa \Hom(\cF, \cF)$, 
there exist locally free subsheaves $\cF_1, ..., \cF_n$ of 
$\cF_{\ol{\kappa}} := \cF \otimes_{\kappa} \ol{\kappa}$ such that 
\[
\cF_{\ol{\kappa}} = \cF_1 \oplus \cdots \oplus \cF_n, 
\]
$r(\cF_1) = \cdots =r(\cF_n)$, and 
$s(\cF_1) = \cdots =s(\cF_n)$ \cite[Corollary 3.8]{Tan-bdl}. 
By additivity of Mukai vectors, we get 
\[
r(\cF) = \sum_{i=1}^n r(\cF_i) = n r(\cF_1) \qquad \text{and}\qquad 
s(\cF) = \sum_{i=1}^n s(\cF_i) = n s(\cF_1), 
\]
as required. 
\end{proof}

\begin{prop}\label{p U(H) vanishing}
Let $(S, H)$ be a smooth polarised K3 surface 
of genus $g$ over a perfect $C_1$-field $\kappa$ such that $\Pic(S)   = \Z H$. 
Let  $\cF$ be a $\mu$-stable vector bundle on $S$ 
with Mukai vector $v(\cF) = (r, -H, s)$. 
Assume that $r\geq 2, s \geq 2$, and $g=rs$. 
Then $H^1(S,\cF(H)) = H^2(S,\cF(H)) = 0$. 
\end{prop}

\begin{proof}
For a coherent sheaf $\cG$ on $S$, we set $h^i(\cG,\cG) \coloneqq \dim \Ext^i(\cG,\cG)$ and 
\[
\chi(\cG,\cG) \coloneqq h^0(\cG,\cG)-h^1(\cG,\cG)+h^2(\cG,\cG).
\]

We first prove that
$H^2(S,\cF(H))=0$. 
By Serre
duality, it suffices to show that 
$\Hom(\cF(H),\cO_S) =0$, 
which follows from the fact that 
$\cF(H)$ is $\mu$-stable, 
$\mu_H(\MO_S)=0$, and $\mu_H(\cF(H)) = \frac{c_1(\cF(H)) \cdot H}{r} 
= \frac{(c_1(\cF) + rH) \cdot H}{r} = \frac{r-1}{r}H^2>0$.

Suppose  $H^1(S, \cF(H)) \ne 0$.
By Serre duality, we get 
\[
\Ext^1(\cF, \MO_S(-H)) \simeq H^1(S, \cF^\vee(-H)) \neq 0,
\]
which induces a non-split extension
\begin{equation}
\label{e1 U(H) vanishing}
0 \longrightarrow \cO_S(-H)
\longrightarrow \widetilde{\cF}
\longrightarrow \cF
\longrightarrow 0.
\end{equation}
We have
\[
v(\widetilde{\cF})
=v(\MO_S(-H)) + v(\cF) = (1, -H, g) + 
(r, -H, s) = (r+1,-2H,s+g)
\]
and the Riemann--Roch theorem implies that
\[
h^0(\widetilde{\cF},\widetilde{\cF}) \geq 
\frac{1}{2}\chi(\widetilde{\cF},\widetilde{\cF})
=
\frac{1}{2}(2(r+1)(s+g)
-
4(2g-2)) =(r+1)(s+rs)
-
2(2rs-2)
\]
\[
=s(r^2 +2r+1) -4rs +4=s(r-1)^2 +4 >r+1. 
\]
In particular, $\wt{\cF}$ is not $\mu$-stable (Lemma \ref{l too big rank}).

Since $\cO_S(-H)$ and $\cF$ are $\mu$-stable bundles of negative slopes, every subsheaf of $\widetilde{\cF}$ has a negative
slope.
Hence the same holds also for all Jordan--H\"older factors in the graded quotients of the Harder--Narasimhan filtration of $\widetilde{\cF}$ with respect to $\mu$-stability.
However,the sum of the first Chern classes of the Jordan--H\"older factors is at least
\[
c_1(\widetilde{\cF})=-2H.
\]
Hence, there are only two Jordan--H\"older factors, and we have the following short exact sequence
\[
0 \longrightarrow \cF_1
\longrightarrow \widetilde{\cF}
\longrightarrow \cF_2
\longrightarrow 0, 
\]
where each $\cF_i$ is a torsion-free $\mu$-stable sheaf. 
For the Mukai vector  $v(\cF_i)=(r_i,-H,s_i)$, we get 
\[
-\frac{1}{r_1}
\ge
-\frac{2}{r+1}
\ge
-\frac{1}{r_2}.
\]
In particular, we have $r_1\ge 2$.

Let us show $r_2 \geq 2$. 
Suppose that $r_2=1$. 
Then 
$\cF_2^{\vee\vee} \simeq \MO_S(-H)$ and the composition
\[
\zeta: 
\cO_S(-H)
\hookrightarrow
\widetilde{\cF}
\twoheadrightarrow
\cF_2 \hookrightarrow \cF_2^{\vee\vee} (\simeq \MO_S(-H))
\]
is either zero or an isomorphism. 
If $\zeta$ is an isomorphism, then 
$\MO_S(-H) \hookrightarrow \widetilde{\cF}$ 
is a split injection, contradicting 
the fact that \eqref{e1 U(H) vanishing} is not split. 
If $\zeta$ is zero, then 
we obtain an injection 
$\cF_1/\cO_S(-H) \hookrightarrow \wt{\cF}/\MO_S(-H) \simeq \cF$, 
which is absurd, because $\cF$ is $\mu$-stable and 
$\mu_H(\cF) <0 = \mu_H(\cF_1/\cO_S(-H))$. 
This completes the proof of  $r_2 \geq 2$. 
In particular, 
\begin{equation}
\label{e2 U(H) vanishing}
r = r_1 + r_2 -1 \geq 2 + 2-1 =3. 
\end{equation}

As $\cF_i$ is a $\mu$-stable torsion free sheaf, 
$\cF_i^{\vee\vee}$ is a $\mu$-stable vector bundle. 
Then Lemma \ref{l too big rank} implies 
\[
h^0(\cF_i, \cF_i) \leq h^0(\cF_i^{\vee\vee}, \cF_i^{\vee\vee}) \leq r_i. 
\]
This, together with the Riemann-Roch theorem, implies that 
\[
2r_is_i - 2g+2 = \chi(\cF_i, \cF_i) = 2h^0(\cF_i, \cF_i) -h^1(\cF_i, \cF_i) \leq 2h^0(\cF_i, \cF_i) \leq 2r_i, 
\]
and hence $s_i \leq 1 + \frac{g-1}{r_i}$. 
Then it holds that 
\[
s+g
=
s_1+s_2
\leq
2+ 
\frac{g-1}{r_1}+\frac{g-1}{r_2}. 
\]
As we have $r_1r_2 = r_1 ( r+1 -r_1) \geq 2(r-1)$ by $r_1 \geq 2$ and $r_2 \geq 2$, we obtain 
\[
s+rs-2 =s+g -2 \leq \frac{g-1}{r_1} + \frac{g-1}{r_2}
=\frac{(g-1)(r_1+r_2)}{r_1r_2} \leq \frac{(rs-1)(r+1)}{2(r-1)}. 
\]
We then get $2(r-1)(s(r+1)-2) \leq (rs-1)(r+1)$, 
which leads to the following contradiction: 
\[
0 \geq 2(r-1)(s(r+1)-2) - (rs-1)(r+1) 
\overset{(\star)}{\geq} 
2(r-1)(2(r+1) -2) -(2r-1)(r+1) 
\]
\[
= (4r^2-4r) -(2r^2 +r-1) = 2r^2 -5r+1 
\overset{ (\ref{e2 U(H) vanishing})}{\geq} 2 \cdot 3^2 - 5 \cdot 3 +1 >0,
\]
where $(\star)$ holds by $s \geq 2$ and the fact that 
the coefficient of $s$ in the left-hand side is positive by 
$2(r-1)(r+1) -r(r+1) = r^2-r -2>0$. 
\end{proof}

\begin{prop}\label{p U(H) vanishing2}\label{p BKM4.4 vanishing}
Let $X$ be a prime Fano threefold of genus $g$ such that $|-K_X|$ is very ample. 
Let $\cU_X$ be a Mukai bundle of type $(r, s)$, 
where $r$ and $s$ are integers satisfying $g=rs$, $r \geq 2$, and $s \geq 2$. 
Take a general member $S$ of $|-K_X|$. Set $\cU_S := \cU_X|_S$ and $H := -K_X|_S$. 
Then $H^1(S,\cU_S(H)) = H^2(S,\cU_S(H)) = 0$. 
\end{prop}

\begin{proof}
Take two general members $S_1$ and $S_2$ of $|-K_X|$. 
Let $T_0$ be the generic member of the pencil generated by $S_1$ and $S_2$. 
For the base field $\kappa_0 := k(t)$ of $T_0$ and its purely inseparable closure $\kappa :=k(t^{1/p^{\infty}})$, 
set $T := T_0\times_{\kappa_0} \kappa$ and 
let $H_T$ be the pullback of  $-K_X$ to $T$.
Then $(T, H_T)$ is a smooth polarised K3 surface over a perfect $C_1$-field $\kappa$. 
Moreover, we have $\Pic T= \Z H_T$  \cite[Proposition 3.5]{Tan-bdl} and  $\cU_T$ is $\mu$-stable  \cite[Proposition 6.6]{Tan-bdl}. 
Therefore, we get 
$H^1(T,\cU_T(H_T)) = H^2(T,\cU_T(H_T)) = 0$ (Proposition \ref{p U(H) vanishing}). 
Then the assertion follows from upper semi-continuity. 
\end{proof}

\subsection{Schubert avoidance}

\begin{definition}[\cite{Muk10}*{Definition 3.1}]
Let $X$ be a projective variety. 
Take a vector bundle $E$ of rank $r$ and a subspace $V \subset H^0(E)$.
We say that $(E,V)$ is {\em irreducible} if the following equivalent conditions hold:
\begin{enumerate}
 \item For every $r$-dimensional subspace $U \subset V$, the natural 
 evaluation homomorphism
\[
U \otimes \cO_X \to E
\]
is generically isomorphic.
 \item The kernel of the natural evaluation map 
 \[
\lambda _X \colon  \wedge^r  V \to H^0(X,\wedge ^r E)
 \]
 contains no decomposable vectors.
 Namely, in $\bP((\wedge^r V)^\vee)$, we have
 \[
 \bP((\Ker \lambda _X)^\vee) \cap \Gr(r,V) =\emptyset.
 \]
\end{enumerate}
\end{definition}

\noindent
In what follows, we only consider the case when $E$ is globally generated and 
$V = H^0(X,E)$.
In this case, the surjection $V \otimes \cO_X \to E$ 
induces a morphism $\varphi \colon X \to \Gr(V,r)$.  
Then, $(E, V)$ is irreducible if and only if 
$\varphi(X)$ 
is not contained in any Schubert divisor.
Here, a \emph{Schubert divisor} $H \subset \Gr(V,r)$ is a hyperplane section of $\Gr(V,r) \subset \bP(\wedge^r V)$ that is parametrised by a point contained in the (dual) Grassmannian variety $\Gr(r,V) \subset \bP(\wedge^r V^{\vee})$.

For a torsion free sheaf $G$ on a smooth variety, set $\det G \coloneqq (\wedge^{\rank G} G)^{\vee\vee}$.
Note that $(\det G) ^\vee \simeq \det(G^\vee)$ since $G$ is locally free in codimension one. 
The following ensures the irreducibility of dual Mukai bundles:

\begin{proposition}\label{proposition:irreducibility}
Let $X$ be a smooth projective variety 
and let $E$ be a vector bundle of rank $r$ satisfying 
$(\star)$.
\begin{enumerate}
\item[$(\star)$] If $E'$ is a subsheaf of $E$ of rank $0<r' < r$, then $H^0(X, (\det E')^\vee) \neq 0$. 
\end{enumerate}
Then $(E,H^0(X,E))$ is irreducible.
\end{proposition}

\begin{proof}
Set $V \coloneqq H^0(X,E)$.
Take an $r$-dimensional subspace $U \subset V$ and 
let $e\colon U \otimes \cO_X \to E$ be the evaluation homomorphism.
Suppose that it is not generically isomorphic. 
Then the image $E'$ is of rank $r' < r$.
By taking the dual of the surjection $U \otimes \cO_X \twoheadrightarrow E'$, we get an injection 
$(E')^\vee \hookrightarrow U^\vee \otimes \cO_X$ from a reflexive sheaf $(E')^\vee$.
Applying $\det(-)$, we obtain an injection 
\[
\det( (E')^\vee) \hookrightarrow \wedge^{r'} U^\vee \otimes \cO_X.
\]
We have $H^0(X, \det( (E')^\vee)) \simeq 
H^0(X, (\det E')^\vee) \neq 0$ by ($\star$). 
Then $(E')^\vee \simeq \MO_X^{\oplus r'}$ \cite[Lemma 2.4]{Tan-bdl}. 
Therefore, we have the following decomposition of $e$: 
\[
e\colon U\otimes \cO_X  \twoheadrightarrow  E' 
\hookrightarrow (E')^{\vee\vee} 
\hookrightarrow E.
\]
By taking $H^0(X, -)$, 
the inclusion $H^0(e) : U \hookrightarrow H^0(E)$ 
from the $r$-dimensional vector space $U$ factors through an $r'$-dimensional vector space 
$H^0(X, E^{\vee\vee}) \simeq H^0(X, \MO_X^{\oplus r'})$.
This is absurd since $r' < r$. 
\end{proof}

\begin{proposition}\label{p Sch avoid X}
Let $X$ be a prime Fano threefold of genus $g$ such that $|-K_X|$ is very ample. 
Let $\cQ_X$ be a dual Mukai bundle of type $(r, s)$, where $r$ and $s$ are integers satisfying $g=rs$, $r \geq 2$, and $s \geq 2$.
Set $V :=H^0(X, \cQ_X)$.
Then $(\cQ_X, V)$ is irreducible. 
In particular, the image $\varphi(X)$ of the morphism $\varphi \colon X \to \Gr(V, r)$ 
induced by $\cQ_X$ is not contained in any Schubert divisor on $\Gr(V, r)$. 
\end{proposition}

\begin{proof}
Since $\cQ_X$ is $\mu$-stable and $\Pic X = \Z K_X$, 
the vector bundle $\cQ_X$ satisfies the condition $(\star)$ of Proposition~\ref{proposition:irreducibility}. 
Thus $(\cQ_X, V)$ is irreducible.
\end{proof}

\begin{prop}\label{p schubert smooth}
Take integers $\nu, r, s, g$ satisfying 
$\nu \geq 0$, 
$r \geq 2, s \geq 2$, and $g =rs$. 
For a $(r+s)$-dimensional linear space $V$, 
let $\Gr(V, r) \subset \P( \wedge^r V)$ denote the Pl\"{u}cker embedding. 
For a $(g+\nu)$-dimensional linear subvariety $\P^{g+\nu}$, 
let $Y$ be a closed subscheme of $\Gr(V, r) \cap \P^{g+\nu}$ which is not contained in any Schubert divisor of $\Gr(V, r)$. 
Then $\dim T_{Y, y} \leq 3 +\nu$ for every closed point $y \in Y$, 
where $T_{Y, y}$ denotes the tangent space of $Y$ at $y$. 
\end{prop}

\begin{proof}
As mentioned in \cite[Remark 3.8]{BKM26b}, 
the same argument as in  \cite[Lemma 3.7]{BKM26b} works. 
\end{proof}

\subsection{Embedding into Grassmannians}\label{ss emb Gr}

\begin{lem}\label{l Uvee rest}
Let $X$ be a prime Fano threefold of genus $g$ such that $|-K_X|$ is very ample.  
Take a dual Mukai bundle $\cQ_X$ of type $(r, s)$, 
where $r$ and $s$ are integers satisfying $g=rs$, $r \geq 2$, and $s \geq 2$. 
Let $S$ be a member of $|-K_X|$. 
Then we have the following induced isomorphism:  
\[
H^0(X, \cQ_X) \xrightarrow{\simeq} H^0(S, \cQ_S). 
\]
\end{lem}

\begin{proof}
The assertion follows from 
$H^i(X, \cQ_X(K_X)) \simeq H^{3-i}(X, \cQ^\vee_X)^\vee=0$. 
\end{proof}

\begin{prop}\label{p wedge2 surje S}
Let $X$ be a prime Fano threefold of genus $g$ and 
let $\cQ_X$ be a dual Mukai bundle of type $(r, s)$, 
where $r$ and $s$ are integers satisfying $g=rs$, $r \geq 2$, and $s \geq 2$.  
Let $S$ be a general member of $|-K_X|$. 
Set $\cQ_S := \cQ_X|_S$. 
Then 
$\wedge^r H^0(S, \cQ_S) \to H^0(S, \wedge^r \cQ_S)$ 
is surjective. 
In particular, the morphism $\varphi_S : S \to \Gr(V, r)$ 
induced by $\cQ_S$ is a closed immersion for $V := H^0(S, \cQ_S)$.  
\end{prop}

\begin{proof}
We have the following commutative diagram
\[
\begin{tikzcd}
\wedge^r H^0(X, \cQ_X) \arrow[r,"\lambda_X"] \arrow[d,"\simeq"] & H^0(X,\wedge^r \cQ_X) = H^0(X, -K_X) \arrow[d] \\
 \wedge^r H^0(S,\cQ_S) \arrow[r,"\lambda_S"] & H^0(S,\wedge^r \cQ_S) = H^0(S, H), 
\end{tikzcd}
\]
where the left vertical arrow is an isomorphism by Lemma \ref{l Uvee rest}. 
It follows from Proposition~\ref{p Sch avoid X} that $\bP((\Ker \lambda_X)^\vee) \cap \Gr(r,V) = \emptyset$ in $\bP(\wedge^rV^\vee)$.
Thus we get
\[
\dim(\Ker \lambda_X) = \dim \bP((\Ker \lambda_X)^\vee) +1 < \dim \bP(\wedge^rV^\vee) - \dim \Gr(r,V) +1 = 
\dim \wedge^rV 
-rs. 
\]
Then
\[
\dim \Im \lambda_ X = \dim (\wedge^r H^0(X, \cQ_X)) - \dim (\Ker \lambda_X)
= \dim \wedge^rV - \dim (\Ker \lambda_X) > rs =g. 
\]
Hence $\dim \Im \lambda_X =g+2$ or $\dim \Im \lambda_X =g+1$. 
As $S \in |-K_X|$ is chosen to be general, the induced map $\Im \lambda_X \to H^0(S,\wedge^r \cQ_S) = H^0(S, H)$ is surjective, and  so is $\lambda_S$.
\end{proof}

\begin{prop}\label{p Gushel emb X}
Let $X$ be a prime Fano threefold of genus $g \geq 8$ and let $\cQ_X$ be a dual Mukai bundle of type $(r, s)$. 
Set $V := H^0(X, \cQ_X)$. 
Then the morphism $\varphi: X \to \Gr(V, r)$  induced by $\cQ_X$ is a closed immersion. 
\end{prop}

\begin{proof}
Since $\varphi|_S : S \to \Gr(V, r)$ is a closed immersion 
induced by the complete linear system $|-K_X|_S|$ 
(Lemma \ref{l Uvee rest}, Proposition \ref{p wedge2 surje S}), 
we get  $\dim \la S \ra =g$. 
Set $\P^g :=\la S \ra$ and $Y := \varphi(X)$. 
If $\la S \ra \subsetneq \la Y \ra$, then we get $\dim \la Y \ra =g+1$, which implies that $\varphi : X \to \Gr(V, r)$ is a 
closed immersion by the complete linear system $|-K_X|$.

In what follows, we assume $\la Y \ra = \la S \ra (=\P^g)$. 
As $Y \subset \Gr(V, r) \cap \P^g$, 
we see that $Y$ is a smooth projective threefold by 
Proposition \ref{p schubert smooth},which is applicable by Proposition \ref{p Sch avoid X}. 
Since $1 = \rho(X) \geq \rho(Y)$ and 
the adjunction formula $(K_{Y}+S)|_S \sim K_S \sim 0$, 
$Y$ is a smooth Fano threefold. 
For the induced morphism $\ol{\varphi}: X \to Y$, 
it holds that 
\[
\frac{2g-2}{\deg \ol{\varphi}} 
=\deg Y \geq \codim Y +1 = (g-\dim Y) +1 = g-2,  
\]
which implies $\deg\ol{\varphi} \leq \frac{2g-2}{g-2} = 2 + \frac{2}{g-2}<3$. 
Hence $\deg\ol{\varphi} =1$ or $\deg\ol{\varphi}=2$. 
The case  $\deg(\ol{\varphi})=2$ is excluded  by Lemma \ref{lemma:not_in_almost_minimal_sm}(2), 
which is applicable by $\dim Y =3$ and $\deg Y = (-K_X)^3/2 = g-1$. 
Hence we get $\deg\ol{\varphi} =1$. 
Then $\ol{\varphi} : X \to Y$ is an isomorphism, because 
$\ol{\varphi}$ is a finite birational morphism of smooth projective varieties.  
\end{proof}

\section{Genus 9}

\subsection{The Lagrangian Grassmannian $\Sigma_9$}
Let $V_6$ be a $6$-dimensional vector space and 
let $\Gr(V_6,3)$ be the Grassmannian variety parametrising  $3$-dimensional quotients of $V_6$.
We have the universal exact sequence
\[
0 \to \cK \to V_6 \otimes \MO_{\Gr(V_6,3)} \to \cQ \to 0.
\]
The third exterior power 
of this sequence 
induces the Pl\"ucker embedding $\Gr(V_6,3) \to \bP(\wedge^3 V_6)$.

\begin{definition}
Let $V_6$ be a $6$-dimensional vector space. 
\begin{enumerate}
 \item An element $\sigma$ of $\wedge^2 V_6$ is called a \emph{bivector}. 
 \item For a bivector $\sigma \in \wedge^2 V_6$,  we have the induced alternating map
\[
e_{\sigma} \colon V_6^\vee \to V_6.
\]
Note that the rank of an alternating map is even, i.e., $\rank e_{\sigma} = 6$, $4$, $2$, or $0$.
\item For a bivector $\sigma \in \wedge^2 V_6$, 
we have the following induced linear map:  
\begin{align*}
m_\sigma  \colon 
V_6 &\to \wedge^3 V_6\\
v &\mapsto \sigma \wedge v
\end{align*}
\end{enumerate}
\end{definition}

\begin{proposition}\label{proposition:bivector_nondegenerate}
Let $V_6$ be a $6$-dimensional vector space and 
take a nonzero bivector $\sigma \in \wedge^2 V_6$. 
Then the following are equivalent:
\begin{enumerate}
 \item $\rank e_\sigma =6$.
 \item The image of $m_\sigma$ contains no nonzero decomposable vector $u_1\wedge u_2 \wedge u_3 \in \wedge^3 V_6$. 
\end{enumerate}
Moreover, if these conditions are satisfied, then 
$m_\sigma : V_6 \to \wedge^3 V_6$ is injective and 
there is a basis $\{e_1,e_2,e_3,e_{-1},e_{-2},e_{-3}\}$ of $V_6$ such that $\sigma =e_1 \wedge e_{-1} +e_2\wedge e_{-2} + e_3\wedge e_{-3}$. 
\end{proposition}

\begin{proof}
Since bivectors are classified by their ranks 
\cite[Lemma 1.4 and Proposition 1.8]{EKM08},  we may assume that $\sigma = 
e_1 \wedge e_{-1},  e_1 \wedge e_{-1} +e_2\wedge e_{-2}$,  or $e_1 \wedge e_{-1} +e_2\wedge e_{-2} + e_3\wedge e_{-3}$ for a basis $\{e_1,e_2,e_3,e_{-1},e_{-2},e_{-3}\}$ of $V_6$.

Let us show the implication $(1) \Leftarrow (2)$. 
Assume that $\rank e_{\sigma} \neq 6$, i.e., 
$\rank e_{\sigma} \in \{2, 4\}$. 
If $\rank e_{\sigma} =2$ (resp.\ $\rank e_{\sigma} =4$), then 
we have 
$\sigma = e_1 \wedge e_{-1}$ 
(resp.\ $\sigma =e_1 \wedge e_{-1} +e_2\wedge e_{-2}$), and hence 
$m_{\sigma}(e_2) = e_1 \wedge e_{-1} \wedge e_2$ is decomposable. 
This completes the proof of the implication $(1) \Leftarrow (2)$.

Let us show that (1) implies (2) and the injectivity of $m_\sigma$. 
Assume (1), i.e., $\rank e_{\sigma} = 6$. 
Then $\sigma = e_1 \wedge e_{-1} +e_2\wedge e_{-2} + e_3\wedge e_{-3}$. 
For $\tau \in \wedge^3 V_6$, we set 
$K(\tau) := \{ v \in V_6\,|\, \tau \wedge v =0\}$. 
It is easy to see that 
\begin{itemize}
\item $K( m_{\sigma}(e_1)) =K(e_2\wedge e_{-2} \wedge e_1 + e_3\wedge e_{-3} \wedge e_1) = k e_1$, and 
\item $\dim K( u_1 \wedge u_2 \wedge u_3) >1$ for every $u_1, u_2, u_3 \in V$. 
\end{itemize}
Recall that the symplectic group 
$\Sp(V_6, \sigma) := \{ g \in \GL(V_6) \,|\, g \sigma = \sigma\}$ acts on $V_6 \setminus \{0\}$ transitively. 
For every  element $v \in V_6 \setminus \{0\}$, 
we can find $g \in \Sp(V_6, \sigma)$ such that $ge_1 = v$, 
which implies 
\[
\dim K(m_\sigma(v)) = \dim K ( (g\sigma) \wedge (ge_1))  = 
\dim K( \sigma \wedge e_1) =\dim  K(m_\sigma(e_1)) = 1.  
\]
In particular, $m_\sigma(v) \neq 0$, and hence $m_\sigma$ is injective. 
Moreover, $m_\sigma(v)$ is not decomposable for every nonzero $v \in V_6$. 
Therefore, (1) implies (2) and the injectivity of $m_\sigma$. 
\end{proof}

\begin{definition}
Let $V_6$ be a $6$-dimensional vector space. 
We say that a bivector $\sigma \in \wedge^2 V_6$ is \emph{non-degenerate} if 
$\sigma \neq 0$ and $\sigma$ satisfies the equivalent conditions (1) and (2) in Proposition~\ref{proposition:bivector_nondegenerate}.
\end{definition}

\begin{definition}
Let $V_6$ be a $6$-dimensional vector space. 
For a bivector $\sigma \in \wedge^2 V_6$, 
$\Gr(V_6,3,\sigma)$ denotes the zero locus of the corresponding 
section in $H^0(\Gr(V_6,3), \wedge^2 \cQ)$ via the canonical isomorphism 
$\wedge^2 V_6\xrightarrow{\simeq} H^0(\Gr(V_6,3), \wedge^2 \cQ)$. 
Specifically, 
\[
\Gr(V_6,3,\sigma) = \{ [q : V_6 \twoheadrightarrow Q] 
\in \Gr(V_6,3) 
\,|\, (\wedge^2 q) (\sigma)=0 \text{ in } \wedge^2 Q\}.
\]
We set $\Sigma_9:= \Gr(V_6,3,\sigma)$ for a non-degenerate bivector $\sigma \in \wedge^2 V_6$. 
\end{definition}

\begin{proposition}\label{proposition:Sigma9_schubert}
Let $V_6$ be a $6$-dimensional vector space and 
let $\Gr(V_6, 3) \subset \P(\wedge^3 V_6)$ denote the Pl\"{u}cker embedding. 
Take a bivector $\sigma \in \wedge^2 V_6$. 
Then the following hold:
\begin{enumerate}
\item \label{proposition:Sigma9_schubert1} 
In $\bP(\wedge^3 V_6)$, the scheme-theoretic equality
$\Gr(V_6,3,\sigma) = \Gr(V_6,3) \cap \bP(\Coker m_\sigma )$ holds.
\item \label{proposition:Sigma9_schubert2} $\sigma$ is non-degenerate if and only if $\Gr(V_6,3,\sigma)$ is not contained in any Schubert divisor of $\Gr(V_6,3)$.
\item \label{proposition:Sigma9_schubert3}
If $\sigma$ is non-degenerate, then 
$\bP(\Coker m_\sigma) \simeq \P^{13}$ and 
$\Gr(V_6,3,\sigma)$ is not contained in any hyperplane on $\bP(\Coker m_\sigma)$. 
\end{enumerate}
\end{proposition}

\begin{proof}
The assertion (\ref{proposition:Sigma9_schubert1}) follows from Proposition~\ref{proposition:bivector_nondegenerate} and \cite[Lemma 2.1]{BKM26b}.

Let us prove (\ref{proposition:Sigma9_schubert2}). 
Take a $3$-dimensional subspace 
$U_3=ku_1\oplus ku_2\oplus ku_3\subset V_6.$ 
Then the Schubert divisor corresponding to $U_3$ is given by 
\[
\Sigma_1(U_3)
:=
\left\{
[q:V_6\twoheadrightarrow Q_3]\in \Gr(V_6,3)
\ \middle|\ 
\dim \Im(U_3\to Q_3)<3
\right\}.
\]
Hence, for $[q:V_6\twoheadrightarrow Q_3]\in \Gr(V_6,3)$, we have
\[
[q:V_6\twoheadrightarrow Q_3]\in \Sigma_1(U_3)
\Longleftrightarrow 
(\wedge^3 q) (u_1\wedge u_2 \wedge u_3)=0
\text{ in }\wedge^3Q_3.
\]
Therefore, 
\[
\begin{aligned}
\Gr(V_6,3,\sigma)\subset \Sigma_1(U_3)
& \overset{\text{(1)}}{\Longleftrightarrow}
\P(\Coker m_\sigma)
\subset
\left\{
[\wedge^3 q]\in \P(\wedge^3V_6)
\ \middle|\ 
(\wedge^3q)(u_1\wedge u_2\wedge u_3)=0
\right\}
\\
&\Longleftrightarrow
u_1\wedge u_2\wedge u_3
\in
\Ker(\wedge^3V_6\to \Coker m_\sigma)
\\
&\Longleftrightarrow
u_1\wedge u_2\wedge u_3
\in
\Image(m_\sigma).
\end{aligned}
\]
Thus (\ref{proposition:Sigma9_schubert2}) holds.

Let us show (\ref{proposition:Sigma9_schubert3}). Since $m_\sigma : V_6 \to \wedge^3 V_6$ is injective 
(Proposition \ref{proposition:bivector_nondegenerate}), 
we get $\dim (\Coker m_\sigma) = \dim (\wedge^3 V_6) -\dim V_6 = \binom{6}{3} -6 = 20-6 =14$, which implies 
$\P(\Coker m_\sigma) \simeq \P^{13}$. 

Take a $k$-linear basis 
$\{e_1,e_2,e_3,e_{-1},e_{-2},e_{-3}\}$ of $V_6$ such that 
\[
\sigma =e_1 \wedge e_{-1} +e_2\wedge e_{-2} + e_3\wedge e_{-3}.
\]
Let $\{e_i^\vee\}$ be the dual basis of $\{e_i\}$.
We can directly check that 
$\Gr(V_6,3,\sigma)$ contains the following $14$ points:
\begin{align*}
& e_{1}^\vee \wedge  e_{2}^\vee \wedge  e_{3}^\vee, \quad
e_{-1}^\vee \wedge e_{-2}^\vee \wedge e_{-3}^\vee,  \\
& e_{1}^\vee \wedge  e_{-2}^\vee \wedge  e_{-3}^\vee, \quad
e_{-1}^\vee \wedge  e_{2}^\vee \wedge  e_{-3}^\vee,\quad
e_{-1}^\vee \wedge  e_{-2}^\vee \wedge  e_{3}^\vee,  \\
& e_{-1}^\vee \wedge  e_{2}^\vee \wedge  e_{3}^\vee, \quad
e_{1}^\vee \wedge  e_{-2}^\vee \wedge  e_{3}^\vee, \quad
e_{1}^\vee \wedge  e_{2}^\vee \wedge  e_{-3}^\vee,  \\
& e_{1}^\vee \wedge  (e_{2}^\vee +e_{3}^\vee) \wedge  (e_{-2}^\vee -e_{-3}^\vee), \quad
e_{2}^\vee \wedge  (e_{3}^\vee +e_{1}^\vee) \wedge  (e_{-3}^\vee -e_{-1}^\vee), \quad
e_{3}^\vee \wedge  (e_{1}^\vee +e_{2}^\vee) \wedge  (e_{-1}^\vee  -e_{-2}^\vee), \\
& e_{-1}^\vee \wedge  (e_{-2}^\vee +e_{-3}^\vee) \wedge  (e_{2}^\vee - e_{3}^\vee),\quad
e_{-2}^\vee \wedge  (e_{-3}^\vee +e_{-1}^\vee) \wedge  (e_{3}^\vee - e_{1}^\vee), \quad
e_{-3}^\vee \wedge  (e_{-1}^\vee +e_{-2}^\vee) \wedge  (e_{1}^\vee  -e_{2}^\vee).
\end{align*}
For example, we get 
\[
[q] := [e_{1}^\vee \oplus   e_{2}^\vee \oplus   e_{3}^\vee] \in \Gr(V_6,3,\sigma) = \{ [q : V_6 \twoheadrightarrow Q] \in  \Gr(V_6,3) \,|\, (\wedge^2 q)(\sigma)=0 \text{ in }\wedge^2 Q\}
\]
by 
$(\wedge^2q)(\sigma) = \sum_{i=1}^3 
(\wedge^2 q)(e_i \wedge e_{-i}) =  \sum_{i=1}^3 
 q(e_i) \wedge q(e_{-i}) =  \sum_{i=1}^3 
 q(e_i) \wedge 0=0.$ 
Since the above  14 elements (considered as elements of  $\wedge^3 V_6^\vee$) are linearly independent over $k$, 
they span $\bP(\Coker m_\sigma) \simeq \P^{13}$. 
Thus (3) holds. 
\end{proof}

\begin{prop}
Let $V_6$ be a $6$-dimensional vector space and 
let $\Gr(V_6, 3) \subset \P(\wedge^3 V_6)$ denote the Pl\"{u}cker embedding. Take a non-degenerate bivector $\sigma \in \wedge^2 V_6$. 
Then $\Sigma_9 = \Gr(V_6,3,\sigma) \simeq \Sp(6)/P$, 
where $\Sp(6)$ denotes the symplectic group of the pair $(V_6,\sigma \in \wedge^2 V_6)$, which is a connected semisimple algebraic group 
of type $C_3$, and $P$ is the maximal parabolic subgroup corresponding to the last node of the Dynkin diagram $C_3$. 
In particular, $\Sigma_9$ is a smooth Fano $6$-fold of index $4$ and degree $16$.
\end{prop}

\begin{proof}
Let $\bP(\wedge^2 \cQ)$ be the projectivisation of $\wedge^2 \cQ$.
Then this is isomorphic to the flag variety $\Fl(V_6;3,2)$, which admits the following natural diagram:
\[
\begin{tikzcd}
&\bP(\wedge^2 \cQ) = \Fl(V_6;3,2) \arrow[ld,"p_1"'] \arrow[rd,"p_2"]&  \\
\Gr(V_6,3) &  & \Gr(V_6,2).
\end{tikzcd}
\]
Here $p_1$ and $p_2$ are projective bundles.
For example, for a point $x =[V_6 \to Q_2] \in \Gr(V_6,2)$, the fiber $p_2^{-1}(x)$ is isomorphic to $\Gr(V_6/Q_2,1) \simeq \bP^3$.

A non-degenerate bivector $\sigma \in \wedge^2 V_6$ defines a smooth hyperplane section of  $\Gr(V_6,2) \subset \bP(\wedge^2 V_6)$, and hence a smooth element of $|\MO_{\bP(\wedge^2 \cQ)}(1)|$.
Thus, by \cite[Proposition 1.9]{Muk10b}, the zero locus $\Gr(V_6,3,\sigma)$ is a smooth variety of dimension $6$.
The remaining assertions are proved by the same argument as in \cite[Proposition 2.3]{Muk10}.
For more details, see, e.g., \cite[Section 2]{Muk10}.
\end{proof}

\begin{lemma}\label{lemma:Sigma9_Chow}
There exist $H \in \CH^1(\Sigma_9)$ and $Y \in \CH^2(\Sigma_9)$ such that 
$\CH^1(\Sigma_9) =\Z H \simeq \Z$, 
$\CH^2(\Sigma_9) = \Z Y \simeq \Z$, and 
the equality $H^2 =2Y$ holds in $\CH^2(\Sigma_9)$, 
where each $\CH^i(\Sigma_9)$ denotes the Chow group of $\Sigma_9$ of codimension $i$. 
\end{lemma}

\begin{proof}
Since $\Sigma_9 = \LG(6, 3) \simeq \LG(3, 6)$, 
it suffices to prove the corresponding assertion for $\LG(3, 6)$. 
By \cite[Theorem 2.1(i) and the paragraph following it]{PR96}, we have 
the decomposition 
\[
\bigoplus_{i\geq 0} \CH^i(\LG(3, 6)) = \bigoplus_{
\substack{I \subset \rho_3,\\ I:\text{strict}}} \Z \sigma_I, 
\]
as a $\Z$-module, 
where $\rho_3 =(3, 2, 1)$ and 
$\sigma_I$ denotes the Schubert class 
corresponding to a strict partition $I \subset \rho_3$ 
(note that, in  \cite{PR96}, $\sigma_I$ is given by $\widetilde Q_I(R^\vee)$). 
In particular, 
$\CH^1(\LG(3, 6)) = \Z \sigma_1$ and $\CH^2(\LG(3, 6)) = \Z \sigma_2$. 
The Pieri formula (or the Chevalley formula) implies 
$\sigma_1^2=2\sigma_2$ \cite[Proposition 4.9]{PR96}, as required. 
\end{proof}

\subsection{Linear section theorem for $g=9$}

\begin{proposition}\label{p g=9 Sigma factor}
Let $X$ be a prime Fano threefold of genus $9$. 
Take a dual Mukai bundle $\cQ_X$ of type $(3, 3)$ and let 
$\varphi : X \hookrightarrow \Gr(V_6, 3)$ be the closed immersion induced by $\cQ_X$ 
(cf.\ Proposition \ref{p Gushel emb X}), 
where  $V_6 \coloneqq H^0(X,\cQ_X)$. 
Then the following hold: 
\begin{enumerate}
\item 
For a bivector $\sigma \in \wedge^2 V_6$ satisfying $\varphi(X) \subset \Gr(V_6,3,\sigma)$, 
$\sigma$ is nonzero if and only if $\sigma$ non-degenerate. 
\item 
There exists a non-degenerate bivector $\sigma \in \wedge^2 V_6$ such that $\varphi(X) \subset \Gr(V_6,3,\sigma)$.
Moreover, such a non-degenerate bivector $\sigma$ is unique up to scalar. 
\end{enumerate}
\end{proposition}

\begin{proof}
Take a smooth general member $S$ of $|-K_X|$. 
In what follows, we identify $X$ and $S$ with $\varphi(X)$ and $\varphi(S)$, respectively. 

Let us show (1). 
A non-degenerate bivector is nonzero by definition. 
Conversely, assume that $\sigma$ is nonzero. 
If $\sigma$ is degenerate, 
then $\Gr(V_6,3,\sigma)$ is contained in a Schubert divisor
(Proposition~\ref{proposition:Sigma9_schubert}). 
This, together with $X \subset \Gr(V_6,3,\sigma)$, 
contradicts Proposition~\ref{p Sch avoid X}. 
Thus $\sigma$ is non-degenerate.
Thus (1) holds. 

Let us show (2). 
We have the following restriction maps: 
\[
\begin{tikzcd}[column sep=large]
\wedge^2V_6 
=
H^0\left(\Gr(V_6, 3),\wedge^2\cQ\right)
\arrow[r, "\rho^{\Gr(V_6, 3)}_X"]
\arrow[rr, bend right=20, "\rho^{\Gr(V_6, 3)}_S"]
&
H^0\left(X,\wedge^2\cQ_X\right)
\arrow[r, "\rho^{X}_S", hook]
&
H^0\left(S,\wedge^2\cQ_S\right), 
\end{tikzcd}
\]
where $\rho^{X}_S$ is injective, since
$H^0(X, \wedge^2 \cQ_X (-S)) \simeq 
H^0(X, \wedge^2 \cQ_X (K_X)) \simeq 
H^0(X,\cQ_X^\vee) =0$. 
In particular, we get 
\[
K := \Ker(\rho^{\Gr(V_6, 3)}_X) = \Ker(\rho^{\Gr(V_6, 3)}_S).
\]
It suffices to show $\dim \Ker(\rho^{\Gr(V_6, 3)}_S) =1$. 
It holds that $\dim \wedge^2 V_6 = \binom{6}{2} =15$. 
By the Riemann-Roch theorem, 
we get 
$h^0\left(S,\wedge^2\cQ_S\right) = 
h^0(S, \cQ_S^\vee(H)) =
\chi(S,\cQ_S^\vee(H)) 
= 14$ (Proposition \ref{p BKM4.4 vanishing}). 
Thus $\dim K = \dim \Ker(\rho^{\Gr(V_6, 3)}_S) \geq  15 -14 =1$. 

It suffices to show $\dim K <2$. 
Suppose that $\dim K\ge 2$, i.e., $\dim \P(K^\vee) \geq 1$. 
The inclusion $K\subset \wedge^2V_6$ induces a closed embedding 
$\P(K^\vee)\subset \P((\wedge^2V_6)^\vee).$ 
We have the following set-theoretic equality: 
\[
\{ [\sigma] \in \P(\wedge^2 V_6^{\vee})\,|\, 
\sigma \text{ is a degenerate bivector}\} = 
Z(\det) \subset  \P(\wedge^2 V_6^{\vee}), 
\]
where $\det$ denotes the homogeneous polynomial of degree $6$ defined by the determinant. 
Since $k$ is algebraically closed, 
there exists a closed point $[\sigma] \in Z(\det) \cap \P(K^\vee)$. 
Then $\sigma$ is a nonzero degenerate bivector satisfying $X \subset \Gr(V_6,3,\sigma)$, contradicting (1). 
Thus (2) holds. 
\end{proof}

\begin{lemma}\label{l hyperplane nondeg}
Let $W\subset \bP^N$ be a projective variety with $\dim W >0$. 
Assume that $W$ is non-degenerate in $\P^N$, i.e., 
$H^0(\bP^N,\cI_{W/\P^N}(1))=0$. 
Let $\P^{N-1}$ be a hyperplane on $\P^{N}$. 
Then $W \cap \P^{N-1}$ is non-degenerate in $\P^{N-1}$.
\end{lemma}

\begin{proof}
See \cite[Lemma 8.1]{BL13}. 
\end{proof}

\begin{thm}\label{t g=9 Lin Sec Thm}
Let $X$ be a prime Fano threefold of genus $9$. 
Let $\Sigma_9 \subset \P^{13}$ be the closed embedding 
induced by the complete linear system 
$|\MO_{\Sigma_9}(1)|$ 
for the ample generator $\MO_{\Sigma_9}(1)$ of $\Pic \Sigma_9$. 
Then there exists a $10$-dimensional linear subvariety $L$ of $\P^{13}$ 
such that $X \simeq \Sigma_9 \cap L$. 
\end{thm}

\begin{proof}
Take a dual Mukai bundle $\cQ_X$ on $X$ and let 
$\varphi : X \hookrightarrow \Gr(V_6, 3)$ be the closed immersion 
induced by $\cQ_X$, where $V_6 := H^0(X, \cQ_X)$ 
(Proposition \ref{p Gushel emb X}). 
In what follows, we identify $X$ with its image $\varphi(X)$. 
Fix a general member $S$ of $|-K_X|$. 
By Proposition \ref{p g=9 Sigma factor}, 
there exists a non-degenerate bivector $\sigma \in \wedge^2 V_6$ 
satisfying 
\[
S \subset X \subset \Sigma_9 \subset \P^{13}, 
\]
where $\Sigma_9 := \Gr(V_6, 3, \sigma)$ and $\P^{13} := \P(\Coker m_\sigma)$. 
Take three general hyperplanes $H_1$, $H_2$, $H_3$ of $\bP^{13}$  containing $\la X\ra $. 
We then get the following commutative diagram consisting of closed immersions: 
\[
\begin{tikzcd}
&\Sigma_9 \arrow[r, hook] & \bP^{13}  \\
& \Sigma_9 \cap H_1 \arrow[u, hook] \arrow[r, hook] & H_1 \arrow[u, hook] \\
& \Sigma_9 \cap H_1 \cap H_2 \arrow[u, hook]\arrow[r, hook] & H_1\cap H_2 \arrow[u, hook] \\
& \Sigma_9 \cap H_1 \cap H_2 \cap H_3 \arrow[u, hook]\arrow[r, hook] & H_1\cap H_2 \cap H_3 \arrow[u, hook] \\
X \arrow[r, hook] & \Sigma_9 \cap \langle X \rangle \arrow[u, hook]\arrow[r, hook] & \langle X \rangle \arrow[u, hook] \\
S \arrow[u, hook]\arrow[r, hook]& \Sigma_9 \cap \langle S \rangle \arrow[u, hook]\arrow[r, hook] & \langle S \rangle=:\P^9. \arrow[u, hook] \\ 
\end{tikzcd}
\]
We treat the following two cases separately: 
\[
\text{(I)}\,\, \la X\ra = \la S \ra \qquad \qquad \qquad 
\text{(II)}\,\, \la X\ra \neq  \la S \ra.  
\]

(I) 
Assume $\la X \ra = \la S \ra =\P^9$. 
Recall that we have inclusions 
\[
X \subset \Sigma_9 \subset \Gr(V_6, 3) \subset 
\P(\wedge^3 V_6) (\simeq \P^{19}).
\]
Applying 
Proposition \ref{p schubert smooth}  for $\Sigma_9 \cap \P^9 = \Sigma_9 \cap \la X \ra$, 
we get $\dim ( \Sigma_9 \cap \la X \ra) \leq 3$. 
Hence 
\[
X \subset \Sigma_9 \cap \P^9  \subset 
\Sigma_9 \cap H_1 \cap H_2 \cap H_3 =:Z.
\]
and $\dim Z =3$. 
As $\deg X = 16 = \deg \Sigma_9 = \deg Z$, 
we get an equality $[X]=[Z]$ as algebraic $3$-cycles on $\Sigma_9$. 
Hence $Z$ is irreducible and generically reduced. 
Since $Z$ is Cohen-Macaulay, $Z$ is an integral scheme. 
Therefore, we get a scheme-theoretic equality $X = Z = \Sigma_9 \cap H_1 \cap H_2 \cap H_3$, as required. 

(II) 
Assume $\la X\ra \neq  \la S \ra$. 
In this case, the closed embedding $X \subset \la X \ra =:\P^{10}$ is given by the complete linear system $|-K_X|$. 
Applying 
Proposition \ref{p schubert smooth} for $\Sigma_9 \cap \P^{10} = \Sigma_9 \cap \la X \ra$, 
we get $\dim ( \Sigma_9 \cap \la X \ra) \leq 4$. 
In particular, $\Sigma_9 \cap H_1 \cap H_2$ 
is a $4$-dimensional lci closed subscheme on $\Sigma_9$. 
If $\Sigma_9 \cap H_1 \cap H_2$ is an integral scheme, 
then we get $\Sigma_9 \cap H_1 \cap H_2 \cap H_3 \subsetneq \Sigma_9 \cap H_1 \cap H_2$ by Lemma \ref{l hyperplane nondeg}, 
and hence 
$\Sigma_9 \cap H_1 \cap H_2 \cap H_3$ is a 
pure $3$-dimensional lci projective scheme containing $X$.
This, together with $\deg X = 16 = \deg (\Sigma_9 \cap H_1 \cap H_2 \cap H_3)$, implies the scheme-theoretic equality $X = \Sigma_9 \cap H_1 \cap H_2 \cap H_3$, as required. 

In what follows, we assume that $\Sigma_9 \cap H_1 \cap H_2$ is not an integral scheme. 
As  $\Sigma_9 \cap H_1 \cap H_2$ is lci and of degree $16$, 
Lemma \ref{lemma:Sigma9_Chow} implies an equality 
\[
[\Sigma_9 \cap H_1 \cap H_2] = [Y]+[Y']
\]
as algebraic $4$-cycles, where $Y$ and $Y'$ are  
$4$-dimensional varieties 
with $\deg Y = \deg Y'=8$ (possibly $Y = Y'$). 
After switching $Y$ and $Y'$ if necessary, 
we may assume $X \subset Y$. 
Then $\langle X \rangle \subset \langle Y \rangle \subset H_1 \cap H_2$. 
If $\langle X \rangle \subsetneq \langle Y \rangle$, then 
\[
X  = X \cap \langle X \rangle \subset Y \cap\langle X \rangle \subsetneq Y.
\]
In this case, we obtain $X \subset Y \cap H$ and $\dim(Y \cap H) =3$ for a general hyperplane $H$ containing $X$, 
which leads to the following contradiction: 
$16 = \deg X \leq \deg (Y \cap H) = \deg Y =8$. 
Therefore, it holds that 
 $\P^{10} = \langle X \rangle = \langle Y \rangle$.  
By 
Proposition \ref{p schubert smooth} applied for $Y \subset \Gr(V_6, 3) \cap \P^{10}$, 
we see that $Y$ is a smooth fourfold of degree $8$, 
contradicting Lemma~\ref{lemma:not_in_almost_minimal_sm}.
\end{proof}
\section{Genus 10}

\subsection{The $G_2$-homogeneous variety $\Sigma_{10}$}

Let $V_7$ be a $7$-dimensional vector space, and 
let $\Gr(V_7,2)$ be the Grassmannian variety of $2$-dimensional quotients of $V_7$.
The variety $\Gr(V_7,2)$ is embedded into $\bP(\wedge^2 V_7)$ via the Pl\"ucker embedding.
On this Grassmannian variety, we have the universal exact sequence:
\[
0 \to \cK \to V_7 \otimes \MO_{\Gr(V_7,2)} \to \cQ \to 0.
\]

\begin{definition}
Let $V_7$ be a $7$-dimensional vector space.
\begin{enumerate}
 \item 
 An element $\tau \in \wedge^3 V_7$ is called a \emph{trivector}.
 \item 
 For an element $f \in V_7^\vee$, we have the following induced linear map: 
\begin{align*}
C_f \colon \wedge^3 V_7 &\to \wedge ^2 V_7\\
x \wedge y \wedge z 
&\mapsto f(x)  y \wedge z  -f(y) x  \wedge z +f(z) x \wedge y,  
\end{align*} 
which is the dual of 
$f\wedge : \wedge^2(V_7^\vee) \to \wedge^3(V_7^\vee), 
g \wedge h \mapsto f \wedge g \wedge h$. 
\item Fix a trivector $\tau \in \wedge^3 V_7$. 
By using $C_f$, we have a linear map
\[
M_\tau \colon V_7^\vee \to \wedge^2 V_7, \quad 
f \mapsto C_f(\tau).
\]
\item
$C_f(\tau) \in \wedge^2 V_7$ defines an alternating map $V_7^\vee \to V_7$.
The \emph{rank} of $C_f(\tau)$ is, by definition, the rank of this alternating map $V_7^\vee \to V_7$. 
Since the rank of an alternating map $V_7^\vee \to V_7$ is even, we have $\rank C_f(\tau)  \in \{ 0, 2, 4, 6\}$.
\item
For $i \in \{0, 1, 2\}$, we define $D_{2i}(\tau)$ by 
\[
D_{2i}(\tau) := \{[f] \in \bP(V_7) \mid \rank C_f(\tau) \leq 2i\}.
\]
\end{enumerate}
\end{definition}

\begin{remark}
Let $V_7$ be a $7$-dimensional linear space and 
let $\tau \in \wedge^3 V_7$ be a trivector. 
The scheme-theoretic structure on $D_{2i}(\tau)$ is given as follows. 
We have the following composite homomorphism: 
\[
\MO_{\bP(V_7)} (-1) \xrightarrow{u} V_7^\vee \otimes \MO_{\bP(V_7)} \xrightarrow{M_\tau \otimes \id}\wedge^2 V_7 \otimes\MO_{\bP(V_7)}, 
\]
where $u$ denotes the universal homomorphism on $\P(V_7)$. 
This composition corresponds to an alternating homomorphism
\[
V_7^\vee \otimes \MO_{\bP(V_7)} (-1) \to  V_7 \otimes\MO_{\bP(V_7)}.
\]
Then $D_{2i}(\tau)$ is the degeneracy locus where this alternating homomorphism is of rank $\leq 2i$. 
Its scheme-theoretic structure is defined 
by the Pfaffians of the
$(2i+2)\times(2i+2)$ principal submatrices of a local matrix representing this alternating homomorphism. 
\end{remark}

\begin{proposition}\label{proposition:trivector_nondegenerate}
Let $V_7$ be a $7$-dimensional linear space and 
let $\tau \in \wedge^3 V_7$ be a nonzero trivector.
Then either $D_4(\tau)  = \P(V_7)$, or 
$D_4(\tau)$  is a quadric hypersurface in $\bP(V_7)$. 
Moreover, the following are equivalent:
\begin{enumerate}
 \item $D_4(\tau)$ is a smooth quadric hypersurface in $\bP(V_7)$.
 \item $D_2(\tau) =\emptyset$.
 \item The image of the map $M_\tau \colon V_7^\vee \to \wedge^2 V_7$ contains no nonzero decomposable vector $v \wedge w$ with $v, w \in V_7$.
\item There exists a basis 
$\{e_{0},e_{1},e_{2},e_{3},e_{-1},e_{-2},e_{-3}\}$ of $V_7$ such that 
\[
\tau = e_1\wedge e_2\wedge e_3+e_{-1}\wedge e_{-2}\wedge e_{-3}+e_0\wedge e_1\wedge e_{-1}+e_0\wedge e_2\wedge e_{-2}+e_0\wedge e_3\wedge e_{-3}. 
\]
\end{enumerate}
\end{proposition}

\begin{proof}
By the classification of trivectors in dimension $7$ \cite[Theorem 2.1]{CH88}, 
the natural $\GL(V_7)$-action on $\wedge^3 V_7 \setminus \{0\}$ has exactly $9$ orbits, 
whose representatives are given in the following table with respect to a suitable basis 
$\{f_1,\dots,f_7\}$:
\[
\begin{tabular}{l|l}
$\tau \in \wedge^3 V_7$ & Defining equation of $D_4(\tau)$ \\ \hline
 $123+456+147+257+367$ & $x_1x_4 + x_2x_5 + x_3x_6 - x_7^2$\\
 $123+145+167$ & $x_1^2 $\\
 $146+157+245+367$ & $x_4x_6+x_5x_7$\\
 $152+174+163+243$ & $x_1^2 $\\
 $123+456+147$ & $x_1x_4 $\\
 $162+243+135$ & $0$\\
 $123+456$ & $0$\\
 $123+145$ & $0$\\
 $123$ & $0$
\end{tabular}
\]
Here $ijk$ stands for $f_i\wedge f_j \wedge f_k$. 
Although 
the assertions follow from direct computations by setting 
$e_0 := f_7, e_1 :=f_1, e_2 :=f_2 , e_3 := f_3, e_{-1} :=f_4, 
e_{-2} :=f_5, e_{-3} :=f_6$, 
we will provide some details below.

\medskip

\underline{$(1) \Leftrightarrow (4)$}: 
In order to prove this equivalence, 
it suffices to show that the defining equation of $D_4(\tau)$ 
is given as in the above table. 
We check this only for the case when $\tau=123+456+147+257+367$. 
For $f:=x_1f_1^\vee+\cdots+x_7f_7^\vee$, we have
\[
C_f(\tau)
=
x_1\,23-x_2\,13+x_3\,12
+x_4\,56-x_5\,46+x_6\,45
\]
\[
+x_1\,47-x_4\,17+x_7\,14
+x_2\,57-x_5\,27+x_7\,25+x_3\,67-x_6\,37+x_7\,36. 
\]
With respect to the dual basis $\{f_1^\vee, ..., f_7^\vee\}$, 
the  corresponding alternating matrix is given by 
\[
A_f(\tau) :=
\begin{pmatrix}
0 & x_3 & -x_2 & x_7 & 0 & 0 & -x_4 \\
-x_3 & 0 & x_1 & 0 & x_7 & 0 & -x_5 \\
x_2 & -x_1 & 0 & 0 & 0 & x_7 & -x_6 \\
-x_7 & 0 & 0 & 0 & x_6 & -x_5 & x_1 \\
0 & -x_7 & 0 & -x_6 & 0 & x_4 & x_2 \\
0 & 0 & -x_7 & x_5 & -x_4 & 0 & x_3 \\
x_4 & x_5 & x_6 & -x_1 & -x_2 & -x_3 & 0
\end{pmatrix}.
\]
For $q:=x_1x_4+x_2x_5+x_3x_6-x_7^2$, the $6\times 6$ principal Pfaffians of the corresponding alternating matrix are $x_1q, -x_2q, x_3q, -x_4q, x_5q, -x_6q, x_7q$. 
Therefore, $D_4(\tau)$ is defined by $q=0$. 
This completes the proof of the equivalence $(1) \Leftrightarrow (4)$. 

\medskip

\underline{$(4) \Rightarrow (2)$}: 
The locus $D_2(\tau)$ is the common zero locus of 
the $4\times 4$ principal Pfaffians of the above alternating matrix $A_f(\tau)$. 
The Pfaffian of the $4\times 4$ principal submatrix of $A_f(\tau)$ corresponding to the indices $\{2,3,4,7\}$ 
is $x_1^2$. 
Similarly, the following are the Pfaffians of some $4 \times 4$ principal submatrices of $A_f(\tau)$: 
\[
x_1^2,\quad
-x_2^2,\quad
x_3^2,\quad
-x_4^2,\quad
x_5^2,\quad
-x_6^2,\quad 
x_3x_6-x_7^2.
\]
Therefore, $D_2(\tau) = \emptyset$.

\medskip

\underline{$(2) \Rightarrow (3)$}: 
This implication follows from the fact that a nonzero bivector $\sigma \in \wedge^2 V_7$ is decomposable if and only if $\rank \sigma =2$.

\medskip

\underline{$(3) \Rightarrow (4)$}: 
Assume that (4) does not hold. 
By the above table, $\tau$ is $\GL(V_7)$-equivalent to one of the remaining eight representatives. 
It is enough to show that (3) does not hold for each of these representatives.
For each of the remaining representatives, the following table gives a nonzero $f\in V_7^\vee$ such that $C_f(\tau)=M_\tau(f)$ is a nonzero decomposable bivector:
\[
\begin{tabular}{l|l|l}
$\tau$ & $f$ & $C_f(\tau)=M_\tau(f)$ \\ \hline
$123+145+167$ & $f_2^\vee$ & $-13$ \\
$146+157+245+367$ & $f_2^\vee$ & $45$ \\
$152+174+163+243$ & $f_5^\vee$ & $-12$ \\
$123+456+147$ & $f_2^\vee$ & $-13$ \\
$162+243+135$ & $f_4^\vee$ & $-23$ \\
$123+456$ & $f_1^\vee$ & $23$ \\
$123+145$ & $f_2^\vee$ & $-13$ \\
$123$ & $f_1^\vee$ & $23$
\end{tabular}
\]
This completes the proof 
of the implication $(3)\Rightarrow(4)$.
\end{proof}

\begin{rem}
Note that each of the conditions (1)-(4) in  Proposition \ref{proposition:trivector_nondegenerate} is equivalent to 
(5) below (cf.\ \cite[Definition 4.1]{BKM26b}). 
\begin{enumerate}
\setcounter{enumi}{4}
\item[(5)] $\tau$ belongs to the open $\GL(V_7)$-orbit 
of $\wedge^3 V_7 \setminus \{0\}$. 
\end{enumerate}
Indeed, it suffices to show that 
the $\GL(V_7)$-orbit $\Orb_{\GL(V_7)}(\tau)$ of 
$\tau=123+456+147+257+367$ is open, which follows from 
\[\dim \Orb_{\GL(V_7)}(\tau)
=
\dim \GL(V_7) - \dim \Stab_{\GL(V_7)}(\tau)
\overset{(\star)}{=}49 -14  
\]
\[
=
35  = \binom{7}{3} = 
\dim (\wedge^3 V_7 \setminus \{0\}), 
\]
where $(\star)$ holds by 
$\Stab_{\GL(V_7)}(\tau)^{\circ}_{\red}  \simeq G_2$ 
\cite[Theorem 2.1]{CH88} 
and $\dim G_2 =14$. 
\end{rem}

\begin{definition}
Let $V_7$ be a $7$-dimensional linear space. 
 A nonzero trivector $\tau \in \wedge^3 V_7$ is called \emph{non-degenerate} if it satisfies the equivalent conditions in Proposition \ref{proposition:trivector_nondegenerate}.
\end{definition}

\begin{definition}
Let $V_7$ be a $7$-dimensional vector space. 
For a trivector $\tau \in \wedge^3 V_7$, 
$\Gr(V_7, 2,\tau)$ denotes the zero locus of the corresponding 
section in $H^0(\Gr(V_7,2), \wedge^4 \cK^\vee)$ via the isomorphism   
$\wedge^3 V_7 \simeq \wedge^4 V^{\vee}_7 = 
H^0(\Gr(V_7, 2), \wedge^4 \cK^\vee)$. 
We set $\Sigma_{10}:= \Gr(V_7,2,\tau)$ for a non-degenerate trivector $\tau \in \wedge^3 V_7$. 
\end{definition}

The following lemma shows that 
$\Gr(V_7,2,\tau)$ parametrises 
$2$-dimensional quotients $V_7 \to Q_2$ such that the composition $V_7^\vee \xrightarrow{M_\tau} \wedge^2 V_7 \to \wedge^2 Q_2$ is zero. 
This is 
dual to the notation of  $\tau$-singular subspaces introduced in  \cite[(1.1) in Section 1]{Asc87}.

\begin{lem}\label{lem:Gr_tau_explicit_description}
Let $V_7$ be a $7$-dimensional vector space. 
Take a trivector $\tau \in \wedge^3 V_7$ and 
let $\tau' \in \wedge^4 V_7^\vee$ be the element 
corresponding to $\tau \in \wedge^3 V_7$ via the isomorphism $\wedge^3 V_7 \simeq \wedge^4 V_7^\vee$. 
Then the following hold:
\begin{align*}
\Gr(V_7,2,\tau)
&=
\left\{
[q:V_7\twoheadrightarrow Q_2]\in \Gr(V_7,2)
\ \middle|\
\tau'(\wedge^4 \Ker q)=0
\right\}\\
&=
\left\{
[q:V_7\twoheadrightarrow Q_2]\in \Gr(V_7,2)
\ \middle|\
V_7^\vee \xrightarrow{M_\tau} \wedge^2 V_7
\xrightarrow{\wedge^2 q} \wedge^2 Q_2
{\rm \,\,\,is\,\,\,zero}
\right\}.
\end{align*}
\end{lem}

\begin{proof}
The first equality holds by definition. 
Let us show the second equality. 
Take a closed point 
$[q:V_7\twoheadrightarrow Q_2]\in \Gr(V_7,2)$. 
Set $K := \Ker q$ and fix a splitting $V_7=K\oplus Q_2$. 
Then we have the following decomposition: 
\begin{align*}
\wedge^3 V_7
&=
\wedge^3K
\oplus
(\wedge^2K\otimes Q_2)
\oplus
(K\otimes \wedge^2Q_2)\\
\tau&=\tau_0 + \tau_1 + \tau_2
\end{align*}
where $\tau_0\in \wedge^3K, 
\tau_1\in \wedge^2K\otimes Q_2, \tau_2\in K\otimes \wedge^2Q_2.$ 
For $k_1,\ldots,k_4\in K$, 
we get $\tau_0\wedge k_1\wedge\cdots\wedge k_4=0$ 
and
$\tau_1\wedge k_1\wedge\cdots\wedge k_4=0$, because they contain more than $5$ factors from $K$. 
Since the natural pairing $K\otimes \wedge^4K \to \wedge^5K$ 
is perfect, we have $\tau_2=0 \Leftrightarrow \tau'(\wedge^4K)=0$.

For a $k$-linear basis $K = \bigoplus_\alpha k \kappa_\alpha$, 
we can write $\tau_2=\sum_\alpha \kappa_\alpha\wedge \eta_\alpha
\in K\otimes \wedge^2Q_2$ for some $\eta_\alpha \in \wedge^2Q_2$. 
For every $f \in V_7^\vee$, we get  
\[
((\wedge^2 q) \circ M_{\tau}) (f) 
=
(\wedge^2 q)(C_f(\tau))  
= (\wedge^2 q)(C_f(\tau_0) +C_f(\tau_1) +C_f(\tau_2))
\]
\[
=  
(\wedge^2 q)(C_f(\tau_2))  
= \sum_\alpha f(\kappa_\alpha) \eta_\alpha. 
\]
Therefore, 
we get the equivalence 
$(\wedge^2 q) \circ M_{\tau} =0 \Leftrightarrow \tau_2=0$. 
\end{proof}

\begin{proposition}\label{proposition:Sigma10_schubert}
Let $V_7$ be a $7$-dimensional linear space and 
let $\Gr(V_7, 2) \subset \P(\wedge^2 V_7)$ denote the Pl\"{u}cker embedding. 
Take a trivector $\tau \in \wedge^3 V_7$. 
Then the following hold:
\begin{enumerate}
 \item \label{proposition:Sigma10_schubert1} In $\bP(\wedge^2 V_7)$, 
 the scheme-theoretic equality $\Gr(V_7,2, \tau) = \Gr(V_7,2) \cap \bP(\Coker M_\tau)$ holds.
 \item \label{proposition:Sigma10_schubert2} $\tau$ is non-degenerate if and only if $\Gr(V_7,2, \tau)$ is not contained in any Schubert divisor 
 of  $\Gr(V_7,2)$.
\item \label{proposition:Sigma10_schubert3} If $\tau$ is non-degenerate, then 
$\bP(\Coker M_\tau) \simeq \P^{13}$ and $\Gr(V_7,2,\tau)$ is not contained in any hyperplane on $\bP(\Coker M_\tau)$.
\end{enumerate}
\end{proposition}

\begin{proof}
The assertion (\ref{proposition:Sigma10_schubert1}) follows from 
Lemma~\ref{lem:Gr_tau_explicit_description} and \cite[Lemma 2.1]{BKM26b}.

Let us prove (\ref{proposition:Sigma10_schubert2}). 
Take a $2$-dimensional subspace 
$U_2=ku_1\oplus ku_2\subset V_7.$ 
Then the Schubert divisor corresponding to $U_2$ is given by 
\[
\Sigma_1(U_2)
:=
\left\{
[q:V_7\twoheadrightarrow Q_2]\in \Gr(V_7,2)
\ \middle|\ 
\dim \Im(U_2\to Q_2)<2
\right\}.
\]
Hence, for $[q:V_7\twoheadrightarrow Q_2]\in \Gr(V_7,2)$, we have
\[
[q:V_7\twoheadrightarrow Q_2]\in \Sigma_1(U_2)
\Longleftrightarrow 
(\wedge^2 q)(u_1\wedge u_2)=0
\text{ in }\wedge^2Q_2.
\]
Therefore,
\[
\begin{aligned}
\Gr(V_7,2,\tau)\subset \Sigma_1(U_2)
& \overset{\text{(1)}}{\Longleftrightarrow}
\bP(\Coker M_\tau)
\subset
\left\{
[\wedge^2 q]\in \bP(\wedge^2V_7)
\ \middle|\ 
(\wedge^2q)(u_1\wedge u_2)=0
\right\}
\\
&\Longleftrightarrow
u_1\wedge u_2
\in
\Ker(\wedge^2V_7\to \Coker M_\tau)
\\
&\Longleftrightarrow
u_1\wedge u_2
\in
\Im(M_\tau).
\end{aligned}
\]
Thus $\Gr(V_7,2,\tau)$ is contained in a Schubert divisor if and only if 
$\Im(M_\tau)$ contains a nonzero decomposable vector. 
By Proposition~\ref{proposition:trivector_nondegenerate}, this is equivalent to the condition that $\tau$ is degenerate. 
Thus (\ref{proposition:Sigma10_schubert2}) holds.

Let us show (\ref{proposition:Sigma10_schubert3}). Since $M_\tau:V_7^\vee\to \wedge^2V_7$ is injective 
by Proposition~\ref{proposition:trivector_nondegenerate}(2), we get
\[
\dim(\Coker M_\tau)
=
\dim(\wedge^2V_7)-\dim V_7^\vee
=
\binom{7}{2}-7
=
21-7
=
14.
\]
Thus $\bP(\Coker M_\tau)\simeq \P^{13}$. 
Take a $k$-linear basis $\{e_0,e_1,e_2,e_3,e_{-1},e_{-2},e_{-3}\}$ 
of $V_7$ such that 
\[
\tau =
e_1\wedge e_2\wedge e_3
+
e_{-1}\wedge e_{-2}\wedge e_{-3}
+
e_0\wedge e_1\wedge e_{-1}
+
e_0\wedge e_2\wedge e_{-2}
+
e_0\wedge e_3\wedge e_{-3}.
\]
Let $\{e_i^\vee\}$ be the dual basis of $\{e_i\}$.

We can directly check that 
$\Gr(V_7,2,\tau)$ contains the following $16$ points:
\[
\begin{gathered}
e_1^\vee\wedge e_{-2}^\vee,\quad
e_1^\vee\wedge e_{-3}^\vee,\quad
e_2^\vee\wedge e_{-3}^\vee,\quad
e_2^\vee\wedge e_{-1}^\vee,\quad
e_3^\vee\wedge e_{-1}^\vee,\quad
e_3^\vee\wedge e_{-2}^\vee,
\\
(e_1^\vee+e_2^\vee)\wedge(e_{-1}^\vee-e_{-2}^\vee),\quad
(e_2^\vee+e_3^\vee)\wedge(e_{-2}^\vee-e_{-3}^\vee),\quad
(e_3^\vee+e_1^\vee)\wedge(e_{-3}^\vee-e_{-1}^\vee),
\\
(e_0^\vee+e_1^\vee+e_{-1}^\vee)\wedge(e_2^\vee+e_{-3}^\vee),\quad
(e_0^\vee+e_1^\vee+e_{-1}^\vee)\wedge(e_3^\vee-e_{-2}^\vee),
\\
(e_0^\vee+e_2^\vee+e_{-2}^\vee)\wedge(e_3^\vee+e_{-1}^\vee),\quad
(e_0^\vee+e_2^\vee+e_{-2}^\vee)\wedge(e_1^\vee-e_{-3}^\vee),
\\
(e_0^\vee+e_3^\vee+e_{-3}^\vee)\wedge(e_1^\vee+e_{-2}^\vee),\quad
(e_0^\vee+e_3^\vee+e_{-3}^\vee)\wedge(e_2^\vee-e_{-1}^\vee),
\\
(e_0^\vee+e_1^\vee+e_2^\vee+e_3^\vee+e_{-1}^\vee)
\wedge
(e_3^\vee+e_{-1}^\vee-e_{-2}^\vee)
\end{gathered}
\]
by using 
\[
\begin{aligned}
M_\tau(e_0^\vee)
&=
e_1\wedge e_{-1}
+
e_2\wedge e_{-2}
+
e_3\wedge e_{-3}.
\end{aligned}
\]
\[
\begin{alignedat}{2}
M_\tau(e_1^\vee)
&=
e_2\wedge e_3
-
e_0\wedge e_{-1},
\qquad&
M_\tau(e_{-1}^\vee)
&=
e_{-2}\wedge e_{-3}
+
e_0\wedge e_1,
\\
M_\tau(e_2^\vee)
&=
-
e_1\wedge e_3
-
e_0\wedge e_{-2},
\qquad&
M_\tau(e_{-2}^\vee)
&=
-
e_{-1}\wedge e_{-3}
+
e_0\wedge e_2,
\\
M_\tau(e_3^\vee)
&=
e_1\wedge e_2
-
e_0\wedge e_{-3},
\qquad&
M_\tau(e_{-3}^\vee)
&=
e_{-1}\wedge e_{-2}
+
e_0\wedge e_3.
\end{alignedat}
\]
For example, we have 
\[
(e_1^\vee\wedge e_{-2}^\vee)\circ M_\tau(e_0^\vee) = 
(e_1^\vee\wedge e_{-2}^\vee) (e_1\wedge e_{-1}
+
e_2\wedge e_{-2}
+
e_3\wedge e_{-3}) = 0. 
\]
Similarly, we get $(e_1^\vee\wedge e_{-2}^\vee) \circ M_\tau(e_i^\vee)$ for every $i \in \{ 0, \pm 1, \pm 2, \pm 3\}$, 
which implies  
\[
[e_1^\vee \oplus e_{-2}^\vee : V_7 \twoheadrightarrow k\oplus k]
\in \{
[q:V_7\twoheadrightarrow Q_2]\in \Gr(V_7,2)
\,|\, 
(\wedge^2 q) \circ M_\tau=0\} 
= \Gr(V_7,2,\tau). 
\]

Then it suffices to show that 
the linear subspace $U$ generated by the above  $16$ elements 
(considered as elements of $\wedge^2 V_7^\vee$) 
has dimension $\geq 14$. For $W := k e_1^\vee \oplus  k e_2^\vee \oplus k e_3^\vee \oplus k e_{-1}^\vee \oplus 
k e_{-2}^\vee \oplus k e_{-3}^\vee  \subset V_7^\vee$, 
it holds that $\wedge^2 V_7^\vee = (\wedge^2 W) \oplus ( e^\vee_0 \wedge W)$. 
For the projection $\pi : \wedge^2 V_7^\vee \to e^\vee_0 \wedge W \xrightarrow{\simeq} W$, 
we have an exact sequence 
\[
0 \to U \cap (\wedge^2 W) \to U \to \pi(U) \to 0.
\]
Thus it is enough to prove 
$\dim (U \cap (\wedge^2 W)) \geq 8$ 
and $\dim \pi(U) \geq 6$. 
The former inequality $\dim (U \cap (\wedge^2 W)) \geq 8$ follows from the fact that 
the first $8$ elements are contained in $U \cap (\wedge^2 W)$ 
and they are linearly independent. 
The remaining inequality $\dim \pi(U) \geq 6$ follows from 
the fact that the $\pi$-images of the latter $7$ elements are given by 
\[
e_2^\vee+e_{-3}^\vee,\quad
e_3^\vee-e_{-2}^\vee,\quad
e_3^\vee+e_{-1}^\vee,\quad
e_1^\vee-e_{-3}^\vee,\quad
e_1^\vee+e_{-2}^\vee,\quad
e_2^\vee-e_{-1}^\vee,\quad
e_3^\vee+e_{-1}^\vee-e_{-2}^\vee. 
\]
Thus (\ref{proposition:Sigma10_schubert3}) holds. 
\end{proof}

\begin{prop}
Let $V_7$ be a $7$-dimensional linear space and 
let $\Gr(V_7, 2) \subset \P(\wedge^2 V_7)$ denote the Pl\"{u}cker embedding. 
Take a non-degenerate trivector $\tau \in \wedge^3 V_7$. 
Then $\Sigma_{10} = \Gr(V_7,2,\tau) \simeq G_2/P$, 
where $G_2$ denotes a connected semisimple algebraic group of type $G_2$
and $P$ is the maximal parabolic subgroup corresponding to the long  root of the root system of $G_2$. 
In particular, $\Sigma_{10}$ is a smooth Fano $5$-fold of index $3$ and degree $18$.
\end{prop}

\begin{proof}
By \cite[(2.12), (3.4) and (5.9)]{Asc87}, $\Gr(V_7,2,\tau)$ is isomorphic to $G_2/P$.
In particular, it is smooth. 
The assertions on the index and degree follows from the same argument to characteristic zero case as in \cite[Lemma 4.2]{BKM26b}. 
\end{proof}

\subsection{Linear section theorem for $g=10$}

\begin{proposition}\label{p g=10 Sigma factor}
Let $X$ be a prime Fano threefold of genus $10$. 
Take a dual Mukai bundle $\cQ_X$ of type $(2,5)$ and let 
$\varphi : X \hookrightarrow \Gr(V_7,2)$ be the closed immersion induced by $\cQ_X$, 
where $V_7 \coloneqq H^0(X,\cQ_X)$ 
(cf.\ Proposition \ref{p Gushel emb X}). 
Then the following hold:
\begin{enumerate}
\item 
For a trivector $\tau \in \wedge^3 V_7$ satisfying $\varphi(X) \subset \Gr(V_7,2,\tau)$, 
$\tau$ is nonzero if and only if $\tau$ non-degenerate. 
\item 
There exists a non-degenerate trivector $\tau \in \wedge^3 V_7$ such that 
$\varphi(X) \subset \Gr(V_7,2,\tau)$.
Moreover, such a non-degerante trivector $\tau$ is unique up to scalar.
\end{enumerate}
\end{proposition}

\begin{proof}
Take a smooth general member $S$ of $|-K_X|$. 
Set $H := -K_X|_S$. 
We denote by $\cK_X$ and $\cK_S$ the pullbacks of the universal kernel bundle 
$\cK$ on $\Gr(V_7,2)$ to $X$ and $S$, respectively.

Let us show (1). 
A non-degenerate trivector is nonzero by definition. 
If $\tau$ is degenerate, then $\Gr(V_7,2,\tau)$ is contained in a Schubert divisor
of $\Gr(V_7,2)$ by Proposition~\ref{proposition:Sigma10_schubert}. 
This, together with  $X \subset \Gr(V_7,2,\tau)$, 
implies that $X$ is contained in a Schubert divisor. 
This contradicts Proposition~\ref{p Sch avoid X}. 
Thus $\tau$ is non-degenerate. 
Thus (1) holds.

Let us show (2). 
After fixing a trivialisation of $\det V_7$, we identify
$\wedge^3 V_7 \simeq \wedge^4 V_7^\vee$. 
We have the following restriction maps:
\[
\begin{tikzcd}[column sep=large]
\wedge^3V_7 
\simeq
\wedge^4V_7^\vee
=
H^0\left(\Gr(V_7,2),\wedge^4\cK^\vee\right)
\arrow[r, "\rho^{\Gr(V_7,2)}_X"]
\arrow[rr, bend right=20, "\rho^{\Gr(V_7,2)}_S"]
&
H^0\left(X,\wedge^4\cK_X^\vee\right)
\arrow[r, "\rho^{X}_S", hook]
&
H^0\left(S,\wedge^4\cK_S^\vee\right).  
\end{tikzcd}
\]

We now show that $\rho^X_S$ is injective. 
Since $\wedge^4\cK_X^\vee \simeq \cK_X(H)$ and $S \sim H$, 
we have $\wedge^4\cK_X^\vee(-S) \simeq \cK_X$. 
Hence it is enough to show $H^0(X,\cK_X)=0$. 
This follows from the exact sequence $0 \to \cK_X \to V_7\otimes \MO_X \xrightarrow{\pi} \cQ_X \to 0$ and the fact that 
$H^0(\pi)$ is an isomorphism.  
This completes the proof of the injectivity of $\rho^X_S$. 
Therefore, 
\[
K := \Ker(\rho^{\Gr(V_7,2)}_X)
=
\Ker(\rho^{\Gr(V_7,2)}_S).
\]

It is enough to prove
$\dim K=1$. 
It holds that $\dim \wedge^3 V_7 = \binom{7}{3}=35$. 
We have $H^{>0}(S,\wedge^4\cK_S^\vee) 
\simeq H^{>0}(S, \cK_S(H)) =0$ by Proposition \ref{p BKM4.4 vanishing}, 
which is applicable because $\cK_S=\cK_X|_S$ and 
$\cK_X$ is a Mukai bundle \cite[Lemma 6.2]{Tan-bdl}. 
By the Riemann-Roch theorem,  
we get
$h^0\left(S,\wedge^4\cK_S^\vee\right)
=
\chi\left(S,\wedge^4\cK_S^\vee\right)
=
34$.
Hence $\dim K = \dim \Ker(\rho^{\Gr(V_7,2)}_S)\geq 35-34=1$.

It suffices to show $\dim K<2$. 
Suppose that $\dim K\geq 2$, i.e., $\dim \P(K^\vee)\geq 1$. 
The inclusion $K\subset \wedge^3V_7$ induces a closed embedding 
$\P(K^\vee)\subset \P(\wedge^3V_7^\vee).$ 
By Proposition~\ref{proposition:trivector_nondegenerate}, 
the degenerate trivectors form a hypersurface
$D
\subset \P(\wedge^3V_7^\vee)$. 
Indeed, this hypersurface is defined by the discriminant of the quadric $D_4(\tau)$. 
Since $k$ is algebraically closed and $\P(K^\vee)$ has positive dimension, 
there exists a closed point
\[
[\tau]\in D\cap \P(K^\vee).
\]
Then $\tau$ is a nonzero degenerate trivector satisfying
$X\subset \Gr(V_7,2,\tau)$, 
contradicting $(1)$. 
Thus (2) holds.
\end{proof}

\begin{thm}\label{t g=10 Lin Sec Thm}
Let $X$ be a prime Fano threefold of genus $10$. 
Let $\Sigma_{10} \subset \P^{13}$ be the closed embedding 
induced by the complete linear system 
$|\MO_{\Sigma_{10}}(1)|$ 
for the ample generator $\MO_{\Sigma_{10}}(1)$ of $\Pic \Sigma_{10}$. 
Then there exists a $11$-dimensional linear subvariety $L$ of $\P^{13}$ 
such that $X \simeq \Sigma_{10} \cap L$. 
\end{thm}

\begin{proof}
Take a dual Mukai bundle $\cQ_X$ on $X$ and let 
$\varphi : X \hookrightarrow \Gr(V_7, 2)$ be the closed immersion 
induced by $\cQ_X$, where $V_7 := H^0(X, \cQ_X)$ 
(Proposition \ref{p Gushel emb X}). 
In what follows, we identify $X$ with its image $\varphi(X)$. 
By Proposition \ref{p g=10 Sigma factor}, 
there exists a non-degenerate trivector $\tau \in \wedge^3 V_7$ 
satisfying 
\[
X \subset \Sigma_{10} \subset \P^{13}, 
\]
where $\Sigma_{10} := \Gr(V_7, 2, \tau)$ and $\P^{13} := \P(\Coker M_\sigma)$. 
Note that $\dim \la X \ra \leq h^0(X, -K_X)-1 =11$.

Take a hyperplane $H$ on $\P^{13}$ containing $\la X \ra$. 
Then $\Sigma_{10} \cap H$ is an integral scheme, because 
$\Sigma_{10} \cap H$ is an ample generator of $\Pic \Sigma_{10} (\simeq \Z)$. 
Moreover, we have $\la \Sigma_{10} \cap H \ra =H (\simeq \P^{12})$ by 
Lemma \ref{l hyperplane nondeg}. 
By $\dim \la X \ra \leq 11 < 12 = \dim H$, 
we can find a hyperplane $H'$ on $\P^{13}$ satisfying 
$\la X \ra \subset H \cap H' \subsetneq H$. 
Set $L := H \cap H'$. 
Then $\Sigma_{10} \cap L$ is pure $3$-dimensional lci projective scheme such that 
$X \subset \Sigma_{10} \cap L$ and 
$\deg X = 18 = \deg \Sigma_{10}= \deg (\Sigma_{10} \cap L)$. 
Hence we get an equality $[X] = [\Sigma_{10} \cap L]$ as algebraic $3$-cycles on $\Sigma_9$, which implies the scheme-theoretic equality  $X = \Sigma_{10} \cap L$. 
\end{proof}
\section{Genus $12$}\label{s g=12}

\subsection{Nets of alternating forms and the triple Veronese surface}

We briefly recall some definitions and results from \cite[Appendix B]{BKM26b}.

\begin{definition}\label{d g=12 non-deg}
Let $W_7$ be a $7$-dimensional vector space.
\begin{enumerate}
 \item An \emph{alternating form} on $W_7$ is a surjective linear map $\sigma \colon \wedge^2 W_7 \to k$. 
 \item An alternating form $\sigma \colon \wedge^2 W_7 \to k$ induces the alternating map
\[
\begin{aligned}
\varphi_{\sigma}\colon W_7 &\longrightarrow W_7^\vee,\\
w &\longmapsto \sigma(w,-).
\end{aligned}
\]
 \item The \emph{rank} of an alternating form $\sigma \colon \wedge^2 W_7 \to k$ is defined as the rank of $\varphi_{\sigma}$.
 \item A \emph{net of alternating forms} on $W_7$ is a surjective linear map $\nu \colon \wedge^2 W_7 \to N_3$ to a three-dimensional vector space $N_3$.
 \item A net of alternating forms $\nu \colon \wedge^2 W_7 \to N_3$ is called \emph{non-degenerate} if, for any linear surjective map $f \colon N_3 \to k$, the rank of the alternating form $f\circ \nu \colon \wedge^2 W_7 \to k$ is $6$.
\end{enumerate}
\end{definition}

Thus alternating forms $\sigma \colon \wedge^2 W_7 \to k$ are parametrised by $\bP(\wedge^2W_7)$, and nets of alternating forms $\nu \colon \wedge^2 W_7 \to N_3$ correspond to projective planes $\bP(N_3) \subset \bP(\wedge^2 W_7)$. 
By definition, a net of alternating forms $\nu \colon \wedge^2 W_7 \to N_3$  is non-degenerate if and only if $\bP(N_3) \subset \bP(\wedge^2 W_7)$ does not intersect the rank $\leq 4$ locus of the space of alternating forms $\bP(\wedge^2W_7)$. 
The following result shows that planes are 
the largest projective linear subspaces of
$\bP(\wedge^2W_7)$ that can be disjoint from the rank $\leq 4$ locus.

\begin{proposition}\label{proposition:codim_rk_4_locus}
Let $W_7$ be a $7$-dimensional vector space  and 
let $\Delta \subset \bP(\wedge^2W_7) $ be the locus of alternating forms of rank at most $4$.
Then 
$\Delta$ is of codimension three in $\bP(\wedge^2W_7)$. 
In particular, every three-dimensional linear subvariety of $\bP(\wedge^2 W_7)$ intersects with $\Delta$. 
\end{proposition}

Although this result is probably well known, we include a proof for the sake of completeness. 

\begin{proof}
On the Grassmannian variety $\Gr(W_7,4)$,  we have the universal surjection
$W_7 \otimes \cO_{\Gr(W_7,4)} \to \cQ' \to 0$, 
which induces another surjection $\wedge^2 W_7 \otimes \cO_{\Gr(W_7,4)} \to \wedge^2 \cQ' \to 0$. 
We then obtain the following morphism: 
\[
\begin{aligned}
\pi\colon \bP_{\Gr(W_7,4)}(\wedge^2\cQ')
&\longrightarrow \bP(\wedge^2W_7),\\
\bigl([q\colon W_7\twoheadrightarrow Q_4],
[\tau\colon \wedge^2Q_4\twoheadrightarrow k]\bigr)
&\longmapsto
\bigl[\tau\circ\wedge^2q
\colon \wedge^2W_7
\overset{\wedge^2 q}{\twoheadrightarrow} \wedge^2 Q_2 \overset{\tau}{\twoheadrightarrow} k\bigr], 
\end{aligned}
\]

Let us show $\Im(\pi) \subset \Delta$. 
Take  a closed point $\bigl([q\colon W_7\twoheadrightarrow Q_4],
[\tau\colon \wedge^2Q_4\twoheadrightarrow k]\bigr) \in \bP_{\Gr(W_7,4)}(\wedge^2\cQ')$.  
We have $\Ker(q) \subset \Ker(\varphi_\sigma)$ for 
$\sigma := \tau\circ\wedge^2q
\colon \wedge^2W_7
\overset{\wedge^2 q}{\twoheadrightarrow} \wedge^2 Q_2 \overset{\tau}{\twoheadrightarrow} k$, which implies 
\begin{equation}\label{e1 codim_rk_4_locus}
\rank(\varphi_\sigma) = 7 -\dim \Ker(\varphi_\sigma) 
\leq 7 -\dim \Ker(q) = \dim \Im(q) \leq \dim Q_4 =4. 
\end{equation}
This completes the proof of $\Im(\pi) \subset \Delta$. 

Let us show that $\pi : \bP_{\Gr(W_7,4)}(\wedge^2\cQ') \to \Delta$ is generically injective. 
Take a general closed point $[\sigma] \in \Delta$, so that 
$\sigma\colon\wedge^2W_7\to k$ is of rank $4$, i.e., $\rank (\varphi_\sigma)= 4$. 
If $\sigma =\tau \circ \wedge^2 q$ as above, 
then we get $\Ker(q) \subset \Ker(\varphi_\sigma)$ and (\ref{e1 codim_rk_4_locus}). 
As $\rank (\varphi_\sigma)= 4$, all the inequalities in (\ref{e1 codim_rk_4_locus}) are equalities. 
Hence $\Ker(q) = \Ker(\varphi_\sigma)$. 
In particular, the closed point $[q]\in\Gr(W_7,4)$ is uniquely determined by
$[\sigma]$.
Let $q\colon W_7\twoheadrightarrow W_7/\Ker(\varphi_\sigma) =:Q_4$ be the quotient linear map. 
Then there exists a unique alternating form
$\tau\colon\wedge^2Q_4\to k$ satisfying $\sigma=\tau\circ\wedge^2q$. 
Hence the pair $([q],[\tau])$ is a unique closed point mapped to $[\sigma]$ by $\pi$. 
This completes the proof of the generic injectivity. 
\end{proof}

Let $\nu \colon \wedge^2 W_7 \to N_3$ be a non-degenerate net of alternating forms.
On the projective plane $\bP(N_3)$, we have the universal quotient map $N_3 \otimes \cO_{\bP(N_3)} \to \cO_{\bP(N_3)} (1)$.
By composing this with the map $ \nu \otimes \id \colon \wedge^2 W_7 \otimes \cO_{\bP(N_3)}  \to N_3 \otimes \cO_{\bP(N_3)} $,
 we have the induced map $\wedge^2 W_7 \otimes \cO_{\bP(N_3)}  \to \cO_{\bP(N_3)} (1)$, which defines an alternating map
\[
\varphi_{\nu} \colon W_7 \otimes \cO_{\bP(N_3)}  \to W_7^\vee \otimes \cO_{\bP(N_3)} (1).
\]

\begin{proposition}\label{proposition:triple_veronese}
Let $W_7$ be a $7$-dimensional vector space and let
$\nu\colon\wedge^2W_7\to N_3$ be a non-degenerate net of alternating forms. 
 Then the following hold:
\begin{enumerate}
 \item 
 There exists an exact sequence 
 \[
 0 \to \cO_{\bP(N_3)}(-4) \to W_7 \otimes \cO_{\bP(N_3)}(-1)  \xrightarrow{\varphi_{\nu} \otimes \cO_{\bP(N_3)}(-1)} W_7^\vee \otimes \cO_{\bP(N_3)}  \xrightarrow{\gamma_{\nu}} \cO_{\bP(N_3)}(3) \to 0.
 \]
 \item 
 Let $\kappa \colon \bP(N_3)\to\bP(W_7^\vee)$ be the morphism induced by the surjection
$ \gamma_\nu \colon 
W_7^\vee\otimes\cO_{\bP(N_3)}
\twoheadrightarrow
\cO_{\bP(N_3)}(3)$. 
Then $\kappa$ is a closed immersion.
\item The image $\kappa (\bP(N_3))$ is not contained in any hyperplane in $\bP(W_7^\vee)$.
\end{enumerate}
\end{proposition}

\noindent
The assertions are essentially proved in \cite[Lemma B.1]{BKM26b}.
We nevertheless include a proof, since the original argument requires additional care in characteristics $2$ and $3$; see Remark~\ref{remark:reduced_power}.

\begin{proof}
By the non-degeneracy of $\nu$, the morphism
$\varphi_{\nu}\otimes\cO_{\bP(N_3)}(-1)$ has constant rank $6$.
Hence its kernel $K$ and cokernel $L$ are line bundles, and we have an exact sequence 
\[
0 \to K
\to W_7\otimes\cO_{\bP(N_3)}(-1)
\xrightarrow{\varphi_{\nu}\otimes\cO_{\bP(N_3)}(-1)}
W_7^\vee\otimes\cO_{\bP(N_3)}
\xrightarrow{\gamma_\nu}
L
\to0.
\]
Since $\varphi_\nu$ is alternating, we get $K\simeq L^\vee\otimes\cO_{\bP(N_3)}(-1)$ by applying $(-)^\vee \otimes \MO_{\P(N_3)}(1)$ to  the above exact sequence. 
Thus we obtain the exact sequence
\[
0 \to L^\vee\otimes\cO_{\bP(N_3)}(-1)
\to W_7\otimes\cO_{\bP(N_3)}(-1)
\xrightarrow{\varphi_{\nu}\otimes\cO_{\bP(N_3)}(-1)}
W_7^\vee\otimes\cO_{\bP(N_3)}
\xrightarrow{\gamma_\nu}
L
\to0.
\]
Then, by comparing the determinant of these vector bundles, we have
\[
L \otimes \det(W_7 \otimes \cO_{\bP(N_3)} (-1)) \simeq \det(W_7^\vee \otimes \cO_{\bP(N_3)}) \otimes L^\vee\otimes \cO_{\bP(N_3)} (-1).
\]
Thus $L \simeq \cO_{\bP(N_3)}(3)$, and hence (1) holds.

Note that $\kappa$ is a closed immersion if so is its restriction to any line $l \subset \bP(N_3)$. 
On a line $\bP^1 \simeq l \subset \bP(N_3)$, we have the exact sequence
\[
 0 \to \cO_{\bP^1}(-4) \to W_7 \otimes \cO_{\bP^1} (-1)  \to W_7^\vee\otimes \cO_{\bP^1}  \xrightarrow{\gamma_{\nu}|_{\bP^1}} \cO_{\bP^1}(3) \to 0.
\]
By looking at the long exact sequence of cohomology groups, we see that the following map
\[
H^0(\bP^1, W_7^\vee\otimes \cO_{\bP^1})  \xrightarrow{H^0(\gamma_{\nu}|_{\bP^1})} H^0(\bP^1, \cO_{\bP^1}(3) )
\]
is surjective, since $H^1(\bP^1,W_7 \otimes \cO_{\bP^1} (-1)) = H^2(\bP^1,\cO_{\bP^1}(-4))=0$.
Hence $\kappa|_{\bP^1}$ is a closed immersion induced by the complete linear system $|\cO_{\bP^1}(3)|$.
Thus (2) holds.

Similarly, 
by taking the cohomology long exact sequence on $\bP(N_3)$, 
we see that the following map
\[
H^0(\bP(N_3), W_7^\vee\otimes \cO_{\bP(N_3)})  \xrightarrow{H^0(\gamma_{\nu})} H^0(\bP(N_3), \cO_{\bP(N_3)}(3) )
\]
is injective, since $H^0(\bP(N_3),W_7 \otimes \cO_{\bP(N_3)} (-1)) = H^1(\bP(N_3),\cO_{\bP(N_3)}(-4))=0$.
Thus (3) holds.
\end{proof}

\begin{remark}\label{remark:reduced_power}
In the proof of \cite[Lemma B.1]{BKM26b}, 
 they defined $\kappa$ by using the following power map 
 \[
 \kappa_0 : \wedge^2 W_7^\vee \to \wedge^6 W_7^\vee, \qquad \sigma \mapsto \sigma \wedge \sigma \wedge \sigma.
 \]
In characteristic $2$ or $3$, this map sends all forms $\sigma$ to zero.
Instead, in characteristic $2$ or $3$, we need to use the \emph{reduced power} map
\[
\wedge^2 W_7^\vee \to \wedge^6 W_7^\vee,  \quad \sigma \mapsto \sigma^{[3]},
\]
defined as in \cite[Section 1]{Muk93}.
It is easy to check that the morphism  
$\kappa \colon \bP(N_3) \to \bP(W_7^ \vee)$ introduced in 
Proposition \ref{proposition:triple_veronese} coincides with the morphism induced by the reduced  power map.

More precisely, if $x\in\P(N_3)$ is a closed point represented by a quotient
$f_x:N_3\twoheadrightarrow k$ and 
$\sigma_x:=f_x\circ\nu$, i.e., $\sigma_x : \wedge^2 W_7 \overset{\nu}{\twoheadrightarrow} N_3 \overset{f_x}{\twoheadrightarrow} k$, then $\sigma_x^{[3]} \neq 0$ and 
the one-dimensional linear subspace $k \sigma_x^{[3]} \subset \wedge^6 W_7^\vee$ is identified with $\Ker(\varphi_{\sigma_x})\otimes\det(W_7)^\vee$ via the natural isomorphism 
$\wedge^6W_7^\vee
\simeq
W_7\otimes\det(W_7)^\vee$. 
\end{remark}

\subsection{The Mukai varieties $\Sigma_{12}$}

\begin{definition}\label{d g=12 Mukai var}
Let $W_7$ be a $7$-dimensional vector space and let $\nu \colon \wedge^2 W_7 \to N_3$ be a net of alternating forms.
\begin{enumerate}
 \item On the Grassmannian variety $\Gr(W_7,4)$,  we have the universal exact sequence
\[
0 \to \cK' \to W_7 \otimes \cO_{\Gr(W_7,4)} \to \cQ' \to 0.
\]
Define $\Gr(W_7,4, \nu)$ 
 as the zero locus of the composite $\MO_{\Gr(W_7,4)}$-module homomorphism 
\[
\wedge^2 \cK' \to \wedge^2 W_7 \otimes \cO_{\Gr(W_7,4)} \xrightarrow{\nu \otimes \id} N_3 \otimes \cO_{\Gr(W_7,4)}.
\] 
Namely, $\Gr(W_7,4, \nu)$ is the zero locus of the corresponding section of 
\[
H^0(\Gr(W_7,4), (\wedge^2 \cK') ^\vee \otimes N_3).
\]
\item We set $\Sigma_{12}(\nu) := \Gr(W_7,4, \nu)$ 
when $\nu$ is non-degenerate. 
We call $\Sigma_{12}(\nu)$ the {\em Mukai variety of genus $12$} associated with $\nu$. 

 \item 
 For a $1$-dimensional vector subspace  $V_1 \subset W_7$, 
the corresponding  closed point $[V_1] \in \bP(W_7^\vee) = \Gr(1,W_7)$ 
induces the natural inclusion $\Gr(W_7/V_1, 4) \subset \Gr(W_7, 4)$.
For a closed point $x \in \bP(N_3)$, we define
\[
C_x \coloneqq \Gr(W_7,4, \nu) \cap \Gr(W_7/V_{1,x}, 4) \subset \Gr(W_7, 4),
\]
where $V_{1,x} \subset W_7$ denotes the $1$-dimensional vector subspace of $W_7$ corresponding to $x$.
\end{enumerate}
\end{definition}

\begin{proposition}\label{proposition:sigma_12_conics}
Let $W_7$ be a $7$-dimensional vector space and let $\nu \colon \wedge^2 W_7 \to N_3$ be a non-degenerate net of alternating forms. 
Then the following hold:
\begin{enumerate}
 \item $\Sigma_{12}(\nu)$ does not contain any plane. 
 \item  If $V_4\subset W_7$ is a $4$-dimensional subspace, then
\[
\iota^{-1}\bigl(\Gr(3,V_4)\bigr)\not\subset\Sigma_{12}(\nu),
\]
where $\iota:\Gr(W_7,4)\xrightarrow{\simeq}\Gr(3,W_7)$ denotes the canonical isomorphism.
 \item For any closed point $x \in \bP(N_3)$, $C_x$ is a $($possibly non-smooth$)$ conic in $\Sigma_{12}(\nu)$. 
 Moreover, any conic on $\Sigma_{12}(\nu)$ coincides with $C_x$ for some $x \in \bP(N_3)$.
 \item The set-theoretic equality $\Sigma_{12}(\nu) = \bigcup_{x \in \P(N_3)} C_x$ holds. 
\end{enumerate}
\end{proposition}

\begin{proof}
By Proposition~\ref{proposition:triple_veronese}, the arguments of
\cite[Appendix B.2]{BKM26b} remain valid over an arbitrary algebraically closed field $k$. 
In what follows, 
we briefly explain how to translate the subspace notation used there
into our quotient notation. 

Set $S_3:=N_3^\vee
\subset \wedge^2W_7^\vee$, where the inclusion is the dual of $\nu$.
Under the canonical isomorphism
\[
\iota:\Gr(W_7,4)\xrightarrow{\simeq}\Gr(3,W_7),
\qquad
[q:W_7\twoheadrightarrow Q_4]
\longmapsto
[\Ker(q)], 
\]
the universal subbundle $\cK'$ on 
$\Gr(W_7,4)$ is identified with the universal subbundle $\cU$
on $\Gr(3,W_7)$.  Moreover, we have the following equivalence: 
\[
\nu(\wedge^2\cK')=0 \qquad \Longleftrightarrow \qquad 
\sigma|_{\wedge^2\cU}=0
\quad\text{for every }\sigma\in S_3.
\]
Thus we have the scheme-theoretic equality $\iota(\Sigma_{12}(\nu))=X$, where $X$ denotes the subvariety of $\Gr(3,W_7)$ appearing in
\cite[Appendix B.2]{BKM26b}. 
More explicitly,  if $x\in\P(N_3)$ is a closed point represented by a
quotient $f_x:N_3\twoheadrightarrow k$, 
then 
we have $V_{1,x}=\Ker(\varphi_{\sigma_x})$ and $\iota(C_x) = X_{\kappa(\sigma_x)}$ 
for $\sigma_x:=f_x\circ\nu\in\wedge^2W_7^\vee$. 
\end{proof}

\begin{proposition}\label{proposition:sigma_12}
Let $W_7$ be a $7$-dimensional vector space and let $\nu \colon \wedge^2 W_7 \to N_3$ be a non-degenerate net of alternating forms.
Then $\Sigma_{12}(\nu)$ is a local complete intersection projective scheme of pure dimension $3$.
Moreover, $\omega_{\Sigma_{12}(\nu)} \simeq \cO_{\Gr(W_7,4)}(-1)|_{\Sigma_{12}(\nu)}$, and $\deg(\Sigma_{12}(\nu)) =22$ with respect to $\cO_{\Gr(W_7,4)}(1)|_{\Sigma_{12}(\nu)}$.
\end{proposition}

\begin{proof}
By using the identification from the proof of
Proposition \ref{proposition:sigma_12_conics},
the assertions follow from the same arguments as in the first two
paragraphs of the proof of \cite[Proposition B.6]{BKM26b}.
\end{proof}

\subsection{Mukai generality of nets of alternating forms}
We briefly recall a non-degeneracy condition introduced by Mukai  in relation to the smoothness of Mukai variety $\Sigma_{12}(\nu)$ \cite[Section 5]{Muk95JAP}.

\begin{definition}\label{d Mukai general nu}
Let $W_7$ be a $7$-dimensional vector space and let $\nu \colon \wedge^2 W_7 \to N_3$ be a net of alternating forms.
\begin{enumerate}
 \item It is easy to see that $\nu$ induces the following linear map: 
 \[
\begin{aligned}
\mu_\nu\colon \wedge^3W_7 &\longrightarrow N_3\otimes W_7,\\
x\wedge y\wedge z
&\longmapsto
\nu(x\wedge y)\otimes z
+\nu(y\wedge z)\otimes x
+\nu(z\wedge x)\otimes y.
\end{aligned}
\]
The dual $\mu_\nu^\vee$ of $\mu_{\nu}$ is the multiplication map given by 
\[
\mu_\nu^\vee: N_3 ^\vee \otimes W_7^\vee 
\xrightarrow{\nu^\vee \otimes {\rm id}_{W_7^\vee}} \wedge^2 W_7^\vee \otimes W_7^\vee \xrightarrow{\alpha \otimes \beta \mapsto \alpha \wedge \beta} \wedge^3 W_7^\vee. 
\]
 \item The net  of alternating forms $\nu \colon \wedge^2 W_7 \to N_3$ is called \emph{Mukai general} if $\bP(\Im \mu_\nu) \subset \bP(\wedge^3 W_7)$ does not intersect the Grassmannian variety $\Gr(W_7, 3) \subset \bP(\wedge^3 W_7)$, i.e., the image of the dual map $\mu_\nu^\vee \colon N_3 ^\vee \otimes W_7^ \vee \to \wedge^3 W_7^\vee$ does not contain any nonzero decomposable vector $f\wedge g \wedge h$, 
 where $f,g, h \in W_7^\vee$.
\end{enumerate}
\end{definition}

\begin{proposition}\label{proposition:mukai-general}
Let $W_7$ be a $7$-dimensional vector space and let $\nu \colon \wedge^2 W_7 \to N_3$ be a net of alternating forms.
If $\nu$ is Mukai general, then it is non-degenerate.
\end{proposition}

\begin{proof}
Assume that $\nu$ is degenerate.
By Definition \ref{d g=12 non-deg}(5), 
there exists a surjective linear map
$f\colon N_3\twoheadrightarrow k$
such that the alternating form
$\sigma\colon
\wedge^2W_7\xrightarrow{\nu}N_3\xrightarrow{f}k$ is of rank $<6$. 
As oth $\nu$ and $f$ are surjective, $\sigma$ is nonzero, and hence 
$\sigma$ is of rank $2$ or $4$. 
Then there exists a basis $e_1, ..., e_7$ of $W_7$ such that 
\[
\sigma=e_1^\vee\wedge e_2^\vee
\qquad\text{or}\qquad
\sigma=e_1^\vee\wedge e_2^\vee
+e_3^\vee\wedge e_4^\vee.
\]
As $\nu^\vee(f)=f \circ \nu = \sigma$, we have $\mu_\nu^\vee(f \otimes e_3^\vee)
=
\sigma\wedge e_3^\vee
=
e_1^\vee\wedge e_2^\vee\wedge e_3^\vee$. 
Thus $\Im(\mu_\nu^\vee)$ contains a nonzero decomposable vector,
so $\nu$ is not Mukai general.
\end{proof}

\begin{remark}
In \cite[Section 5]{Muk95JAP}, it is claimed (at least in characteristic zero) that $\Sigma_{12}(\nu)$ is smooth if and only if $\nu$ is Mukai general.
\end{remark}

\subsection{Linear section theorem for $g=12$}

Let $X$ be a prime Fano threefold of genus $12$ and 
let $\cQ_X$ be a dual Mukai bundle of type $(3,4)$. 
Let $\varphi \colon X \to \Gr(V_7,3)$ be the closed immersion induced by $\cQ_X$, where $V_7 \coloneqq H^0(X,\cQ_X)$ (Proposition~\ref{p Gushel emb X}).
In what follows, we identify $X$ with its image $\varphi(X)$.
Set $W_7:=V_7^\vee$, and identify $\Gr(V_7,3)$ with
$\Gr(W_7,4)$ via the canonical isomorphism
\[
\begin{aligned}
\theta: \Gr(V_7,3)&\xrightarrow{\simeq}\Gr(W_7,4),\\
[q\colon V_7\twoheadrightarrow Q_3]
&\mapsto
[W_7\twoheadrightarrow\Ker(q)^\vee].
\end{aligned}
\]
Under this identification, the dual of the universal exact sequence 
\[
0 \to \cK \to V_7 \otimes \MO_{\Gr(V_7, 3)} \to \cQ \to 0
\]
coincides with the universal exact sequence on $\Gr(W_7,4)$:
\[
0 \to \cK' \to W_7 \otimes \MO_{\Gr(W_7, 4)}  \to \cQ' \to 0.
\]

\begin{definition}\label{d Mukai general S3}
We use the notation introduced above. 
\begin{enumerate}
 \item A three-dimensional subspace $S_3 \subset \wedge ^2 V_7$ is called a \emph{net of bivectors} on $V_7$.
Nets of bivectors $S_3 \subset \wedge^2 V_7$ on $V_7$ correspond to 
nets of alternating forms 
$\wedge^2 W_7 \twoheadrightarrow N_3$ on $W_7$ under the canonical isomorphism 
$(\wedge^2  V_7)^\vee \simeq  \wedge^2 (V_7^\vee) =\wedge^2 W_7$:
\[
S_3\subset\wedge^2V_7
\quad
\begin{gathered}
\xrightarrow{\;\scriptstyle S_3^\vee =:N_3\;}\\[-2.2ex]
\xleftarrow[\;\scriptstyle S_3:=N_3^\vee\;]{}
\end{gathered}
\quad
\wedge^2W_7\twoheadrightarrow N_3.
\]
\item A net of bivectors $S_3 \subset  \wedge^2 V_7$ is called \emph{non-degenerate} 
(resp.\ \emph{Mukai general}) 
if the corresponding net of alternating forms $\wedge^2 W_7 \twoheadrightarrow S_3^\vee$ is \emph{non-degenerate} 
(resp.\ \emph{Mukai general}).
\item 
For a net of bivectors $S_3 \subset \wedge^2 V_7$ and the corresponding 
net of alternating forms $\nu : \wedge^2 W_7 \twoheadrightarrow S_3^\vee$, 
we define $\Gr(V_7,3,S_3)$ as the closed subscheme of $\Gr(V_7, 3)$ corresponding to $\Gr(W_7,4, \nu) \subset \Gr(W_7,4)$: 
\[
\begin{tikzcd}
\Gr(V_7,3) \arrow[r, "\simeq", "\theta"']
&
\Gr(W_7,4)
\\
\Gr(V_7,3,S_3) \arrow[u, hook] \arrow[r, "\simeq"]
&
\Gr(W_7,4, \nu) \arrow[u, hook].
\end{tikzcd}
\]
It is easy to see that 
\[
\Gr(V_7, 3, S_3)
=
\left\{
[q\colon V_7\twoheadrightarrow Q_3]\in\Gr(V_7,3)
\ \middle|\
(\wedge^2q)(S_3)=0
\right\}.
\]
\item We set $\Sigma_{12}(S_3) := \Gr(V_7, 3, S_3)$ 
when $S_3$ is non-degenerate. 
\end{enumerate}
\end{definition}

\begin{thm}\label{t g=12 lin sec thm}
Let $X$ be a prime Fano threefold of genus $12$.
Let $\cQ_X$ be a dual Mukai bundle of type $(3,4)$ and let $\varphi \colon X \hookrightarrow \Gr(V_7,3)$ be the closed immersion induced by $\cQ_X$.
Then the following hold:
\begin{enumerate}
 \item If a net of bivectors $S_3 \subset \wedge^2 V_7$ satisfies $\varphi(X) \subset  \Gr(V_7,3,S_3)$, then it is Mukai general.
 \item There is a unique net of bivectors $S_3 \subset \wedge^2 V_7$ such that $\varphi(X) \subset 
 \Gr(V_7,3,S_3)$.
 \item The morphism $X \to  \Gr(V_7,3,S_3)$ 
 induced by $\varphi$ is an isomorphism.
\end{enumerate}
\end{thm}

\begin{proof}
In what follows, we identify $X$ with its image $\varphi(X)$. 

Let us show (1). 
Take a net of bivectors $S_3\subset\wedge^2V_7$ satisfying
$X\subset\Sigma_{12}(S_3)=\Gr(V_7,3,S_3)$. 
Suppose that $S_3$ is not Mukai general. 
By Definition~\ref{d Mukai general nu}(2), there exists a nonzero
decomposable vector
$\xi=v_1\wedge v_2\wedge v_3$ in the image of the linear map
\[
S_3\otimes V_7\longrightarrow\wedge^3V_7,
\qquad
s\otimes v\longmapsto s\wedge v.
\]
We can write $\xi=\sum_i s_i\wedge u_i$ for some $s_i\in S_3$ and $u_i\in V_7$. 
For every closed point
$[q\colon V_7\twoheadrightarrow Q_3]\in\Sigma_{12}(S_3)$, we have
$(\wedge^2q)(s_i)=0$ (Definition \ref{d Mukai general S3}), and hence
\[
(\wedge^3q)(\xi)
=
\sum_i(\wedge^2q)(s_i)\wedge q(u_i)
=
0.
\]
Thus $\Sigma_{12}(S_3)$ is contained in the Schubert divisor
\[
\left\{
[q\colon V_7\twoheadrightarrow Q_3]\in\Gr(V_7,3)
\ \middle|\
(\wedge^3q)(\xi)=0
\right\}.
\]
In particular, $X$ is contained in a Schubert divisor of $\Gr(V_7,3)$, 
contradicting Proposition \ref{p Sch avoid X}. 
Thus (1) holds. 

Let us show (2).
Let $S \in |-K_X|$ be a general smooth member.
Then we have the following restriction maps:
 \[
\begin{tikzcd}[column sep=large]
\wedge^2V_7 
=
H^0\left(\Gr(V_7, 3),\wedge^2\cQ\right)
\arrow[r, "\rho^{\Gr(V_7, 3)}_X"]
\arrow[rr, bend right=20, "\rho^{\Gr(V_7, 3)}_S"]
&
H^0\left(X,\wedge^2\cQ_X\right)
\arrow[r, "\rho^{X}_S", hook]
&
H^0\left(S,\wedge^2\cQ_S\right).
\end{tikzcd}
\]
Since
$H^0(X, \wedge^2 \cQ_X (-S)) \simeq 
H^0(X, \wedge^2 \cQ_X (K_X)) \simeq 
H^0(X,\cQ_X^\vee) =0$, 
the map $\rho^{X}_S$ is injective. 
Set $K \coloneqq \Ker \rho^{\Gr(V_7, 3)}_X =\Ker \rho^{\Gr(V_7, 3)}_S$.
By (1), any $3$-dimensional subspace $S_3 \subset K$, 
considered as a nef of bivectors on $V_7$, 
is Mukal general, and hence non-degenerate (Proposition~\ref{proposition:mukai-general}). 
This, together with Proposition \ref{proposition:codim_rk_4_locus}, 
implies $\dim K \leq 3$. 
Therefore, it is enough to prove $\dim K \geq 3$.
By the Riemann-Roch theorem, we have $\chi \left(S,\wedge^2\cQ_S\right) = 18$.
Since $H^{>0} (S,\wedge^2\cQ_S )=H^{>0} (S,\cQ_S^\vee (H) ) = 0$ by Proposition~\ref{p BKM4.4 vanishing}, we have $h^{0} (S,\wedge^2\cQ_S) = 18$.
On the other hand, $\dim \wedge^2 V_7 = \binom{7}{2} = 21$.
Hence we have
\[
\dim K \geq \dim \wedge^2 V_7 - h^{0} (S,\wedge^2\cQ_S) = 21 -18 =3. 
\]
Thus (2) holds.

Let us show (3). 
Let $S_3 \subset \wedge^2 V_7$ be the net of bivectors 
satisfying $X \subset \Sigma_{12}(S_3) = \Gr(V_7,3,S_3)$.
By (1), $S_3$ is Mukai general, and hence non-degenerate (Proposition~\ref{proposition:mukai-general}). 
Then $\Sigma_{12}(S_3)$ is a 
pure $3$-dimensional lci closed subscheme in $\Gr(V_7, 3)$ 
with $\deg \Sigma_{12}(S_3) =22$ (Proposition~\ref{proposition:sigma_12}). 
Since $\deg \Sigma_{12}(S_3) =22 =\deg X$, 
we have $[\Sigma_{12}(S_3)] = [X]$ as algebraic $3$-cycles on $\Gr(V_7, 3)$. 
Hence $\Sigma_{12}(S_3)$ is irreducible and generically reduced. 
As $\Sigma_{12}(S_3)$ is Cohen-Macaulay, $\Sigma_{12}(S_3)$ is an integral scheme, which implies the scheme-theoretic equality $\Sigma_{12}(S_3) = X$. 
Thus (3) holds. 
\end{proof}

\begin{rem}
Fix $g \in \{8, 9, 10, 12\}$ and let $\Sigma_g$ be a Mukai variety of genus $g$.
Let $X := \Sigma_g \cap \la X \ra = \Sigma_g \cap  \P^{g+n-2}$ be an $n$-dimensional subvariety, obtained as a dimensionally transversal linear section of $\Sigma_g$.
Recall that $\Sigma_g$ is a subvariety of $\Gr(r+s,r)$, 
where $(r,s)=(2,4),(3,3),(2,5),(3,4)$ for $g=8,9,10,12$, respectively. 
We can check that $X$ is not contained in sub-Grassmannian varieties $\Gr(r+s-1,r)$ or $\Gr(r+s-1,r-1)$ as follows:

Note that we have the equality $\Sigma_g = \Gr(r+s,r) \cap \la \Sigma_g \ra$ (for $g=12$, see Section~\ref{s g=12 span}).
Thus $X = \Gr(r+s,r) \cap \la X \ra$.
Therefore, if $X \subset G$ for $G = \Gr(r+s-1,r)$ or $\Gr(r+s-1,r-1)$, we have the equality $X =G \cap \la X \ra$.
This yields the following inequality:
\begin{align*}
n &=\dim X = \dim (G \cap \la X \ra )\\
&\geq  \dim G + \dim \la X \ra - \dim \la G\ra
=
\begin{cases}
r(s-1)+g+n-2-\binom{r+s-1}{r}+1,\\
s(r-1)+g+n-2-\binom{r+s-1}{r-1}+1.
\end{cases}
\end{align*}
This inequality holds only if $g=12$, and $X \subset \Gr(6,3) \subset \Gr(7,3)$.
In this case, by the dimensional reason, $X=G\cap \la X \ra$ is a transversal linear section of $G = \Gr(6,3)$, which is absurd (for example by adjunction or by degree-count).

In particular, if $X$ is a prime Fano threefold, then $\cQ|_X$ is the unique dual Mukai bundle, where $\cQ$ is the universal quotient bundle of the Grassmannian variety $\Gr(r+s,r)$ \cite[Proposition 6.11]{Tan-bdl}.
It follows that $X$ is uniquely described as a linear section of $\Sigma_g$.
\end{rem}
\section{Linear span of $\Sigma_{12}$}\label{s g=12 span}

\begin{nota}\label{notation1 g=12 linear clsoure}
Let $V_7$ be a $7$-dimensional vector space  and set
$G \coloneqq \Gr(V_7,3)$. 
Let $\cQ$ be the universal quotient bundle on $G$ and 
let $H$ be the ample generator of $\Pic(G)$. 
Set
\[
\cE_0:=(\wedge^2 \cQ)^{\oplus 3}.  
\]
Fix an element $\sigma_0\in H^0(G,\cE_0)$ satisfying $\dim X =3$ for 
the zero locus $X:=Z(\sigma_0)\subset G$  of $\sigma_0$. 
Given $n \in \Z$ and a coherent sheaf $\cF$ on $G$, we set 
$\MO_G(n) := \MO_G(nH)$ and $\cF(n) := \cF \otimes_{\MO_G} \MO_G(n)$. 
For every $1\leq j\leq9$, set $A_j:=(\wedge^j\cE_0^\vee)(1)$. 
\end{nota}

\subsection{Reduction to cohomology vanishings}

\begin{lem}\label{l reduction to Aj}
We use Notation \ref{notation1 g=12 linear clsoure}. 
If $H^j(G,A_j)=0$ for every $1\leq j\leq9$, 
then 
$H^0(G,\MO_G(1))
\to H^0(X,\MO_G(1)|_X)$ is surjective. 
\end{lem}

\begin{proof}
See, e.g., \cite[Example B.1.3 and Section B.2]{Laz04}.
\end{proof}

\begin{lem}\label{l Aj summands}
We use Notation \ref{notation1 g=12 linear clsoure}. 
For every $1 \leq j \leq 9$, $A_j$ is isomorphic to the following $\MO_G$-module. 
\[
\begin{array}{c|l}
j
&
A_j
\\
\hline
1
&
\cQ^{\oplus 3}
\\[2mm]
2
&
(\cQ^\vee)^{\oplus 3}
\oplus
\left(
\cQ^{\otimes 2}(-1)
\right)^{\oplus 3}
\\[2mm]
3
&
\MO_G(-1)^{\oplus 3}
\oplus
\left(
\cQ^\vee \otimes \cQ(-1)
\right)^{\oplus 6}
\oplus
\cQ^{\otimes 3}(-2)
\\[2mm]
4
&
\left(
\cQ^{\otimes 2}
\otimes
\cQ^\vee(-2)
\right)^{\oplus 3}
\oplus
\left(
(\cQ^\vee)^{\otimes 2}(-1)
\right)^{\oplus 3}
\oplus
\cQ(-2)^{\oplus 6}
\\[2mm]
5
&
\left(
\cQ
\otimes
(\cQ^\vee)^{\otimes 2}(-2)
\right)^{\oplus 3}
\oplus
\left(
\cQ^{\otimes 2}(-3)
\right)^{\oplus 3}
\oplus
\cQ^\vee(-2)^{\oplus 6}
\\[2mm]
6
&
(\cQ^\vee)^{\otimes 3}(-2)
\oplus
\left(
\cQ^\vee \otimes \cQ(-3)
\right)^{\oplus 6}
\oplus
\MO_G(-3)^{\oplus 3}
\\[2mm]
7
&
\left(
(\cQ^\vee)^{\otimes 2}(-3)
\right)^{\oplus 3}
\oplus
\cQ(-4)^{\oplus 3}
\\[2mm]
8
&
\cQ^\vee(-4)^{\oplus 3}
\\[2mm]
9
&
\MO_G(-5)
\end{array}
\]
\end{lem}

\begin{proof}
For $E:=\cQ(-1)$, we have $\cE_0^\vee= E^{\oplus3}$. 
Since $\rank E=3$, we have
\[
\wedge^0E\simeq\MO_G,
\quad
\wedge^1E\simeq\cQ(-1),
\quad 
\wedge^2E
\simeq
E^\vee \otimes \det E
\simeq
\cQ^\vee(-1),
\quad
\wedge^3E
\simeq
\MO_G(-2).
\]
Then the assertion follows from 
\[
A_j =(\wedge^j\cE_0^\vee)(1)
=(\wedge^j E^{\oplus 3})(1)
\simeq
\bigoplus_{\substack{a+b+c=j\\0\leq a,b,c\leq3}}
(\wedge^aE
\otimes
\wedge^bE
\otimes
\wedge^cE)(1).
\]    
\end{proof}

\subsection{Flag varieties}
\begin{nota}\label{notation2 g=12 linear clsoure}
Let $V_7$ be a 
$7$-dimensional vector space  and set
$G:=\Gr(V_7,3)$. 
Let $\cQ$ be the universal quotient bundle on $G$.
Consider the following commutative diagram: 
{\scriptsize
\[
\begin{tikzcd}[column sep=tiny, row sep=small]
&&
B \coloneqq \Fl(V_7;2,1)
\arrow[dddll, bend right, "\delta_1"']
\arrow[dddrr, bend left, "\delta_2"'']
&
\\
&&
F \coloneqq \Fl(V_7;3,2,1)
\arrow[u]
\arrow[dl, "\alpha_1"]
\arrow[dr, "\alpha_2"']
\arrow[dd, "\pi"]
\arrow[ddll, bend right, "\gamma_1"]
\arrow[ddrr, bend left, "\gamma_2"']
&
\\
&Y_1 :=\bP_G(\cQ) = \Fl(V_7;3,1)
\arrow[dr, "\pi_1"']
\arrow[ld, "\beta_1"]
&&
Y_2:=\bP_G(\cQ^\vee)= \Fl(V_7;3,2)
\arrow[dl, "\pi_2"]
\arrow[rd, "\beta_2"']
\\
\Gr(V_7,1)&&
G=\Gr(V_7,3) 
&& \Gr(V_7,2)
\end{tikzcd}
\]
}
\hspace{-3.5mm}
Let $H, H_1, H_2$ be the ample generators of 
$\Pic G, \Pic \Gr(V_7,1), \Pic \Gr(V_7,2)$, respectively. 
Set $\xi_1 := c_1( \alpha_1^*\MO_{Y_1}(1))$ and 
$\xi_2 := c_1( \alpha_2^*\MO_{Y_2}(1))$, 
where each $\MO_{Y_i}(1)$ denotes the tautological line bundle of the 
$\bP^2$-bundle $\pi_i$.
\end{nota}

\begin{lem}\label{l g=2 lin closure div compute}
We use Notation \ref{notation2 g=12 linear clsoure}. 
Then the following hold: 
\begin{enumerate}
\item $\MO_{Y_1}(1) = \beta_1^*H_1$ and $\xi_1 = \gamma_1^*H_1$. 
\item $\MO_{Y_2}(1) +\pi_2^*H = \beta_2^*H_2$ and 
$\xi_2 + \pi^*H = \gamma_2^*H_2$. 
\item $-K_{Y_1}  \sim 3\beta_1^*H_1
+6 \pi_1^*H$.
\item $-K_{Y_2} \sim 3\beta_2^*H_2
+5 \pi_2^*H$.
\item $-K_{B}\sim 2\delta_1^*H_1 + 6\delta_2^*H_2$.
\item 
$-K_F \sim 2\xi_1+2\xi_2+7\pi^*H$. 
\item We have the following exact sequences:
\[
0 \to \cO_{Y_1}(-1) \to \pi_1^* \cQ ^ \vee \to 
(\alpha_1)_{*}\cO_F(\xi_2) \to 0,
\]
\[
0 \to \cO_{Y_2}(-1) \to \pi_2^* \cQ \to (\alpha_2)_{*}\cO_F(\xi_1) \to 0.
\]
\end{enumerate}
\end{lem}

\begin{proof}
Given a smooth projective variety $Z$, 
a vector bundle $\cE$ of rank $r$, and 
the induced projection $\rho : \P_Z(\cE) \to Z$, 
we have 
\begin{enumerate}
\item[(i)] $K_{\bP_Z(\cE)}
\sim
\MO_{\bP_Z(\cE)}(-r) 
+
\rho^*(K_Z+\det\cE)$, and 
\item[(ii)] 
the relative Euler exact sequence 
\[
0 \to
\Omega_{\rho}\otimes \cO_{\bP_Z(\cE)}(1)
\to
\rho^*\cE 
\to
\cO_{\bP_Z(\cE)}(1)
\to
0.\]
\end{enumerate}
Here (ii) holds by \cite[Chapter III, Exercise 8.4(b)]{Har77}, while 
(i) is obtained by applying the determinant to (ii). 

Let us prove (1) and (3).
On $Y_1 =\bP_G(\cQ)$, we have the following exact sequence by (ii): 
\begin{equation}\label{equation:euler_sequence1}
 0 \to \Omega_{\pi_1}\otimes \cO_{Y_1}(1) \to \pi_1^*\cQ \to \cO_{Y_1}(1) \to 0.
\end{equation}
By the universal property of $\Gr(V_7,1)$, 
 the composite surjection 
$V_7 \otimes \cO_{Y_1} \to \pi_1^*\cQ \to \cO_{Y_1}(1) $
corresponds to the morphism $\beta_1 \colon Y_1\longrightarrow \Gr(V_7,1)$. 
Thus $\beta_1^* H_1 = \cO_{Y_1}(1)$ and (1) holds.
We then  get 
\[
K_{Y_1} \overset{{\rm (i)}}{\sim} 
\MO_{Y_1}(-3) + \pi_1^*(K_G + \det \cQ) 
\overset{{\rm (1)}}{\sim} -3\beta_1^*H_1 -6\pi_1^*H. 
\]
Thus (3) holds. 

Let us prove (2) and (4).
Similarly, on $Y_2 \simeq \bP_G(\cQ^\vee)$, 
we have the following exact sequence by (ii): 
\begin{equation}\label{equation:euler_sequence2}
0 \to  \cO_{Y_2}(-1)  \to \pi_2^*\cQ \to T_{\pi_2}\otimes \cO_{Y_2}(-1) \to 0.
\end{equation}
By the universal property of $\Gr(V_7,2)$, 
 the composite surjection 
 \begin{equation}\label{equation:euler_sequence3}
  V_7 \otimes \cO_{Y_2} \to \pi_2^*\cQ  \to T_{\pi_2}\otimes \cO_{Y_2}(-1) 
\end{equation}
corresponds to the morphism $\beta_2 \colon Y_2\longrightarrow \Gr(V_7,2)$. 
Thus $\beta_2^* \cQ_2 \simeq T_{\pi_2}\otimes \cO_{Y_2}(-1)$, where $\cQ_2$ denotes the universal quotient bundle on $\Gr(V_7,2)$.
Hence we have
\[
\beta_2^* H_2 \sim \beta_2^* c_1(\cQ_2) \sim c_1(T_{\pi_2}\otimes \cO_{Y_2}(-1)) 
\overset{(\ref{equation:euler_sequence2})}{\sim} 
c_1(\pi_2^*\cQ) -c_1(\MO_{Y_2}(-1)) = \cO_{Y_2}(1) + \pi_2^*H,
\]
and (2) holds. We then get 
\[
K_{Y_2} \overset{{\rm (i)}}{\sim} 
\MO_{Y_2}(-3) + \pi_2^*(K_G + \det \cQ^\vee) 
\overset{{\rm (2)}}{\sim} -3\beta_2^*H_2 -5\pi_2^*H. 
\]
Thus (4) holds. 

Let us prove (5). 
By the same argument as in (1), we get $\MO_B(1) = \delta_1^*H_1$. 
This, together with the description $B \simeq \bP_{\Gr(V_7,2)}(\cQ_2)$, implies 
\[
K_{B} \overset{{\rm (i)}}{\sim} 
\MO_{B}(-2) + 
\delta_2^*
\left(
K_{\Gr(V_7,2)}
+
\det\cQ_2
\right)
\sim 
-2\delta_1^*H_1-6\delta_2^*H_2. 
\]
Thus (5) holds. 

Let us prove (6) and (7).
Consider $\bP_{Y_2}( T_{\pi_2}\otimes \cO_{Y_2}(-1) )$.
This is naturally isomorphic to $F = \Fl(V_7;3,2,1)$ as follows.
On the projective bundle $f \colon \bP_{Y_2}( T_{\pi_2}\otimes \cO_{Y_2}(-1) ) \to Y_2$, we have the universal surjection
\[
f^* (T_{\pi_2}\otimes \cO_{Y_2}(-1)) \to \cO_{\bP_{Y_2}( T_{\pi_2}\otimes \cO_{Y_2}(-1) )} (1).
\]
Composing this with the map obtained by pulling back \eqref{equation:euler_sequence3}, we have the sequence of surjections
\[
V_7 \otimes \cO_{\bP_{Y_2}( T_{\pi_2}\otimes \cO_{Y_2}(-1) )} \to f^* \pi_2^*\cQ  \to f^*(T_{\pi_2}\otimes \cO_{Y_2}(-1)) \to \cO_{\bP_{Y_2}( T_{\pi_2}\otimes \cO_{Y_2}(-1) )} (1).
\]
This induces the isomorphism $\bP_{Y_2}( T_{\pi_2}\otimes \cO_{Y_2}(-1) ) \to F = \Fl(V_7;3,2,1)$.
In particular,  we have
\[
\cO_{\bP_{Y_2}( T_{\pi_2}\otimes \cO_{Y_2}(-1) )} (1) =  \alpha_1^* \cO_{Y_1}(1) = \xi_1, \quad \text{and}  \quad T_{\pi_2}\otimes \cO_{Y_2}(-1) =  (\alpha_2)_*\cO_F(\xi_1).
\]
By \eqref{equation:euler_sequence2}, we have the second sequence of (7).
By symmetry, we also have the first sequence of (7). 
Since
$c_1\left(
T_{\pi_2}\otimes\cO_{Y_2}(-1)
\right)
\sim
\beta_2^*H_2$, 
the following holds:
\[
\begin{aligned}
K_F
\overset{{\rm (i)}}{\sim}&
-2\xi_1
+
\alpha_2^*
\left(
K_{Y_2}
+
\det\left(
T_{\pi_2}\otimes\cO_{Y_2}(-1)
\right)
\right)
\\
\overset{{\rm (4)}}{\sim}&
-2\xi_1
+\alpha_2^*
( (-3\beta_2^*H_2 -5\pi_2^*H) + \beta_2^*H_2)
\\
\overset{{\rm (2)}}{\sim}&
-2\xi_1
-2\xi_2
-7\pi^*H.
\end{aligned}
\]
Thus (6) holds. 
\end{proof}

\subsection{Vanishing results}

\begin{lem}\label{lemma}
We use Notation \ref{notation2 g=12 linear clsoure}. 
Let $a, b, \ell$ be integers satisfying $a \geq 0$, $3 \geq b \geq  0$, and $\ell \geq b-5$.
Set
\[
D_{a,b,\ell}
:=
a\xi_1+b\xi_2+\ell\pi^*H.
\]
Then $H^i\left(
F,
\MO_F(D_{a,b,\ell})
\right)
=
0$ for every $i>0$.
\end{lem}

\begin{proof}
 By Lemma \ref{l g=2 lin closure div compute}, we have
\[
\begin{aligned}
D_{a,b,\ell}-K_F
&\sim
(a+2)\xi_1+(b+2)\xi_2+(\ell+7)\pi^*H
\\
&\sim
(a+2)\gamma_1^*H_1
+
(b+2)\gamma_2^*H_2
+(\ell-b+5)\pi^*H,
\end{aligned}
\]
which is an ample divisor if $\ell > b-5$.
Then the assertion follows from the Kodaira vanishing theorem on the flag variety $F$ if $\ell > b-5$ \cite[Theorem 2 in page 38]{MR85}.

In the rest of the proof, assume $\ell =b-5$.
We then  have 
\[
\begin{aligned}
D_{a,b,\ell}-K_F
&\sim
(a+2)\gamma_1^*H_1
+
(b+2)\gamma_2^*H_2.
\end{aligned}
\]
Set $A_B := (a+2)\delta_1^*H_1+(b+2)\delta_2^*H_2$.
Then the following holds by Serre duality: 
\[
\begin{aligned}
H^i\left(
F,
D_{a,b,\ell} 
\right) 
&\simeq  H^i(F, K_F+(a+2)\gamma_1^*H_1
+
(b+2)\gamma_2^*H_2)
\\
&\simeq H^{\dim F-i}(F, -(a+2)\gamma_1^*H_1
-
(b+2)\gamma_2^*H_2.)^\vee
\\
&\simeq H^{\dim F-i}(B, -A_B)^\vee
\\
&\simeq H^{\dim B -(\dim F-i)}(B, K_B +A_B)
\\
&= H^{i-4}(B, K_B +A_B). 
\end{aligned}
\]
By the Kodaira vanishing on the flag variety $B = \Fl(V_7;2,1)$, 
it holds that $H^{i-4}(B, K_B + A_B)=0$ for $i\neq 4$.
On the other hand, the following holds (Lemma \ref{l g=2 lin closure div compute}): 
\[
K_B + A_B \sim  (-2\delta_1^*H_1 - 6\delta_2^*H_2) + ((a+2)\delta_1^*H_1+(b+2)\delta_2^*H_2)
= a\delta_1^*H_1 +(b-4)\delta_2^*H_2.
\]
We then get $H^0(B, K_B + A_B) =0$ by $b-4<0$, because 
$(K_B +A_B) \cdot C = (a\delta_1^*H_1 +(b-4)\delta_2^*H_2) \cdot C = (b-4)\delta_2^*H_2 \cdot C <0$ 
for every curve $C$ contained in a fibre of $B \to \Gr(V_7,1)$. 
\end{proof}

\begin{prop}\label{p self contained unified vanishing}
We use Notation \ref{notation2 g=12 linear clsoure}. 
Let $a, b, \ell$ be integers 
satisfying $a \geq 0$, $3 \geq b \geq 0$, 
and $\ell\geq b-5$.
Assume $a \leq 1$ or $b \leq 1$. 
Then, for every $i>0$, we have
\[
H^i\left(
G,
\Sym^a\cQ
\otimes
\Sym^b(\cQ^\vee)(\ell)
\right)
=
0.
\]
\end{prop}

\begin{proof}
Set
$D_{a,b,\ell}
:=
a\xi_1+b\xi_2+\ell\pi^*H$. 
We treat the following three cases separately: 
\begin{enumerate}
\item[(i)] $a=0$ or $b=0$. 
\item[(ii)] $a>0$ and $b=1$. 
\item[(iii)] $a=1$ and $b>0$. 
\end{enumerate}

(i) If $b=0$, then  we have
\[
H^i\left(G,\Sym^a\cQ (\ell) \right) 
\simeq
H^i\left(Y_1, \cO_{Y_1}(a) + \ell \pi_1^*  H \right)
\simeq
H^i(F, \MO_F(D_{a, 0, \ell})) 
\overset{\ref{lemma}}{=}0. 
\]
Similarly, the case $a=0$ is settled by
\[
\begin{aligned}
H^i\left(G,\Sym^b(\cQ^\vee)(\ell)\right) 
\simeq
H^i\left(Y_2, \cO_{Y_2}(b) + \ell \pi_2^*  H \right)
\simeq
H^i(F, \MO_F(D_{0, b, \ell}))
\overset{\ref{lemma}}{=}
0. 
\end{aligned}
\]
This completes the proof for the case (i).

(ii) Assume $a>0$ and $b=1$.
By Lemma~\ref{l g=2 lin closure div compute}(7),  we have an exact sequence
\[
0 \to \cO_{Y_1}(-1) \to \pi_1^* \cQ ^ \vee \to 
(\alpha_1)_{*}\cO_F(\xi_2) \to 0.
\]
Applying  $(\pi_1)_*(\cO_{Y_1}(a) \otimes  -) \otimes \cO_G(\ell H)$ to this sequence, 
we get the following exact sequence:
\[
0 \to \Sym^{a-1}\cQ (\ell )\to \Sym^a\cQ \otimes  \cQ ^ \vee (\ell )\to \pi_*\cO_F (D_{a, 1, \ell}) \to 0.
\]
Since $H^{>0}(G, \Sym^{a-1}\cQ (\ell )) \overset{{\rm (i)}}{=}0$ 
and 
$H^{>0} (G, \pi_*\cO_F (D_{a, 1, \ell}) ) 
\simeq H^{>0}(F, \cO_F (D_{a, 1, \ell})) =
0$ (Lemma \ref{lemma}), we obtain $H^{>0} (G,\Sym^a\cQ \otimes  \cQ ^ \vee (\ell )) =0.$ 
This completes the proof for the case (ii). 

(iii) Assume $a=1$ and $b>0$. 
Arguing as in (ii) by using Lemma~\ref{l g=2 lin closure div compute}(7),  we get an exact sequence
\[
0 \to \Sym^{b-1}(\cQ^\vee) (\ell )\to 
\cQ \otimes  \Sym^b(\cQ ^ \vee) (\ell )\to \pi_*\cO_F (D_{1, b, \ell}) \to 0.
\]
Since $H^{>0}(G,\Sym^{b-1}(\cQ^\vee) (\ell )) \overset{{\rm (i)}}{=}0$ 
and $H^{>0} (G, \pi_*\cO_F (D_{1, b, \ell}) ) 
\simeq H^{>0}(F, \cO_F (D_{1, b, \ell}) ) =0$ (Lemma \ref{lemma}), 
it holds that 
$H^{>0} (G,\cQ \otimes  \Sym^b(\cQ ^ \vee) (\ell )) =0$, as required. 
\end{proof}

\begin{thm}\label{t g=12 linear closure}
We use Notation \ref{notation1 g=12 linear clsoure}. 
Then the restriction map
\[
H^0(G,\MO_G(1))
\longrightarrow
H^0(X,\MO_G(1)|_X)
\]
is surjective.
\end{thm}

\begin{proof}
By Lemma \ref{l reduction to Aj}, 
it is enough to prove
$H^j(G,A_j)=0$ for every $1\leq j\leq9$. 
Each $A_j$ is a direct sum of some copies of the following sheaves (Lemma \ref{l Aj summands}):
\[
\begin{array}{c|l}
j
&
\text{Direct summands of } A_j
\\
\hline
1
&
\cQ
\\
2
&
\cQ^\vee,\quad 
\cQ^{\otimes2}(-1)
\\
3
&
\MO_G(-1),\quad
\cQ^\vee \otimes \cQ(-1)
,\quad
\cQ^{\otimes3}(-2)
\\
4
&
\cQ^{\otimes2}
\otimes
\cQ^\vee(-2)
,\quad
(\cQ^\vee)^{\otimes2}(-1)
,\quad
\cQ(-2)
\\
5
&
\cQ
\otimes
(\cQ^\vee)^{\otimes2}(-2)
,\quad
\cQ^{\otimes2}(-3)
,\quad
\cQ^\vee(-2)
\\
6
&
(\cQ^\vee)^{\otimes3}(-2)
,\quad
\cQ^\vee \otimes \cQ(-3)
,\quad
\MO_G(-3)
\\
7
&
(\cQ^\vee)^{\otimes2}(-3)
,\quad
\cQ(-4)
\\
8
&
\cQ^\vee(-4)
\\
9
&
\MO_G(-5)
\end{array}
\]
In what follows, we use Notation \ref{notation2 g=12 linear clsoure}. 
By Proposition
\ref{p self contained unified vanishing}, we have 
\begin{enumerate}
\item[(A)] 
$H^{>0}\left(
G,
\Sym^a\cQ
\otimes
\Sym^b(\cQ^\vee)(\ell)
\right)
=
0$ if 
$
\begin{cases}
\text{ $0 \leq a \leq 1$, $0 \leq b \leq 3$, and $\ell \geq b-5$,} \\
\text{ $0 \leq a$, $0 \leq b \leq 1$, and $\ell \geq b-5$.} 
\end{cases}
$
\end{enumerate}
Comparing the above list and (A), 
it is enough to prove the following: 
\begin{enumerate}
\item $H^2(G,\cQ^{\otimes2}(-1))=0$.
\item $H^3(G, \cQ^{\otimes3}(-2))=0$.
\item 
$H^4(G, \cQ^{\otimes2}
\otimes
\cQ^\vee(-2)) =0$.
\item 
$H^4(G, (\cQ^\vee)^{\otimes2}(-1))=0$.  
\item 
$H^5(G, \cQ
\otimes
(\cQ^\vee)^{\otimes2}(-2))=0$. 
\item 
$H^5(G, \cQ^{\otimes2}(-3))=0$.  
\item 
$H^6(G, (\cQ^\vee)^{\otimes3}(-2))=0$.
\item $H^7(G, (\cQ^\vee)^{\otimes2}(-3))=0$. 
\end{enumerate}
For a vector bundle $F$ of rank $3$, we have 
$\wedge^2F \simeq F^\vee \otimes \det F$ and the natural exact sequences 
$0
\to
\wedge^2F 
\to
F^{\otimes2}
\to
\Sym^2F
\to
0$
and 
$0
\to
\wedge^3F
\to
\wedge^2F\otimes F
\to
F\otimes\Sym^2F
\to
\Sym^3F
\to
0$ 
\cite[Example B.2.1]{Laz04}. 
Hence we have the following exact sequences: 
\begin{equation}
\label{e1 U tensor sym-wedge}
0 \to \cQ(-1) \to (\cQ^\vee)^{\otimes 2} \to \Sym^2 (\cQ^\vee) \to 0.
\end{equation}
\begin{equation}
\label{e2 Uvee tensor sym-wedge}
0 \to \cQ^\vee(1) \to \cQ^{\otimes 2} \to \Sym^2 \cQ \to 0.
\end{equation}
\begin{equation}
\label{e3 tripleU}
0 \to \MO_G(-1) \to 
\cQ^\vee \otimes \cQ (-1) \to \cQ^\vee \otimes \Sym^2 (\cQ^\vee) \to \Sym^3 (\cQ^\vee) \to 0.
\end{equation}
\begin{equation}
\label{e4 tripleUvee}
0 \to \MO_G(1) 
\to 
\cQ^\vee \otimes \cQ (1) 
\to \cQ\otimes \Sym^2 \cQ
\to \Sym^3 \cQ\to 0.
\end{equation}
By (A) and the exact sequences (\ref{e1 U tensor sym-wedge}) and 
(\ref{e2 Uvee tensor sym-wedge}), (B)  holds. 
\begin{enumerate}
\item[(B)] 
$H^{>0}(
G, 
(\cQ^\vee)^{\otimes 2}(\ell))
= H^{>0}(
G, 
\cQ^{\otimes 2}(\ell))
=0$ if $\ell\geq -3$. 
\end{enumerate}
Thus (1), (4), (6), (8) hold. 
The equality (3) holds by using the exact sequence 
(\ref{e2 Uvee tensor sym-wedge})$\otimes (\cQ^\vee) (-2)$ 
together with (A) and (B). 
Similarly, the equality (5) holds by using the exact sequence 
(\ref{e1 U tensor sym-wedge})$\otimes \cQ (-2)$ 
together with (A) and (B). 

\medskip

Let us show (2). 
By (A), we have 
\[
H^{>0}(G, \MO_G(-1)) = H^{>0}(G, \cQ^\vee \otimes \cQ(-1)) = H^{>0}(G, \Sym^3 \cQ(-2))=0. 
\]
This, together with (\ref{e4 tripleUvee})$\otimes \MO_G(-2)$, implies 
$H^{>0}(G, \cQ \otimes \Sym^2 \cQ (-2))=0$. 
Then (2) follows from (A) and 
the exact sequence (\ref{e2 Uvee tensor sym-wedge})$\otimes \cQ(-2)$. 

Let us show (7). 
By (A), we have 
\[
H^{>0}(G, \MO_G(-3)) 
= H^{>0}(G, \cQ^\vee \otimes \cQ(-3)) = H^{>0}(G, \Sym^3 (\cQ^\vee)(-2))=0.
\]
This, together with (\ref{e3 tripleU})$\otimes \MO_G(-2)$, implies 
$H^{>0}(G, \cQ^\vee \otimes \Sym^2 (\cQ^\vee) (-2))=0$. 
Then (7) follows from (A) and 
the exact sequence 
(\ref{e1 U tensor sym-wedge})$\otimes \cQ^\vee(-2)$. 
\end{proof}

\section{Genus $7$}\label{s g=7}

\subsection{The spinor $10$-fold $\Sigma_7$ and canonical curves of genus $7$}\label{ss spinor10}
Here, we briefly recall the definition of the spinor $10$-fold $\Sigma_7$, and Mukai's linear section theorem for curves of genus $7$ \cite{Muk95}.

Let $V_{10}$ be a $10$-dimensional vector space, and take an element $q \in \Sym^2 V_{10}$.
Assume that $q$ is non-degenerate, i.e., the quadric hypersurface defined by $q$ is smooth.
Define $\OGr(V_{10},5, q)$ as the orthogonal Grassmannian variety of maximal $q$-isotropic quotients of $V_{10}$, which is the zero locus of $q$ considered as a section of $\Sym^2 \cQ$ via
\[
q \in \Sym^2 V_{10} \simeq \Sym^2 H^0 (\Gr(V_{10},5), \cQ) \simeq H^0(\Gr(V_{10},5), \Sym^2 \cQ).
\]
The orthogonal Grassmannian variety $\OGr(V_{10},5, q)$ has exactly two connected components, which are isomorphic to each other.
Denote by $\Sigma_7$ (the isomorphism class of) a connected component of $\OGr(V_{10},5, q)$.

It is known that $\Pic(\Sigma_7) \simeq \bZ$, and $\cO_{\Sigma_7}(2) \simeq \cO_{\Gr(V_{10},5)} (1)|_{\Sigma_7}$, where $\cO_{\Sigma_7}(1)$ is the ample generator of $\Pic(\Sigma_7)$.
The line bundle $\cO_{\Sigma_7}(1)$ is very ample, and its complete linear system defines the so-called spinor embedding $\Sigma_7 \to \bP^{15}$.
By \cite[Proposition~1.9]{Muk95}, the variety $\Sigma_7$ is defined by $10$ quadratic equations in $\bP^{15}$.
More precisely, we have an identification $V_{10}  \simeq H^0(\bP^{15},\cI_{\Sigma_7 / \bP^{15}}(2))$ such that the universal quotient $V_{10} \otimes \cO_{\Gr(V_{10},5)}\to \cQ$ on $\Gr(V_{10},5)$ restricts to the natural surjective map
\[
H^0(\bP^{15},\cI_{\Sigma_7 / \bP^{15}}(2)) \otimes \cO_{\Sigma_7} \to \cN^\vee _{\Sigma_7 / \bP^{15}} (2) \coloneqq (\cI_{\Sigma_7 / \bP^{15}}/\cI^2_{\Sigma_7 / \bP^{15}})(2).
\]
The quadratic form $q \in \Sym^2 V_{10}$ defines the quadratic relations among the $10$ quadratic equations defining $\Sigma_7 \subset \bP^{15}$ \cite[Corollary 1.11]{Muk95}.
Namely, $q$ is contained in the kernel of the following multiplication map
\[
\Sym^2 V_{10} \simeq \Sym^2 H^0(\bP^{15},\cI_{\Sigma_7 / \bP^{15}}(2)) \to H^0(\bP^{15}, \cI^2_{\Sigma_7 / \bP^{15}}(4))) \subset H^0(\bP^{15}, \cO(4)).
\]

By \cite[Section 2]{Muk95}, we have $\dim \Sigma_7 = 10$, $\cO(-K_{\Sigma_7}) \simeq \cO_{\Sigma_7}(8)$, and $\deg \Sigma_7 =12$.
Thus, by the adjunction formula and the Bertini theorem, a general linear section $C \subset \bP^{6}$ of $\Sigma_7 \subset \bP^{15}$ is a canonically embedded curve of genus $7$. 
Conversely, any curve of genus $7$ without $g^1_4$ is obtained as a linear section of $\Sigma_7 \subset \bP^{15}$ \cite[Main Theorem]{Muk95}.
The closed embedding $C \to \Sigma_7$ is recovered as follows:
The canonical curve $C \subset \bP^6$ is defined by $10$ quadratic equations, and hence we have the natural surjective map
\[
V_{10}\otimes \cO_{C} \coloneqq H^0(\bP^{6},\cI_{C / \bP^{6}}(2)) \otimes \cO_{C} \to \cN^\vee _{C / \bP^{6}} (2),
\]
with a $10$-dimensional vector space $V_{10} \coloneqq H^0(\bP^{6},\cI_{C / \bP^{6}}(2))$.
Moreover, if $C$ has no $g^1_4$, the dimension of the kernel of the following multiplication map is one:
\[
\Sym^2 V_{10} = \Sym^2 H^0(\bP^{6},\cI_{C/ \bP^{6}}(2)) \to H^0(\bP^{6}, \cI^2_{C / \bP^{6}}(4))) \subset H^0(\bP^{6}, \cO(4)).
\]
Namely, there is a unique quadratic relation $q \in \Sym^2 V_{10}$ of the $10$ quadratic equations of $C \subset \bP^6$.
This induces the map $C \to \Sigma_{7} \subset \OGr(V_{10},5,q)$ (for a 
connected component $\Sigma_7$ of $\OGr(V_{10},5,q)$).

\subsection{$\Ker \mu_X \neq 0$}

\begin{lem}\label{l g=7 euler number}
Let $X \subset \P^8$ be a prime Fano threefold of genus $7$. 
Then $c_3(T_X) = -10$ or  $c_3(T_X) = -6$. 
\end{lem}

\begin{proof}
Take a conic $\Gamma$ on $X$, whose existence is guaranteed by 
\cite[Theorem 2.14]{KTLift1}. 
Let $\sigma : Y \to X$ be the blowup along $\Gamma$. 
We then obtain the following diagram \cite[Proposition 7.3 and Theorem 7.4]{FanoII}: 
\[
\begin{tikzcd}
Y \arrow[d, "\sigma"'] \arrow[rd, "\psi"]& & Y^+ \arrow[ld, "\psi^+"'] \arrow[d, "\tau"]\\
X & Z & W, 
\end{tikzcd}
\]
where 
$Y \xrightarrow{\psi} Z \xleftarrow{\psi^+} Y^+$ is a flop, 
$\tau : Y^+  = \Bl_B\, W \to W$ is a blowup along a smooth curve $B$, 
and $(W, B)$ satisfies one of (1) and (2) below. 
\begin{enumerate}
\item $W \subset \P^4$ is a smooth quadric threefold and 
$B$ is a smooth curve of genus $7$. 
\item $W \subset \P^4$ is a smooth cubic threefold and 
$B$ is a smooth rational curve. 
\end{enumerate}
Fix a prime number $\ell \neq p$. For a quasi-projective scheme $Z$, we set 
$e(Z) := \sum_{i\geq 0} (-1)^i \dim_{\Q_{\ell}} H^i_{c, \et}(Z, \Q_{\ell}) \in \Z$, 
where $H^i_{c, \et}(Z, \Q_{\ell})$ denotes the $i$-th \'etale cohomology with compact support. Recall that the following hold: 
\begin{itemize}
\item $e(X) = c_3(T_X)$ \cite[Tag 0FH0]{SP}. 
\item $e(V) = e(D) + e(V \setminus D)$ for a quasi-projective scheme 
$V$ and a closed subscheme $D$ of $V$ \cite[Tag 0GKP]{SP}. 
\item $e(\P_C(F)) = 2e(C)$ for a smooth projective curve $C$ and a vector bundle $F$ of rank $2$ \cite[Tag 0FGQ]{SP} or \cite[Expos\'e VII, Th\'eor\`eme 2.2.1]{SGA5}. 
\end{itemize}
Therefore, 
\begin{enumerate}
\renewcommand{\labelenumi}{(\roman{enumi})}
\item $e(Y) = e(X) + e(\Gamma)$, $e(Y^+) = e(W) + e(B)$, and 
\item $e(Y) = e(Y^+)$, 
\end{enumerate}
where (ii) follows from the fact that $\Ex(\psi) \simeq \Ex(\psi^+)$ \cite[the proof of Theorem 6.13]{Tan25}. 
Moreover, it is well known that  $e(\Gamma) = 2$ and $e(B) = 2-2g(B)$ for the genus $g(B)$ of $B$. 
We see that  $e(W_d) = d(10-10d + 5d^2 -d^3)$ for a smooth hypersurface $W_d \subset \P^4$ of degree $d$ by using 
$0 \to T_{W_d} \to T_{\P^4}|_{W_d} \to \MO_{W_d}(d) \to 0$ and the Euler sequence. 
To summarise, 
\[
c_3(T_X) = e(X) = e(W) +e(B) -e(\Gamma) = 
\begin{cases}
4 + (2-2 \cdot 7) -2 = -10
&\text{ if (1) holds}. \\
-6 + 2 -2 = -6
&\text{ if (2) holds}. 
\end{cases}
\]
\end{proof}

\begin{prop}\label{p Ker nontriv g=7}
Let $X \subset \P^8$ be a prime Fano threefold of genus $7$ and let 
$\beta \colon \Bl_X \bP^8 \to \bP^8$ be the blowup of $\bP^8$ along $X$. 
Set $E :=\Ex(\beta)$ and $H := \beta^*\MO_{\P^8}(1)$. 
Then the following hold: 
\begin{enumerate}
\item $c_3(T_X) = -10$. 
\item $(2H-E)^8 =2$. 
\item 
$\Ker(\mu_X) \neq 0$ for the induced linear map 
\[
\mu_X : \Sym^2 H^0(\bP^8, \cI_{X/\bP^8} (2)) \to H^0(X, \Sym^2(\cN^\vee_{X/\P^8}(2))).
\]
\end{enumerate}
\end{prop}

\begin{proof}
Let $\pi : E = \bP(\cN_{X/\P^8}^\vee) \to X$ be the induced $\P^4$-bundle. 

Let us show (1) and (2). 
Let $h$ be the class of a hyperplane section of $X$. 
By \cite[Theorem 2.4(3) and (2.5.2)]{FanoI}, we have 
\[
c(T_X)\equiv 1+h+2h^2+ch^3 
\]
for some $c \in \Q$. 
It follows from Lemma \ref{l g=7 euler number} 
that $c = -10/12$ or $c = -6/12$. 
Since $c(T_{\bP^8}|_X)=(1+h)^9$, we get 
\[
c(\cN_{X/\bP^8})
\equiv
\frac{c(T_{\bP^8}|_X)}{c(T_X)}
\equiv
1+8h+26h^2+(42-c)h^3, 
\]
which implies that the total Segre classe is given as follows:
\[
s(\cN_{X/\bP^8})
\equiv
1-8h+38h^2+(c-138)h^3.
\]
For the tautological class $\xi$ of the $\P^4$-bundle $\pi : E \to X$, 
the projective bundle formula $\pi_*(\xi^{4+i}) = s_i(\cN_{X/\P^8})$ 
\cite[\S 3.1, page 47]{Ful98}, together with $h^3 =(-K_X)^3=12$, implies 
\[
\xi^4 (\pi^*h)^3 =12, \qquad
\xi^5 (\pi^*h)^2 =12\cdot (-8),
\]
\[
\xi^6 (\pi^*h) =12 \cdot 38, \qquad
\xi^7  =12\cdot(c-138).
\]
By $H^8=1$, $H^4 \cdot E \equiv 0$, and 
the equality $E|_E=-\xi$, we obtain
\[
\begin{aligned}
(2H-E)^8
&=
2^8
-\binom{8}{5}2^3\cdot12
+\binom{8}{6}2^2\cdot(12 \cdot 8) -\binom{8}{7}2\cdot(12 \cdot 38)
+12 \cdot (138-c) \\
&=-12c +2^8
+24(-56\cdot4
+28\cdot16
-8\cdot38
+69) \\
&=-12c+256+24 \cdot (-11) = -12c-8.
\end{aligned}
\]
Since $X$ is an intersection of quadrics \cite[Theorem 1.2]{FanoI}, 
$|2H-E|$ is base point free, and hence $(2H-E)^8\geq 0$. 
As $12c \in \{-10,  -6\}$, we obtain $12 c  = -10$ (i.e., $c_3(T_X)= -10)$ and $(2H-E)^8 =2$. 
Thus (1) and (2) hold.

Let us show (3). 
Since $|2H-E|$ is base point free, 
we have the morphism 
\[
\alpha \colon \Bl_X \bP^8 \to \bP(H^0(\Bl_X \bP^8, 2 H -E)) 
=\bP(H^0(\bP^8, \cI_{X/\bP^8} (2)))
\]
induced by $|2H-E|$   and its restriction 
\[
\alpha_E \colon \bP(\cN_{X/\P^8}^\vee) = E \to 
\bP(H^0(\bP^8, \cI_{X/\bP^8} (2))).
\]
Since the scheme-theoretic image $\Im \alpha$ is not contained in 
any hyperplane, 
the equality 
\[
2 \overset{{\rm (2)}}{=} (2H-E)^8  = 
\deg( \Bl_X \bP^8 \to \Im \alpha) \cdot \deg (\Im \alpha)
\]
implies that 
$\Im \alpha$ is a quadric hypersurface on $\bP(H^0(\bP^8, \cI_{X/\bP^8} (2))) \simeq \P^9$. 
In particular, $\Im \alpha_E$ is contained in the quadric hypersurface $\Im \alpha$.

As $\alpha \colon \Bl_X \bP^8 \to  \bP(H^0(\bP^8, \cI_{X/\bP^8} (2)))$ corresponds to the composite surjection 
\[
\rho: H^0(\bP^8, \cI_{X/\bP^8} (2)) \otimes \MO_{\Bl_X \P^8}
= H^0(\Bl_X \bP^8, 2 H -E) \otimes \MO_{\Bl_X \P^8}
\twoheadrightarrow \MO_{\Bl_X \bP^8}(2H-E),
\]
$\alpha_E$ is given by 
\[
\rho_E : H^0(\bP^8, \cI_{X/\bP^8} (2)) \otimes \MO_{E} 
\twoheadrightarrow 
\MO_E(1) \otimes \beta_E^*\MO_X(-2K_X).
\]
Applying $H^0(E, \Sym^2(-))$, we get 
\[
H^0(\Sym^2\rho_E) : \Sym^2 H^0(\bP^8, \cI_{X/\bP^8} (2)) \to 
H^0(E, \Sym^2(\MO_E(1) \otimes \beta_E^*\MO_X(-2K_X)))
\]
\[
=H^0(E, (\MO_E(1) \otimes \beta_E^*\MO_X(-2K_X))^{\otimes 2}) = H^0(X, \Sym^2(\cN^\vee_{X/\P^8}(2))).
\]
By construction, 
$H^0(\Sym^2\rho_E)$ coincides with $\mu_X$. 
Since $\Im \alpha_E$ is contained in a quadric hypersurface, 
we get $\Ker \mu_X = \Ker(H^0(\Sym^2\rho_E)) \neq 0$. 
Thus (3) holds. 
\end{proof}

\begin{rem}
As an alternative proof, we can prove Proposition \ref{p Ker nontriv g=7}(3) by using the liftability \cite[Theorem A]{KTLift1} and the corresponding result in characteristic zero, which follows from the linear section theorem \cite[Lemma 6.9, Corollary 6.10]{BKM26b}. 
\end{rem}

\subsection{Linear section theorem for $g=7$}

\begin{nota}\label{n g=7}
Let $X \subset \P^8_k$ be a prime Fano threefold of genus $7$. 
Fix general linear sections $S:= X\cap \bP^7_k$ and $C:= S\cap \bP^6_k (= X \cap \P^6_k)$. 
Set $V_{10} := H^0(\P^8_k, \cI_{X/\bP^8_k}(2))$, 
\[
\cQ_X := \cN^{\vee}_{X/\bP^8_k}(2),\qquad 
\cQ_S := \cN^{\vee}_{S/\bP^7_k}(2), \qquad
\cQ_C := \cN^{\vee}_{C/\bP^6_k}(2). 
\]
We have the morphism 
\[
\varphi : X \to \Gr(V_{10}, 5)
\]
corresponding to the surjection $V_{10} \otimes \cO_X \to \cN^\vee_{X/ \bP^8}(2) = \cQ_{X}$ (Lemma \ref{l g=7 10 5}(2)). 
\end{nota}

\begin{rem}
In what follows, we use the following canonical identifications:
$\cQ_X|_S \simeq \cQ_S, \cQ_S|_C \simeq \cQ_C,$ and
\[
V_{10}
=
H^0(\bP^8,\cI_{X/\bP^8}(2))
\xrightarrow{\simeq} 
H^0(\bP^7,\cI_{S/\bP^7}(2))
\xrightarrow{\simeq} 
H^0(\bP^6,\cI_{C/\bP^6}(2)). 
\]
For instance, 
the equality 
$H^0(\bP^8,\cI_{X/\bP^8}(2)) = H^0(\bP^7,\cI_{S/\bP^7}(2))$ means the canonical restriction isomorphism. 
\end{rem}

\begin{lemma}\label{l g=7 10 5}
We use Notation \ref{n g=7}. 
Then the following hold:
\begin{enumerate}
 \item $\dim V_{10} = \dim H^0(\bP^8,\cI_{X/\bP^8}(2)) =10$.
 \item The induced $\MO_X$-module homomorphism $V_{10} \otimes \cO_X \to \cN^\vee_{X/ \bP^8}(2) = \cQ_{X}$ is surjective.
\end{enumerate} 
\end{lemma}

\begin{proof}
By \cite[Corollary 4.5(2)]{FanoI}, it holds that $h^0(X, \cO_{X}(2)) = h^0(X, -2K_{X}) = 35$.
By \cite[Theorem 6.2]{FanoI}, the following holds: 
\[
h^0\left(\bP^8,\cI_{X/\bP^8}(2)\right)
=
h^0(\bP^8,\MO_{\bP^8}(2)) - h^0(X, \MO_{X}(2))
=45 - 35 =10.
\]
Thus (1) holds.
Since $X \subset  \bP^8_k$ is an intersection of quadrics \cite[Theorem 7.13]{FanoI}, 
$\beta \colon H^0(X, I_{X/ \bP^8}(2)) \otimes_k \MO_{\P^8} \to I_{X/ \bP^8}$ is surjective, which implies the surjectivity of
\[
V_{10} \otimes \cO_X \to \cN^\vee_{X/ \bP^8}(2) = \cQ_{X}.
\]
Thus (2) holds. 
\end{proof}

\begin{prop}\label{p Sigma7 unique exist}
We use Notation \ref{n g=7}. 
Consider the following commutative diagram consisting of 
the induced linear maps:
\[
\begin{tikzcd}
\Sym^2V_{10}
\arrow[r,equal]
\arrow[d,equal] \arrow[rr, bend left, "\mu_X"']& 
\Sym^2H^0(\bP^8,\cI_{X/\bP^8}(2))
\arrow[r]
\arrow[d, "\simeq"]
&
H^0(X,\Sym^2\cQ_X)
\arrow[d]
\\
\Sym^2V_{10}\arrow[r, "\simeq"]  \arrow[rr, bend right, "\mu_C"]& 
\Sym^2H^0(\bP^6,\cI_{C/\bP^6}(2))
\arrow[r]
&
H^0(C,\Sym^2\cQ_C).
\end{tikzcd}
\]
Then there exists a non-degenerate quadratic form
$q\in\Sym^2V_{10}$ such that
\[
\Ker(\mu_X)
=
\Ker(\mu_C)
=
kq.
\]
In particular, the induced morphism
$\varphi : X\to \Gr(V_{10},5)$ factors through a connected component
$\Sigma_7\subset\OGr(V_{10},5,q)$. 
\end{prop}

\begin{proof}
In what follows, we use the identification 
$H^0(\bP^8,\cI_{X/\bP^8}(2)) = H^0(\bP^6,\cI_{C/\bP^6}(2))$ via the canonical restriction isomorphism. 
It holds that 
\[
\Ker(\mu_X)
\subset
\Ker(\mu_C). 
\]
It follows from \cite[Theorem~1.3]{Tan-bdl} 
that $C \subset \bP^6$ is a BN-general canonically embedded curve of genus $7$.
Then every divisor $A$ on $C$ with $\deg A =4$ satisfies 
\[
h^0(A)^2 +2h^0(A)=h^0(A) (h^0(A) +g -1 -\deg A)  
=h^0(A) h^1(A) \leq g = 7
\]
by \cite[Remark 2.8]{Tan-bdl}, and hence $h^0(A) \leq 1$. 
Therefore, $C$ has no $g^1_4$.
By \cite[Corollary~5.3 and Theorem~4.2]{Muk95}, 
there exists a non-degenerate quadratic form
$q\in\Sym^2V_{10}$ such that $\Ker(\mu_C)=kq$. 
Then it suffices to show $\Ker(\mu_X) \neq 0$, which follows from Proposition \ref{p Ker nontriv g=7}.
\end{proof}

\begin{thm}\label{t g=7 Lin Sec Thm}
Let $X$ be a prime Fano threefold of genus $7$ and 
let $\Sigma_7 \subset \P^{15}$ be the spinor embedding (cf.\ Subsection \ref{ss spinor10}). 
Then $X \simeq \Sigma_7 \cap \P^8$ for some $8$-dimensional linear subvariety $\P^8 \subset \P^{15}$. 
\end{thm}

\begin{proof}
We use Notation \ref{n g=7}. 
By Proposition \ref{p Sigma7 unique exist}, 
we have the induced morphism $\psi : X \to \Sigma_7$. 
Since $C$ is BN-general \cite[Theorem 1.3]{Tan-bdl}, $C$ has no $g^1_4$.
Hence $\psi|_C : C \to \Sigma_7$ is a closed immersion such that its image $\psi(C)$ is a linear section $\psi(C) = \Sigma_7 \cap \la \psi(C) \ra$ 
(Subsection \ref{ss spinor10}). 
In what follows, we identify $C$ with its image $\psi(C)$, and hence 
$C = \Sigma_7 \cap \la C \ra$. 
Since $\P^6 := \la C \ra \subset \la \psi(S) \ra$ and $\dim \la \psi(S) \ra \leq 7$, 
we get  $\dim \la \psi(S) \ra =6$ or  $\dim \la \psi(S) \ra=7$. 
If $\dim \la \psi(S) \ra =6$ (i.e., $\la C \ra = \la \psi(S) \ra$), 
then we would get the following contradiction: 
\[
2 = \dim \psi(S)  \leq \dim (\Sigma_7 \cap \la \psi(S) \ra ) 
=  \dim (\Sigma_7 \cap \la \psi(C) \ra )  = \dim C =1. 
\]
Hence $\dim \la \psi(S) \ra=7$. 
Then $\psi : S \to \Sigma_7$ is the closed immersion induced by $|-K_X|_S|$. 
We identify $S$ with its image $\psi(S)$. 
Note that each of $\Sigma_7 \subset \P^{15}$ and $S \subset \P^7 := \la S \ra$  is an intersection of quadrics \cite[Theorem 1.2]{FanoI}.  
Since $H^0(\cI_{\Sigma_7/\bP^{15}}(2)) \simeq H^0(\cI_{C / \bP^6}(2))  \simeq H^0(\cI_{S / \bP^7}(2))$, 
we have $S = \langle S \rangle \cap \Sigma_7$, i.e., $S$ is a linear section of $\Sigma_7$.
By repeating this argument replacing $(S,C)$ by $(X,S)$, 
$\psi: X \to \Sigma_7$ is a closed immersion such that 
$\psi(X) =  \langle \psi(X) \rangle \cap \Sigma_7$, as required.
\end{proof}

\bibliographystyle{skalpha}
\bibliography{reference.bib}

@article {Asc87,
    AUTHOR = {Aschbacher, Michael},
     TITLE = {Chevalley groups of type {$G_2$} as the group of a trilinear
              form},
   JOURNAL = {J. Algebra},
  FJOURNAL = {Journal of Algebra},
    VOLUME = {109},
      YEAR = {1987},
    NUMBER = {1},
     PAGES = {193--259},
      ISSN = {0021-8693},
   MRCLASS = {20G15 (11E99 20G40)},
  MRNUMBER = {898346},
MRREVIEWER = {Ulrich\ Dempwolff},
       DOI = {10.1016/0021-8693(87)90173-6},
       URL = {https://doi.org/10.1016/0021-8693(87)90173-6},
}

@article {BKM26a,
    AUTHOR = {Bayer, Arend and Kuznetsov, Alexander and Macr\`i, Emanuele},
     TITLE = {Mukai bundles on {F}ano threefolds},
   JOURNAL = {Compos. Math.},
  FJOURNAL = {Compositio Mathematica},
    VOLUME = {162},
      YEAR = {2026},
    NUMBER = {1},
     PAGES = {59--99},
      ISSN = {0010-437X,1570-5846},
   MRCLASS = {14J28 (14D20 14J30 14J45 14J60)},
  MRNUMBER = {5076141},
       DOI = {10.1017/S0010437X26102917},
       URL = {https://doi.org/10.1017/S0010437X26102917},
}

@article {BKM26b,
    AUTHOR = {Bayer, Arend and Kuznetsov, Alexander and Macr\`i, Emanuele},
     TITLE = {Mukai models of {F}ano varieties},
   JOURNAL = {J. Reine Angew. Math.},
  FJOURNAL = {Journal f\"ur die Reine und Angewandte Mathematik. [Crelle's
              Journal]},
    VOLUME = {836},
      YEAR = {2026},
     PAGES = {111--162},
      ISSN = {0075-4102,1435-5345},
   MRCLASS = {14J45},
  MRNUMBER = {5094113},
       DOI = {10.1515/crelle-2026-0024},
       URL = {https://doi.org/10.1515/crelle-2026-0024},
}

@article {BS07,
    AUTHOR = {Brodmann, Markus and Schenzel, Peter},
     TITLE = {Arithmetic properties of projective varieties of almost
              minimal degree},
   JOURNAL = {J. Algebraic Geom.},
  FJOURNAL = {Journal of Algebraic Geometry},
    VOLUME = {16},
      YEAR = {2007},
    NUMBER = {2},
     PAGES = {347--400},
      ISSN = {1056-3911,1534-7486},
   MRCLASS = {14N05 (14M05)},
  MRNUMBER = {2274517},
MRREVIEWER = {L\^e\ Tu\^an\ Hoa},
       DOI = {10.1090/S1056-3911-06-00442-5},
       URL = {https://doi.org/10.1090/S1056-3911-06-00442-5},
}

@article{BL13,
  author  = {Buczy{\'n}ski, Jaros{\l}aw and Landsberg, J. M.},
  title   = {Ranks of tensors and a generalization of secant varieties},
  journal = {Linear Algebra and its Applications},
  volume  = {438},
  number  = {2},
  pages   = {668--689},
  year    = {2013},
}

@article {CH88,
    AUTHOR = {Cohen, Arjeh M. and Helminck, Aloysius G.},
     TITLE = {Trilinear alternating forms on a vector space of dimension
              {$7$}},
   JOURNAL = {Comm. Algebra},
  FJOURNAL = {Communications in Algebra},
    VOLUME = {16},
      YEAR = {1988},
    NUMBER = {1},
     PAGES = {1--25},
      ISSN = {0092-7872,1532-4125},
   MRCLASS = {20G10 (15A69 15A75 20G15)},
  MRNUMBER = {921939},
MRREVIEWER = {James\ F.\ Hurley},
       DOI = {10.1080/00927878808823558},
       URL = {https://doi.org/10.1080/00927878808823558},
}

@book {EKM08,
    AUTHOR = {Elman, Richard and Karpenko, Nikita and Merkurjev, Alexander},
     TITLE = {The algebraic and geometric theory of quadratic forms},
    SERIES = {American Mathematical Society Colloquium Publications},
    VOLUME = {56},
 PUBLISHER = {American Mathematical Society, Providence, RI},
      YEAR = {2008},
     PAGES = {viii+435},
      ISBN = {978-0-8218-4329-1},
   MRCLASS = {11Exx (11-02 11E04 11E81 14C15 14C25)},
  MRNUMBER = {2427530},
MRREVIEWER = {Andrzej\ S\l adek},
       DOI = {10.1090/coll/056},
       URL = {https://doi-org.kyoto-u.idm.oclc.org/10.1090/coll/056},
}

@article {Fuj82b,
    AUTHOR = {Fujita, Takao},
     TITLE = {On polarized varieties of small {$\Delta $}-genera},
   JOURNAL = {Tohoku Math. J. (2)},
  FJOURNAL = {The Tohoku Mathematical Journal. Second Series},
    VOLUME = {34},
      YEAR = {1982},
    NUMBER = {3},
     PAGES = {319--341},
      ISSN = {0040-8735},
   MRCLASS = {14D20},
  MRNUMBER = {676113},
MRREVIEWER = {Autorreferat},
       DOI = {10.2748/tmj/1178229197},
       URL = {https://doi-org.utokyo.idm.oclc.org/10.2748/tmj/1178229197},
}

@book {Ful98,
    AUTHOR = {Fulton, William},
     TITLE = {Intersection theory},
    SERIES = {Ergebnisse der Mathematik und ihrer Grenzgebiete. 3. Folge. A
              Series of Modern Surveys in Mathematics [Results in
              Mathematics and Related Areas. 3rd Series. A Series of Modern
              Surveys in Mathematics]},
    VOLUME = {2},
   EDITION = {Second},
 PUBLISHER = {Springer-Verlag, Berlin},
      YEAR = {1998},
     PAGES = {xiv+470},
      ISBN = {3-540-62046-X; 0-387-98549-2},
   MRCLASS = {14C17 (14-02)},
  MRNUMBER = {1644323},
       DOI = {10.1007/978-1-4612-1700-8},
       URL = {https://doi-org.kyoto-u.idm.oclc.org/10.1007/978-1-4612-1700-8},
}

@book {Har77,
    AUTHOR = {Hartshorne, Robin},
     TITLE = {Algebraic geometry},
    SERIES = {Graduate Texts in Mathematics, No. 52},
 PUBLISHER = {Springer-Verlag, New York-Heidelberg},
      YEAR = {1977},
     PAGES = {xvi+496},
      ISBN = {0-387-90244-9},
   MRCLASS = {14-01},
  MRNUMBER = {0463157},
MRREVIEWER = {Robert Speiser},
}

@article {Isk77,
    AUTHOR = {Iskovskih, V. A.},
     TITLE = {Fano threefolds. {I}},
   JOURNAL = {Izv. Akad. Nauk SSSR Ser. Mat.},
  FJOURNAL = {Izvestiya Akademii Nauk SSSR. Seriya Matematicheskaya},
    VOLUME = {41},
      YEAR = {1977},
    NUMBER = {3},
     PAGES = {516--562, 717},
      ISSN = {0373-2436},
   MRCLASS = {14J10 (14M20 14N05)},
  MRNUMBER = {463151},
MRREVIEWER = {Miles Reid},
}

@article {Isk78,
    AUTHOR = {Iskovskih, V. A.},
     TITLE = {Fano threefolds. {II}},
   JOURNAL = {Izv. Akad. Nauk SSSR Ser. Mat.},
  FJOURNAL = {Izvestiya Akademii Nauk SSSR. Seriya Matematicheskaya},
    VOLUME = {42},
      YEAR = {1978},
    NUMBER = {3},
     PAGES = {506--549},
      ISSN = {0373-2436},
   MRCLASS = {14J10 (14J30 14M20 14N05)},
  MRNUMBER = {503430},
MRREVIEWER = {Miles Reid},
}

@incollection {IP99,
    AUTHOR = {Iskovskikh, V. A. and Prokhorov, Yu. G.},
     TITLE = {Fano varieties},
 BOOKTITLE = {Algebraic geometry, {V}},
    SERIES = {Encyclopaedia Math. Sci.},
    VOLUME = {47},
     PAGES = {1--247},
 PUBLISHER = {Springer, Berlin},
      YEAR = {1999},
   MRCLASS = {14J45 (14E07 14F22 14K30)},
  MRNUMBER = {1668579},
MRREVIEWER = {Takao Fujita},
}

@article{KT26,
  author =        {Kanemitsu, Akihiro and Tanaka, Hiromu},
  journal =       {preprint available at arXiv:2604.08817v1},
  title =         {Prime {F}ano threefolds of genus 8 in positive characteristic},
  year =          {2026},
}

@article{KTLift1,
  author =        {Kawakami, Tatsuro and Tanaka, Hiromu},
  journal =       {preprint available at arXiv:2503.10236v1},
  title =         {Liftability and vanishing theorems for {F}ano threefolds in positive characteristic {I}},
  year =          {2025},
}

@book {Laz04,
    AUTHOR = {Lazarsfeld, Robert},
     TITLE = {Positivity in algebraic geometry. {I}},
    SERIES = {Ergebnisse der Mathematik und ihrer Grenzgebiete. 3. Folge. A
              Series of Modern Surveys in Mathematics [Results in
              Mathematics and Related Areas. 3rd Series. A Series of Modern
              Surveys in Mathematics]},
    VOLUME = {48},
      NOTE = {Classical setting: line bundles and linear series},
 PUBLISHER = {Springer-Verlag, Berlin},
      YEAR = {2004},
     PAGES = {xviii+387},
      ISBN = {3-540-22533-1},
   MRCLASS = {14-02 (14C20)},
  MRNUMBER = {2095471},
MRREVIEWER = {Mihnea\ Popa},
       DOI = {10.1007/978-3-642-18808-4},
       URL = {https://doi.org/10.1007/978-3-642-18808-4},
}

@article {MR85,
    AUTHOR = {Mehta, V. B. and Ramanathan, A.},
     TITLE = {Frobenius splitting and cohomology vanishing for {S}chubert
              varieties},
   JOURNAL = {Ann. of Math. (2)},
  FJOURNAL = {Annals of Mathematics. Second Series},
    VOLUME = {122},
      YEAR = {1985},
    NUMBER = {1},
     PAGES = {27--40},
      ISSN = {0003-486X,1939-8980},
   MRCLASS = {14M15 (20G10)},
  MRNUMBER = {799251},
MRREVIEWER = {H.\ H.\ Andersen},
       DOI = {10.2307/1971368},
       URL = {https://doi-org.kyoto-u.idm.oclc.org/10.2307/1971368},
}

@article {MM81,
    AUTHOR = {Mori, Shigefumi and Mukai, Shigeru},
     TITLE = {Classification of {F}ano {$3$}-folds with {$B_{2}\geq 2$}},
   JOURNAL = {Manuscripta Math.},
  FJOURNAL = {Manuscripta Mathematica},
    VOLUME = {36},
      YEAR = {1981/82},
    NUMBER = {2},
     PAGES = {147--162},
      ISSN = {0025-2611},
   MRCLASS = {14J30 (14J10)},
  MRNUMBER = {641971},
MRREVIEWER = {Mary Schaps},
       DOI = {10.1007/BF01170131},
       URL = {https://doi.org/10.1007/BF01170131},
}

@incollection {MM83,
    AUTHOR = {Mori, Shigefumi and Mukai, Shigeru},
     TITLE = {On {F}ano {$3$}-folds with {$B_{2}\geq 2$}},
 BOOKTITLE = {Algebraic varieties and analytic varieties ({T}okyo, 1981)},
    SERIES = {Adv. Stud. Pure Math.},
    VOLUME = {1},
     PAGES = {101--129},
 PUBLISHER = {North-Holland, Amsterdam},
      YEAR = {1983},
   MRCLASS = {14J30},
  MRNUMBER = {715648},
MRREVIEWER = {I. Dolgachev},
       DOI = {10.2969/aspm/00110101},
       URL = {https://doi.org/10.2969/aspm/00110101},
}

@article {MM03,
    AUTHOR = {Mori, Shigefumi and Mukai, Shigeru},
     TITLE = {Erratum: ``{C}lassification of {F}ano 3-folds with {$B_2\geq
              2$}'' [{M}anuscripta {M}ath. {\bf 36} (1981/82), no. 2,
              147--162; {MR}0641971 (83f:14032)]},
   JOURNAL = {Manuscripta Math.},
  FJOURNAL = {Manuscripta Mathematica},
    VOLUME = {110},
      YEAR = {2003},
    NUMBER = {3},
     PAGES = {407},
      ISSN = {0025-2611},
   MRCLASS = {14J45 (14E30 14J30)},
  MRNUMBER = {1969009},
       DOI = {10.1007/s00229-002-0336-2},
       URL = {https://doi.org/10.1007/s00229-002-0336-2},
}

@article {Muk89,
    AUTHOR = {Mukai, Shigeru},
     TITLE = {Biregular classification of {F}ano {$3$}-folds and {F}ano
              manifolds of coindex {$3$}},
   JOURNAL = {Proc. Nat. Acad. Sci. U.S.A.},
  FJOURNAL = {Proceedings of the National Academy of Sciences of the United
              States of America},
    VOLUME = {86},
      YEAR = {1989},
    NUMBER = {9},
     PAGES = {3000--3002},
      ISSN = {0027-8424},
   MRCLASS = {14J30 (14J35 14J40)},
  MRNUMBER = {995400},
MRREVIEWER = {A. S. Tikhomirov},
       DOI = {10.1073/pnas.86.9.3000},
       URL = {https://doi-org.utokyo.idm.oclc.org/10.1073/pnas.86.9.3000},
}

@incollection {Muk93,
    AUTHOR = {Mukai, Shigeru},
     TITLE = {Curves and {G}rassmannians},
 BOOKTITLE = {Algebraic geometry and related topics ({I}nchon, 1992)},
    SERIES = {Conf. Proc. Lecture Notes Algebraic Geom., I},
     PAGES = {19--40},
 PUBLISHER = {Int. Press, Cambridge, MA},
      YEAR = {1993},
   MRCLASS = {14H45 (14M15)},
  MRNUMBER = {1285374},
MRREVIEWER = {Raquel Mallavibarrena},
}

@incollection {Muk95JAP,
    AUTHOR = {Mukai, Shigeru},
     TITLE = {New developments in the theory of {F}ano threefolds: vector
              bundle method and moduli problems [translation of {S}\={u}gaku
              {\bf 47} (1995), no. 2, 125--144; {MR}1364825 (96m:14059)]},
      NOTE = {Sugaku expositions},
   JOURNAL = {Sugaku Expositions},
  FJOURNAL = {Sugaku Expositions},
    VOLUME = {15},
      YEAR = {2002},
    NUMBER = {2},
     PAGES = {125--150},
      ISSN = {0898-9583},
   MRCLASS = {14J45 (14J10 14J60)},
  MRNUMBER = {1944132},
}

@article {Muk95,
    AUTHOR = {Mukai, Shigeru},
     TITLE = {Curves and symmetric spaces. {I}},
   JOURNAL = {Amer. J. Math.},
  FJOURNAL = {American Journal of Mathematics},
    VOLUME = {117},
      YEAR = {1995},
    NUMBER = {6},
     PAGES = {1627--1644},
      ISSN = {0002-9327},
   MRCLASS = {14H45 (14J45 14M17)},
  MRNUMBER = {1363081},
MRREVIEWER = {Raquel Mallavibarrena},
       DOI = {10.2307/2375032},
       URL = {https://doi-org.utokyo.idm.oclc.org/10.2307/2375032},
}

@article {Muk10,
    AUTHOR = {Mukai, Shigeru},
     TITLE = {Curves and symmetric spaces, {II}},
   JOURNAL = {Ann. of Math. (2)},
  FJOURNAL = {Annals of Mathematics. Second Series},
    VOLUME = {172},
      YEAR = {2010},
    NUMBER = {3},
     PAGES = {1539--1558},
      ISSN = {0003-486X,1939-8980},
   MRCLASS = {14H45 (14C20 14H51 14M15)},
  MRNUMBER = {2726093},
MRREVIEWER = {Raquel\ Mallavibarrena},
       DOI = {10.4007/annals.2010.172.1539},
       URL = {https://doi-org.kyoto-u.idm.oclc.org/10.4007/annals.2010.172.1539},
}

@incollection {Muk10b,
    AUTHOR = {Mukai, Shigeru},
     TITLE = {Polarized {$K3$} surfaces of genus {$18$} and {$20$}},
 BOOKTITLE = {Complex projective geometry ({T}rieste, 1989/{B}ergen, 1989)},
    SERIES = {London Math. Soc. Lecture Note Ser.},
    VOLUME = {179},
     PAGES = {264--276},
 PUBLISHER = {Cambridge Univ. Press, Cambridge},
      YEAR = {1992},
      ISBN = {0-521-43352-5},
   MRCLASS = {14J28 (14J10 32J15)},
  MRNUMBER = {1201388},
MRREVIEWER = {Shigeyuki\ Kondo},
       DOI = {10.1017/CBO9780511662652.019},
       URL = {https://doi.org/10.1017/CBO9780511662652.019},
}

@article{PR96,
  author  = {Pragacz, Piotr and Ratajski, Jan},
  title   = {Formulas for Lagrangian and orthogonal degeneracy loci; $\widetilde Q$-polynomial approach},
  journal = {Compositio Mathematica},
  volume  = {107},
  number  = {1},
  pages   = {11--87},
  year    = {1997}
}

@book{SGA5,
  editor    = {Illusie, Luc},
  title     = {Cohomologie {$\ell$}-adique et fonctions {$L$}},
  subtitle  = {S{\'e}minaire de {G}{\'e}om{\'e}trie {A}lg{\'e}brique du {B}ois-{M}arie 1965--66, {SGA} 5},
  series    = {Lecture Notes in Mathematics},
  volume    = {589},
  publisher = {Springer-Verlag},
  address   = {Berlin},
  year      = {1977},
  doi       = {10.1007/BFb0096802}
}

@article {Sho79a,
    AUTHOR = {\v{S}okurov, V. V.},
     TITLE = {The existence of a line on {F}ano varieties},
   JOURNAL = {Izv. Akad. Nauk SSSR Ser. Mat.},
  FJOURNAL = {Izvestiya Akademii Nauk SSSR. Seriya Matematicheskaya},
    VOLUME = {43},
      YEAR = {1979},
    NUMBER = {4},
     PAGES = {922--964, 968},
      ISSN = {0373-2436},
   MRCLASS = {14J30 (14M20)},
  MRNUMBER = {548510},
MRREVIEWER = {Miles Reid},
}

@article {Sho79b,
    AUTHOR = {\v{S}okurov, V. V.},
     TITLE = {Smoothness of a general anticanonical divisor on a {F}ano
              variety},
   JOURNAL = {Izv. Akad. Nauk SSSR Ser. Mat.},
  FJOURNAL = {Izvestiya Akademii Nauk SSSR. Seriya Matematicheskaya},
    VOLUME = {43},
      YEAR = {1979},
    NUMBER = {2},
     PAGES = {430--441},
      ISSN = {0373-2436},
   MRCLASS = {14J30},
  MRNUMBER = {534602},
MRREVIEWER = {Werner Kleinert},
}

@misc{SP,
  author =        {{The} {Stacks Project Authors}},
  howpublished =  {\url{http://stacks.math.columbia.edu}},
  title =         {\itshape {S}tacks {P}roject},
}

@article {Tak89,
    AUTHOR = {Takeuchi, Kiyohiko},
     TITLE = {Some birational maps of {F}ano {$3$}-folds},
   JOURNAL = {Compositio Math.},
  FJOURNAL = {Compositio Mathematica},
    VOLUME = {71},
      YEAR = {1989},
    NUMBER = {3},
     PAGES = {265--283},
      ISSN = {0010-437X},
   MRCLASS = {14J30 (14E05)},
  MRNUMBER = {1022045},
MRREVIEWER = {Peter Nielsen},
       URL = {http://www.numdam.org/item?id=CM_1989__71_3_265_0},
}

@article {Tan25,
    AUTHOR = {Tanaka, Hiromu},
     TITLE = {Elliptic singularities and threefold flops in positive
              characteristic},
   JOURNAL = {Proc. Edinb. Math. Soc. (2)},
  FJOURNAL = {Proceedings of the Edinburgh Mathematical Society. Series II},
    VOLUME = {68},
      YEAR = {2025},
    NUMBER = {4},
     PAGES = {1188--1244},
      ISSN = {0013-0915,1464-3839},
   MRCLASS = {14J17 (14E30)},
  MRNUMBER = {4973442},
       DOI = {10.1017/S0013091525000185},
       URL = {https://doi-org.kyoto-u.idm.oclc.org/10.1017/S0013091525000185},
}

@article{Tan-bdl,
  author =        {Tanaka, Hiromu},
  journal =       {in preparation},
  title =         {Mukai bundles and {B}rill-{N}oether generality for prime {F}ano threefolds in positive characteristic},
  year =          {2026},
}

@article{FanoI,
  author =        {Tanaka, Hiromu},
  journal =       {arXiv:2308.08121},
  title =         {Fano threefolds in positive characteristic {I}},
  year =          {2023},
}

@article{FanoII,
  author =        {Tanaka, Hiromu},
  journal =       {arXiv:2308.08122},
  title =         {Fano threefolds in positive characteristic {II}},
  year =          {2023},
}

@article{FanoIII,
  author =        {Asai, Masaya and Tanaka, Hiromu},
  journal =       {arXiv:2308.08124},
  title =         {Fano threefolds in positive characteristic {III}},
  year =          {2023},
}

@article{FanoIV,
  author =        {Tanaka, Hiromu},
  journal =       {arXiv:2308.08127},
  title =         {Fano threefolds in positive characteristic {IV}},
  year =          {2023},
}

\end{document}